\documentclass[mathpazo]{cicp}
\usepackage[margin=1in]{geometry}
\usepackage{amsmath,epsf,cite}
\usepackage{amssymb,amsthm,tocvsec2}
\usepackage{graphicx}
\usepackage{epsfig,epsf,latexsym,subfigure}
\usepackage{float}
\usepackage{color}
\usepackage{mathrsfs}
\usepackage{indentfirst}

\date{ }

\numberwithin{equation}{section}
\numberwithin{figure}{section}
\numberwithin{table}{section}

\theoremstyle{plain}
\newtheorem{thm}{Theorem}[section]
\newtheorem{lem}{Lemma}[section]

\theoremstyle{remark}
\newtheorem{rem}{Remark}[section]

\newcommand{\bu}{{\mathbf u}}

\def\bu{\mathbf{u}}
\def\bv{\mathbf{v}}

\def\bn{\mathbf{n}}

\def\bx{\mathbf{x}}
\def\ba{\mathbf{a}}

\def\bI{\mathbf{I}}

\def\half{\frac{1}{2}}

\def\bu{\mathbf{u}}
\def\bv{\mathbf{v}}
\def\bw{\mathbf{w}}
\def\bV{\mathbf{V}}
\def\bH{\mathbf{H}}

\def\cG{\mathcal{G}}

\def\cE{\mathcal{E}}

\newcommand{\ben}{\begin{eqnarray}}
\newcommand{\een}{\end{eqnarray}}
\newcommand{\beq}{\begin{equation}}
\newcommand{\eeq}{\end{equation}}
\newcommand{\bea}{\begin{array}}
\newcommand{\eea}{\end{array}}
\newcommand{\bef}{\begin{figure}[H]}
\newcommand{\eef}{\end{figure}}

\newcommand{\bse}{\begin{subequations}}
\newcommand{\ese}{\end{subequations}}

\newtheorem{scheme}{Scheme}[section]

\begin{document}
\title{Second-order Decoupled Energy-stable Schemes for Cahn-Hilliard-Navier-Stokes equations}
\author[J. Zhao]{
Jia Zhao\affil{1}\comma \corrauth}
\address{\affilnum{1}\ Department of Mathematics \& Statistics, Utah State University, Logan, UT, USA }
\email{ {\tt jia.zhao@usu.edu.} (J.~Zhao)}

\begin{abstract}
The Cahn-Hilliard-Navier-Stokes (CHNS) equations represent the fundamental building blocks of hydrodynamic phase-field models for multiphase fluid flow dynamics. Due to the coupling between the Navier-Stokes equation and the Cahn-Hilliard equation, the CHNS system is non-trivial to solve numerically. Traditionally, a numerical extrapolation for the coupling terms is used. However, such brute-force extrapolation usually destroys the intrinsic thermodynamic structures of this CHNS system. This paper proposes a new strategy to reformulate the CHNS system into a constraint gradient flow formation. Under the new formulation, the reversible and irreversible structures are clearly revealed. This guide us to propose operator splitting schemes. The operator splitting schemes have several advantageous properties.
First of all, the proposed schemes lead to several decoupled systems in smaller sizes to be solved at each time marching step. This significantly reduces the computational costs. Secondly, the proposed schemes still guarantee the thermodynamic laws of the CHNS system at the discrete level. These are known as structure-preserving schemes. This structure-preserving property is desired. It ensures the thermodynamic laws, the accuracy and stability for the numerical solutions. In addition, unlike the recently populated IEQ or SAV approach using auxiliary variables, our resulting energy laws are formulated in the original variables. This is a significant improvement, as the modified energy laws with auxiliary variables sometimes deviate from the original energy law.  Our proposed framework lays out a foundation to design decoupled and energy stable numerical algorithms for hydrodynamic phase-field models.
Furthermore, given different splitting steps, various numerical algorithms can be obtained, making this framework rather general. The proposed numerical algorithms are implemented. Their second-order accuracy in time is verified numerically. Some numerical examples and benchmark problems are calculated to verify the effectiveness of the proposed schemes.
\end{abstract}


\maketitle

\section{Background}

Multiphase interfacial problems are ubiquitously in nature and industrial processes. As one of the most widely used approaches, phase field methods/models \cite{Cahn&H1958} have been broadly utilized in various fields to investigate interfacial problems. The phase-field method's major advantage is that the evolving interface is captured intrinsically, instead of explicitly like other interface tracking methods. This dramatically simplifies the modeling and computational processes. Mainly, one needs to introduce phase-field variables that could either be the volume fractions or labels for the phases. Then continuous partial differential equations (PDEs) are proposed in the computational domain. Once the PDEs are solved, the interfaces will be retrieved through level sets of the phase variables. The phase-field models are also known as diffuse interface models, since an artificial diffuse interface is usually introduced to regularize the phase variables which are the  PDEs' solutions. Due to its simplicity in the theoretical formulation and numerical implementation, the phase-field method has been widely used in the fields where multiple material phases are involved.
When the material systems are fluids, the velocity fields shall be considered since the interactions between kinetics energy and free energy are not ignorable. Thus, the hydrodynamic equations for fluid flows will be proposed along with the phase-field equations for the phase variables. The well-known Cahn-Hilliard-Navier-Stokes equations are the fundamental system  for the hydrodynamic phase-field models. 
It has been widely used in interfacial problems for incompressible two fluid mixture. 

Consider the domain $\Omega$ and time $t \in (0, T]$, and denote $\Omega_T = \Omega \times (0, T]$. Here we use $\phi \in [-1, 1]$ as the phase variable, with $\phi=1$ to label one phase, $\phi=-1$ to label the other phase, and $\phi \in (-1, 1)$ representing the interface.
The total energy of the two phase fluid-mixture system $\mathcal{E}$ include the Helmholtz free energy $F$ and the kinetic energy $E$, i.e.
\beq \label{eq:energy}
\cE(\bu, \phi)  = F(\phi) + E(\bu), \quad F(\phi)= \int_\Omega  \gamma \Big( \frac{\varepsilon}{2} | \nabla \phi |^2 + \frac{1}{\varepsilon} f(\phi)  \Big) d\bx, \quad E(\bu) = \int_\Omega \frac{\rho}{2}|\bu|^2d \bx.
\eeq
where $\gamma$ is the surface tension between the two fluid phases, 
$\varepsilon$ is an artificial parameter controlling the interfacial thickness. $f(\phi)$ is the bulk free energy for the two phase material. $\bu$ the volume-averaged velocity, and $\rho$ is the volume-averaged density. In this paper, we assume the fluid mixture is incompressible, and both phases have the same density.  The double-well potential
\beq \label{eq:bulk_potential}
f(\phi) = \frac{1}{4} (\phi^2-1)^2
\eeq 
will be used as the bulk potential in the rest of this paper. Other cases, such as the Flory-Huggins bulk potential, could be treated similarly with our proposed algorithms in this paper.
Notice that our proposed methodology can be easily applied to compressible or quasi-compressible CHNS models. These topics will be investigated in our later research. but will not be pursued in this paper.

Then, the Cahn-Hilliard-Navier-Stokes (CHNS) equations are proposed as
\beq \label{eq:CHNS}
\left\{
\bea{l}
\rho \Big( \partial_t \bu + \bu \cdot \nabla \bu \Big) = - \nabla p + \eta \nabla^2 \bu -   \phi \nabla \mu,  \quad   (\bx, t) \in \Omega_T, \\
\nabla \cdot \bu  =  0,  \quad  (\bx, t) \in \Omega_T,  \\
\partial_t \phi +\nabla \cdot (\bu \phi) =  \nabla \cdot( M(\phi)  \nabla  \mu), \quad  (\bx, t) \in \Omega_T, \\
\mu = - \gamma \varepsilon \Delta \phi + \frac{\gamma}{\varepsilon} f'(\phi), \quad (\bx, t) \in \Omega_T, \\
\eea
\right.
\eeq
where $\eta$ is the viscosity parameter, $M(\phi) \geq 0$ is the mobility operator, $p$ is the hydrodynamic pressure, and $\mu = \frac{\delta F}{\delta \phi}$ is the chemical potential. 
The boundary conditions are not unique. In this paper, if not otherwise specified, we focus on the physical boundary conditions:
\beq  \label{eq:CHNS-boundary}
\bu(\bx, t) = 0, \quad \nabla \mu(\bx, t) \cdot \bn = 0 , \quad \nabla \phi(\bx, t) \cdot \bn =0, \quad (\bx, t) \in \partial \Omega \times (0, T],
\eeq 
with $\bn$ the outward normal vector at the boundary.
The CHNS system  in \eqref{eq:CHNS}-\eqref{eq:CHNS-boundary} is known to satisfy the second law of thermodynamics, with the energy dissipation rate calculated as
\beq \label{eq:energy-law-continous}
\frac{d\cE}{dt} = - \int_\Omega \Big( |\sqrt{M(\phi)}\nabla \mu|^2 + \eta |\nabla \bu|^2  \Big) d\bx.
\eeq 

One principle in developing numerical algorithms for solving the CHNS system is to guarantee that the numerical solutions also satisfy the energy law in \eqref{eq:energy-law-continous}. A numerical scheme also guarantees the monotone property of the energy is known as an energy stable scheme. The energy-stable scheme usually can allow large time marching steps while preserving accuracy and stability.
Given the significant role it played in the phase-field models, the CHNS system has drawn a considerable amount of attention. Many seminal works have been published investigating different aspects. Here we briefly review some relevant results in numerical analysis. There are several pieces of existing works on fully discrete schemes for the hydrodynamic phase-field model or its simplified versions  \cite{WiseJSC2010, Chen&S2016, Guo&Lin&Lowengrub&Wise2017, Han&WangJCP2015}, all of them are either only first-order in time or nonlinear. In comparison, our decoupled schemes can be more efficient in implementation and computational costs.
Han and Wang introduce a second-order velocity projection method for the CHNS system in \cite{Han&WangJCP2015}. The resulted numerical scheme is second-order accurate in time. This is one of the earliest works to introduce second-order numerical algorithms for the CHNS system while preserving energy stability.    Later Gong et al. manage to improve the scheme by using the energy quadratization (EQ) idea, such that only linear systems need to be solved in each time step \cite{Gong&Zhao&WangSISC2}.  Another linear and energy stable scheme is introduced in \cite{Chen&ZhaoJCP2020}, where a stabilized leap-frog type time-marching strategy is adapted. However, the velocity field and phase variables have to be solved simultaneously in all the schemes just mentioned. This requires either a Newton iteration method or a fixed-point iteration method. This drawback shall not be ignored since the solution's existence and uniqueness for the iteration method usually have substantial requirements on the time step, making the unconditional stability of the original numerical algorithms less appealing.  We also note that some second-order (linear) energy stable schemes have been developed for thermodynamic phase-field equations in recently years \cite{Gomez&HughesJCP2011, Lee&ShinCMAME2017, SAV-2, WangChengCMS2015, Feng&Wang&Wise&Zhang2017}, which may potentially be applicable to hydrodynamic phase-field models.  

Due to the coupling between the hydrodynamic equation and the phase-field equation, it is desirable to develop accurate and efficient numerical algorithms that can decouple the two equations. It is preferred that, in each time marching step, only a smaller size of problems shall be solved sequentially. This is the primary motivation of this paper. Seeking decoupled numerical algorithms that can solve the velocity field $\bu$ and phase variable $\phi$ independently has been an active research topic. Minjeaud realized a first-order relaxation on the velocity field could uncouple the velocity field and phase variable such that decoupled numerical algorithms for triphasic Cahn-Hilliard-Navier-Stoke model can be developed \cite{MinjeaudS13}. This idea has been further populated to investigate various hydrodynamic phase field models \cite{Shen&Yang2014SISC, Zhao&Yang&Shen&WangJCP2016, Zhao&Wang&YangJSC2017}, for which one can get a linearly decoupled scheme such that the velocity field, phase variable, and pressure can be solved sequentially. Each of the sub-problems is an elliptic-type equation so that fast and efficient solvers can be applied. But due to the unavoidable  first-order modification on the velocity field, the accuracy is restricted to first-order, which is not much attractive for practical application. In parallel, the idea of the scalar auxiliary variable (SAV) has been used to decouple the phase field and hydrodynamic systems \cite{Li&ShenDecoupled2020, YangDecoupled2021,YangDecoupled-1,YangDecoupled-2}, from which second-order decoupled schemes could be developed. However, such strategies usually introduce new variables so that discrete energy laws are modified using the auxiliary variables, making the connections with the energy law with original variables less clear.

In this paper, we come up with a novel approach to overcome all these difficulties mentioned above. Mainly, we propose an elegant strategy to develop decoupled and energy stable numerical algorithms for the CHNS system in a confined geometry subject to physical boundary conditions.  Instead of designing algorithms by a trial-and-error approach, we first reformulate the CHNS system into a constraint gradient flow system. This reformulation provides insights into the numerical algorithm design. With the constraint gradient flow formation, we propose several variants of second-order operator splitting schemes. All the schemes uncouple the velocity field and phase variables so that only smaller systems need to be solved in each time marching step. Besides the efficiency, all the schemes are rigorously shown to be energy stable, i.e., satisfy the discrete energy law.  We emphasize that our proposed schemes hold the discrete energy laws in the original variables, which differs from the EQ or SAV approaches. 

The rest of the paper is organized as follows. We first reformulate the CHNS system into a constraint gradient flow formulation in Section \ref{sect:model-reformulation}. Then we discretize the space with a second-order finite difference on staggered grids in Section \ref{sect:space}. Afterward, we introduce the second-order splitting techniques to design second-order accurate in time and decoupled numerical algorithms for solving the CHNS system in Section \ref{sect:time}. The energy stable property of the proposed schemes is also rigorously proved. Then, in Section \ref{sect:results}, we present the time-step convergence tests and several benchmark problems. The numerical results highlight the effectiveness of our proposed decoupled numerical schemes. In the end, we give a brief conclusion.

\section{Model Reformulation of Cahn-Hilliard-Navier-Stokes equations} \label{sect:model-reformulation}
\subsection{Model reformulation}
First of all, we illustrate the reformation of incompressible CHNS system. We emphasize that the pressure $p$ in \eqref{eq:CHNS} is a Lagrangian multiplier to enforce the in-compressibility of the velocity field $\bu$. With this in mind, we can reformulate the CHNS system into a constraint gradient flow form. This will guide us in designing decoupled numerical algorithms.  We follow the notations in \cite{OttingerPRE1997A,OttingerPRE1997B}.
Let $P_\bu$ be a functional space defined by
\beq
P_\bu = \left\{ \bu(x, t): \bu \in \bV; \quad \nabla \cdot \bu =0  \mbox{ in } \Omega, \quad \bu = 0 \mbox{ on } \partial \Omega   \right\},
\eeq 
with $\bV$ being the space of three-dimensional vector fields. $\Pi_\bu$ denotes a projection operator defined as
\beq \label{eq:projection-operator}
\Pi_\bu(\ba) =
\left\{
\bea{l}
\ba - \nabla p, \quad \mbox{in } \Omega / \partial \Omega, \\
0 \quad \mbox{ on } \partial \Omega,
\eea 
\right.
\eeq
where $p$ satisfies a Poisson condition with a Neumann-type boundary condition, i.e. 
$$
\left\{
\bea{l}
\Delta p = \nabla \cdot \ba, \quad  \mbox{ in } \Omega, \\
\frac{\partial p}{\partial \bn} = \ba \cdot \bn, \quad  \mbox{ on } \partial \Omega.
\eea 
\right.
$$
With the projection operator in \eqref{eq:projection-operator}, we denote the constraint variational derivative of the kinetic energy in \eqref{eq:energy} with respect to the velocity field as
\beq \label{eq:projection-u}
\frac{\delta E}{\delta \bu} = \Pi_\bu \frac{\partial E}{\partial \bu}.
\eeq

Next, we illustrate the reformulation of the convection term in the Navier-Stokes equation. We rewrite the nonlinear convection term as
\beq \label{eq:reformulation-convection}
B(\bv, \bu) = \frac{1}{2} \Big[  \bv \cdot \nabla \bu + \nabla \cdot (\bv \bu) \Big].
\eeq 
In addition, the skew-symmetric form $B(\bv,\bu)$ induce a trilinear form $b$ defined as \cite{Han&WangJCP2015}
\beq
b(\bv,\bu,\bw) = \Big( B(\bv,\bu), \bw \Big) = \frac{1}{2} \left[ \Big( \bv \cdot \nabla \bu, \bw \Big) - \Big( \bv \cdot \nabla \bw, \bu \Big)  \right],  \forall \bu, \bv, \bw \in \bH_0^1(\Omega).
\eeq 
It follows immediately that 
\beq
b(\bv, \bu, \bu) = 0, \quad \forall \bu, \bv \in \bH_0^1(\Omega).
\eeq 

\begin{rem}
The reformulation of the convection term is not unique. We can also write the convection term in a skew-symmetric form as \cite{Han&WangJCP2015}
\beq 
B(\bv,\bu) = (\bv \cdot \nabla )\bu + \frac{1}{2} (\nabla \cdot \bv) \bu.
\eeq 
\end{rem}

We can easily see that 
$B(\bu, \bu) = \bu \cdot \nabla \bu$, given that $\nabla \cdot \bu =0$. Hence, the incompressible CHNS system in \eqref{eq:CHNS}  can be rewritten as
\beq
\rho (\partial_t \bu + B(\bu, \bu))  =  -\nabla p + \eta \Delta \bu - \phi \nabla \mu.
\eeq 

With the constraint variation in \eqref{eq:projection-u} and reformulation of the convection term in \eqref{eq:reformulation-convection}, we are ready to rewrite the CHNS equation in \eqref{eq:CHNS} as a constraint gradient flow form
\beq  \label{eq:CHNS-gradient-flow}
\Lambda \partial_t \Psi = \cG\frac{\delta E}{\delta \Psi}, \quad \Lambda = 
\begin{bmatrix}
\rho  & 0 \\
0 &  1 \\
\end{bmatrix}, \quad 
\Psi = \begin{bmatrix}
\bu \\ \phi 
\end{bmatrix}, \quad \cG= \cG_a + \cG_s,
\eeq 
with proper boundary conditions and initial values.
Here $\frac{\delta E}{\delta \bu}: =\Pi_\bu \frac{\partial E}{\partial \bu}$ is defined using the projection operator as shown in \eqref{eq:projection-u}. The mobility operator $\cG=\cG_a + \cG_s$ is defined as
\beq \label{eq:Splitting-Case1}
\cG_a = 
\begin{pmatrix}
0 & -\phi \nabla \bullet  \\
-\nabla \cdot (\bullet \phi) & 0 \\
\end{pmatrix}, 
\quad 
\cG_s =
\begin{pmatrix}
\eta \Delta \bullet - B(\bu, \bullet ) & 0 \\
0 &  \nabla \cdot ( M(\phi) \nabla \bullet)
\end{pmatrix}.
\eeq 
Here $\cG_a$ controls the reversible dynamics, representing the energy exchanges between the kinetic energy and the Helmholtz free energy, and $\cG_s$ controls the irreversible dynamics, representing the energy dissipation.

\begin{rem}
This constraint gradient flow reformulation is not limited to the Cahn-Hilliard-Navier-Stokes system. Other thermodynamic consistent hydrodynamic models could be reformulated in a similar manner. We will not elaborate on it due to space limitations.
\end{rem}

\section{Spatial Discretization on Staggered Grids} \label{sect:space}
In this section, we present the spatial discretization of the CHNS system in \eqref{eq:CHNS} with physical boundary conditions in \eqref{eq:CHNS-boundary}.

\subsection{Notations  for spatial discretization}\label{sect:Notations}

To simplify the presentation, we introduce some finite difference notations for spatial discretization. Although these notations can also be found in \cite{WiseJSC2010,Shen&Wang&Wise2012,Wise&Kim&LowengrubJCP2007,Chen&S2016,Gong&Zhao&WangSISC2},  we summarize them in this section to make this paper reader-friendly.  

We consider a rectangular spatial domain  $\Omega = [0,L_x]\times[0,L_y]$  with $L_x$ and $L_y$ two positive numbers. The domain is discretized into uniform rectangular meshes with mesh size $h_x = \frac{L_x}{N_x}$ and $h_y = \frac{L_y}{N_y}$. Here $N_x$ and $N_y$ are two positive integers. We define the following 1D sets for grid points
\ben
&& E_x = \{x_{i+\frac{1}{2}}|i=0,1,\ldots,N_x\},\quad C_x = \{x_i|i=1,2,\ldots,N_x\},\quad C_{\overline{x}} = \{x_i|i=0,1,\ldots,N_x+1\},\nonumber\\
&& E_y = \{y_{j+\frac{1}{2}}|j=0,1,\ldots,N_y\},\quad C_y = \{y_j|j=1,2,\ldots,N_y\},\quad C_{\overline{y}} = \{y_j|j=0,1,\ldots,N_y+1\},\nonumber
\een
where $x_l = (l-\frac{1}{2})h_x,~y_l = (l-\frac{1}{2})h_y,$ and $l$ can take on either integer or half-integer values. $E_x$ is called a uniform partition of $[0,L_x]$ of size $N_x$, and its elements are called edge-centered points. The elements of $C_x$ and $C_{\overline{x}}$ are called cell-centered points. The two points belonging to $C_{\overline{x}}\setminus C_x$ are called ghost points. Analogously, the set $E_y$ is a uniform partition of $[0,L_y]$ of size $N_y$, called edge-centered points, and $C_y$ and $C_{\overline{y}}$ contain the cell-centered points of the interval $[0,L_y]$.

We define the following discrete function spaces
\begin{eqnarray*}
&& \mathcal{C}_{x\times y} = \{\phi:C_x\times C_y\rightarrow \mathbb{R}\},\quad \mathcal{C}_{\overline{x}\times y} = \{\phi:C_{\overline{x}}\times C_y\rightarrow \mathbb{R}\},\quad \mathcal{C}_{x\times \overline{y}} = \{\phi:C_x\times C_{\overline{y}}\rightarrow \mathbb{R}\},\\
&& \mathcal{C}_{\overline{x}\times \overline{y}} = \{\phi:C_{\overline{x}}\times C_{\overline{y}}\rightarrow \mathbb{R}\},\quad \mathcal{E}^{ew}_{x\times y} = \{u:E_x\times C_y\rightarrow \mathbb{R}\},\quad \mathcal{E}^{ew}_{x\times \overline{y}} = \{u:E_x\times C_{\overline{y}}\rightarrow \mathbb{R}\},\\
&& \mathcal{E}^{ns}_{x\times y} = \{v:C_x\times E_y\rightarrow \mathbb{R}\},\quad \mathcal{E}^{ns}_{\overline{x}\times y} = \{v:C_{\overline{x}}\times E_y\rightarrow \mathbb{R}\},\quad \mathcal{V}_{x\times y} = \{f:E_x\times E_y\rightarrow \mathbb{R}\}.
\end{eqnarray*}

Throughout this paper, we denote the cell-centered, edge-centered and vertex-centered discrete functions as follows:
\begin{eqnarray*}
\textrm{cell~centered~functions:}&&\phi,\psi,\mu \in \mathcal{C}_{x\times y} \cup \mathcal{C}_{\overline{x}\times y} \cup \mathcal{C}_{x\times \overline{y}} \cup \mathcal{C}_{\overline{x}\times \overline{y}},\\
\textrm{east~west~edge~centered~functions:}&&u,r \in \mathcal{E}^{ew}_{x\times y} \cup \mathcal{E}^{ew}_{x\times \overline{y}},\\
\textrm{north~south~edge~centered~functions:}&&v,w \in \mathcal{E}^{ns}_{x\times y} \cup \mathcal{E}^{ns}_{\overline{x}\times y},\\
\textrm{vertex~centered~functions:}&&f,g\in \mathcal{V}_{x\times y}.
\end{eqnarray*}

We define the discrete function spaces with homogeneous Dirichlet boundary conditions as follows:
\begin{eqnarray*}
&& \mathcal{E}^{ew0}_{x\times y} = \{u\in \mathcal{E}^{ew}_{x\times y} \cup \mathcal{E}^{ew}_{x\times \overline{y}}\big|u_{\frac{1}{2},j}=u_{N_x+\frac{1}{2},j}=0,~j=1,2,\ldots,N_y\},\\
&& \mathcal{E}^{ew0}_{x\times \overline{y}} = \{u\in \mathcal{E}^{ew}_{x\times \overline{y}}\big|u_{\frac{1}{2},j}=u_{N_x+\frac{1}{2},j}=0,~j=0,1,\ldots,N_y+1\},\\
&& \mathcal{E}^{ns0}_{x\times y} = \{v\in \mathcal{E}^{ns}_{x\times y} \cup \mathcal{E}^{ns}_{\overline{x}\times y}\big|v_{i,\frac{1}{2}}=v_{i,N_y+\frac{1}{2}}=0,~i=1,2,\ldots,N_x\},\\
&& \mathcal{E}^{ns0}_{\overline{x}\times y} = \{v\in \mathcal{E}^{ns}_{\overline{x}\times y}\big|v_{i,\frac{1}{2}}=v_{i,N_y+\frac{1}{2}}=0,~i=0,1,\ldots,N_x+1\},\\
&& \mathcal{V}^{ew0}_{x\times y} = \{f\in\mathcal{V}_{x\times y}\big|f_{\frac{1}{2},j+\frac{1}{2}}=f_{N_x+\frac{1}{2},j+\frac{1}{2}}=0,~j=0,1,\ldots,N_y\},\\
&& \mathcal{V}^{ns0}_{x\times y} = \{f\in\mathcal{V}_{x\times y}\big|
f_{i+\frac{1}{2},\frac{1}{2}}=f_{i+\frac{1}{2},N_y+\frac{1}{2}}=0,~i=0,1,\ldots,N_x\},\\
&& \mathcal{V}^0_{x\times y} = \mathcal{V}^{ew0}_{x\times y}\cap \mathcal{V}^{ns0}_{x\times y}.
\end{eqnarray*}
Note that $\mathcal{E}^{ew0}_{x\times \overline{y}}\subset\mathcal{E}^{ew0}_{x\times y}$ and $\mathcal{E}^{ns0}_{\overline{x}\times y}\subset\mathcal{E}^{ns0}_{x\times y}.$
The east-west-edge-to-center average and difference operators are defined as
$a_x,~d_x:\mathcal{E}^{ew}_{x\times \overline{y}}\cup \mathcal{V}_{x\times y}\rightarrow \mathcal{C}_{x\times \overline{y}} \cup \mathcal{E}^{ns}_{x\times y}$  in component-wise forms:
\begin{eqnarray*}
&& a_x u_{i,j} = \frac{1}{2}(u_{i+\frac{1}{2},j}+u_{i-\frac{1}{2},j}),\quad  d_x u_{i,j} = \frac{1}{h_x}(u_{i+\frac{1}{2},j}-u_{i-\frac{1}{2},j}),\quad a_x u, d_x u\in \mathcal{C}_{x\times \overline{y}},\\
&& a_x f_{i,j+\frac{1}{2}} = \frac{1}{2}(f_{i+\frac{1}{2},j+\frac{1}{2}}+f_{i-\frac{1}{2},j+\frac{1}{2}}),\quad  d_x f_{i,j+\frac{1}{2}} = \frac{1}{h_x}(f_{i+\frac{1}{2},j+\frac{1}{2}}-f_{i-\frac{1}{2},j+\frac{1}{2}}),\quad a_x f, d_x f\in \mathcal{E}^{ns}_{x\times y}.
\end{eqnarray*}
The north-south-edge-to-center average and difference operators are defined as $a_y,~d_y:\mathcal{E}^{ns}_{\overline{x}\times y}\cup \mathcal{V}_{x\times y}\rightarrow \mathcal{C}_{\overline{x}\times y} \cup \mathcal{E}^{ew}_{x\times y}$  in component-wise forms:
\begin{eqnarray*}
&& a_y v_{i,j} = \frac{1}{2}(v_{i,j+\frac{1}{2}}+v_{i,j-\frac{1}{2}}),\quad  d_y v_{i,j} = \frac{1}{h_y}(v_{i,j+\frac{1}{2}}-v_{i,j-\frac{1}{2}}),\quad a_y v, d_y v\in \mathcal{C}_{\overline{x}\times y},\\
&& a_y f_{i+\frac{1}{2},j} = \frac{1}{2}(f_{i+\frac{1}{2},j+\frac{1}{2}}+f_{i+\frac{1}{2},j-\frac{1}{2}}),\quad  d_y f_{i+\frac{1}{2},j} = \frac{1}{h_y}(f_{i+\frac{1}{2},j+\frac{1}{2}}-f_{i+\frac{1}{2},j-\frac{1}{2}}),\quad a_y f, d_y f\in \mathcal{E}^{ew}_{x\times y}.
\end{eqnarray*}
The center-to-east-west-edge average and difference operators are defined as $A_x,~D_x:\mathcal{C}_{\overline{x}\times \overline{y}}\cup \mathcal{E}^{ns}_{\overline{x}\times y}\rightarrow \mathcal{E}^{ew}_{x\times \overline{y}} \cup \mathcal{V}_{x\times y}$    in component-wise forms:
\begin{eqnarray*}
&& A_x \phi_{i+\frac{1}{2},j} = \frac{1}{2}(\phi_{i+1,j}+\phi_{i,j}),\quad  D_x \phi_{i+\frac{1}{2},j} = \frac{1}{h_x}(\phi_{i+1,j}-\phi_{i,j}),\quad A_x \phi, D_x \phi\in \mathcal{E}^{ew}_{x\times \overline{y}},\\
&& A_x v_{i+\frac{1}{2},j+\frac{1}{2}} = \frac{1}{2}(v_{i+1,j+\frac{1}{2}}+v_{i,j+\frac{1}{2}}),\quad  D_x v_{i+\frac{1}{2},j+\frac{1}{2}} = \frac{1}{h_x}(v_{i+1,j+\frac{1}{2}}-v_{i,j+\frac{1}{2}}),\quad A_x v, D_x v\in \mathcal{V}_{x\times y}.
\end{eqnarray*}
The center-to-north-south-edge average and difference operators are defined as $A_y,~D_y:\mathcal{C}_{\overline{x}\times \overline{y}}\cup \mathcal{E}^{ew}_{x\times \overline{y}}\rightarrow \mathcal{E}^{ns}_{\overline{x}\times y} \cup \mathcal{V}_{x\times y}$ in component-wise forms:
\begin{eqnarray*}
&& A_y \phi_{i,j+\frac{1}{2}} = \frac{1}{2}(\phi_{i,j+1}+\phi_{i,j}),\quad  D_y \phi_{i,j+\frac{1}{2}} = \frac{1}{h_y}(\phi_{i,j+1}-\phi_{i,j}),\quad A_y \phi, D_y \phi\in \mathcal{E}^{ns}_{\overline{x}\times y},\\
&& A_y u_{i+\frac{1}{2},j+\frac{1}{2}} = \frac{1}{2}(u_{i+\frac{1}{2},j+1}+u_{i+\frac{1}{2},j}),\quad  D_y u_{i+\frac{1}{2},j+\frac{1}{2}} = \frac{1}{h_y}(u_{i+\frac{1}{2},j+1}-u_{i+\frac{1}{2},j}),\quad A_y u, D_y u\in \mathcal{V}_{x\times y}.
\end{eqnarray*}
The discrete Laplacian operator $\Delta_h:\mathcal{C}_{\overline{x}\times \overline{y}}\rightarrow \mathcal{C}_{x\times y}$ is defined as
$$\Delta_h \phi = d_x(D_x \phi)+d_y(D_y \phi).$$

We discretize the physical variables that satisfy Neumann boundary conditions at the cell-center and the ones that satisfy Dirichlet boundary conditions at the edge-center. So, the cell-centered functions $\phi,\mu\in \mathcal{C}_{\overline{x}\times \overline{y}}$ satisfy homogeneous Neumann boundary conditions if and only if
\ben
&& \phi_{0,j} = \phi_{1,j},\quad \phi_{N_x,j} = \phi_{N_x+1,j},\quad \mu_{0,j} = \mu_{1,j},\quad \mu_{N_x,j} = \mu_{N_x+1,j},\quad j = 1,2,\ldots,N_y,\label{dis-BD1}\\
&& \phi_{i,0} = \phi_{i,1},\quad \phi_{i,N_y} = \phi_{i,N_y+1},\quad \mu_{i,0} = \mu_{i,1},\quad \mu_{i,N_y} = \mu_{i,N_y+1},\quad i = 0,1,\ldots,N_x+1.\label{dis-BD2}
\een

The velocity $\bu = (u, v)$ (for $u \in\mathcal{E}^{ew}_{x\times \overline{y}},~v \in\mathcal{E}^{ns}_{\overline{x}\times y}$) satisfies the no-slip (Dirichlet) boundary conditions $\bu|_{\Omega}=0$ if and only if
\ben
{u}_{\frac{1}{2},j} = {u}_{N_x+\frac{1}{2},j} = 0,&& ~j = 1,2,\ldots,N_y,\label{dis-BD3}\\
A_y {u}_{i+\frac{1}{2},\frac{1}{2}} = A_y {u}_{i+\frac{1}{2},N_y+\frac{1}{2}} = 0,&& ~i = 0,1,\ldots,N_x,\label{dis-BD4}\\
{v}_{i,\frac{1}{2}} = {v}_{i,N_y+\frac{1}{2}} = 0,&& ~i = 1,2,\ldots,N_x,\label{dis-BD5}\\
A_x {v}_{\frac{1}{2},j+\frac{1}{2}} = A_x {v}_{N_x+\frac{1}{2},j+\frac{1}{2}} = 0,&& ~j = 0,1,\ldots,N_y.\label{dis-BD6}
\een
It is easy to show that
\beq\label{dis-BD}
D_x \phi,D_x \mu,u \in \mathcal{E}^{ew0}_{x\times \overline{y}},~ D_y \phi,D_y \mu,v \in \mathcal{E}^{ns0}_{\overline{x}\times y},~ A_y u ,A_x v \in \mathcal{V}^{0}_{x\times y},~ D_y u \in \mathcal{V}^{ew0}_{x\times y},~ D_x v \in \mathcal{V}^{ns0}_{x\times y}.
\eeq

Based on the above definitions, we define the following discrete 2D weighted inner-products:
\begin{eqnarray*}
&& (\phi,\psi)_2= h_x h_y \sum\limits_{i=1}^{N_x}\sum\limits_{j=1}^{N_y}\phi_{i,j}\psi_{i,j}, \\
&& [u,r]_{ew} = (a_x(u r),1)_2,\quad [v,w]_{ns} = (a_y(v w),1)_2,\quad \langle f,g \rangle_{vc} = \left(a_x\big(a_y(f g)\big),1\right)_2,
\end{eqnarray*}
and the corresponding discrete norms:
$$\|\phi\|_2 = (\phi,\phi)_2^{\frac{1}{2}},\quad \|u\|_{ew} = [u,u]_{ew}^{\frac{1}{2}},\quad \|v\|_{ns} = [v,v]_{ns}^{\frac{1}{2}},\quad \|f\|_{vc} = \langle f,f \rangle_{vc}^{\frac{1}{2}}.$$
Specially, we have
\begin{eqnarray*}
&& [u,r]_{ew} = h_x h_y \sum\limits_{i=1}^{N_x-1}\sum\limits_{j=1}^{N_y}u_{i+\frac{1}{2},j} r_{i+\frac{1}{2},j},~\textrm{if}~ u\in \mathcal{E}^{ew0}_{x\times y},\\
&& [v,w]_{ns} = h_x h_y \sum\limits_{i=1}^{N_x}\sum\limits_{j=1}^{N_y-1}v_{i,j+\frac{1}{2}} w_{i,j+\frac{1}{2}},~\textrm{if}~ v\in \mathcal{E}^{ns0}_{x\times y},\\
&& \langle f,g \rangle_{vc} = h_x h_y \sum\limits_{i=1}^{N_x-1}\sum\limits_{j=1}^{N_y-1}f_{i+\frac{1}{2},j+\frac{1}{2}} g_{i+\frac{1}{2},j+\frac{1}{2}},~\textrm{if}~ f\in \mathcal{V}^{0}_{x\times y}.
\end{eqnarray*}
For the edge-centered function $\bu = (u , v),~u \in \mathcal{E}^{ew}_{x\times \overline{y}}, v \in \mathcal{E}^{ns}_{\overline{x}\times y}$ and the cell-centered function $\phi\in C_{\overline{x}\times \overline{y}},$ we define the following norms
$$\|\bu\|_2 := \sqrt{\|u\|_{ew}^2+\|v\|_{ns}^2}, \quad \|\nabla \phi\|_2:= \sqrt{\|D_x \phi\|_{ew}^2+\|D_y \phi\|_{ns}^2}.$$

Next, we introduce some useful lemmas \cite{Gong&Zhao&WangSISC2}.

\begin{lem}\label{lem-1}
For $\phi\in \mathcal{C}_{\overline{x}\times \overline{y}},~u \in \mathcal{E}^{ew0}_{x\times y},~v \in \mathcal{E}^{ns0}_{x\times y},$ there exist the following identities:
\ben
&& [A_x \phi,u]_{ew} = (\phi,a_x u)_2,\quad [D_x \phi,u]_{ew} + (\phi,d_x u)_2 = 0,\label{SumPart1}\\
&& [A_y \phi,v]_{ns} = (\phi,a_y v)_2,\quad [D_y \phi,v]_{ns} + (\phi,d_y v)_2 = 0.\label{SumPart2}
\een
\end{lem}

\begin{lem}\label{lem-2}
For $f\in \mathcal{V}^0_{x\times y}$, $u \in \mathcal{E}^{ew}_{x\times \overline{y}},$ $v \in \mathcal{E}^{ns}_{\overline{x}\times y},$  there exists the identities:
\ben
&&\langle f,A_y u\rangle_{vc} = [a_y f,u]_{ew},\quad \langle f,D_y u\rangle_{vc} + [d_y f,u]_{ew} = 0,\label{SumPart3}\\
&&\langle f,A_x v\rangle_{vc} = [a_x f,v]_{ns},\quad \langle f,D_x v\rangle_{vc} + [d_x f,v]_{ns} = 0.\label{SumPart4}
\een
\end{lem}

\begin{lem}\label{lem-3}
For $f\in \mathcal{V}_{x\times y}$, $u \in \mathcal{E}^{ew}_{x\times \overline{y}},$ $v \in \mathcal{E}^{ns}_{\overline{x}\times y}$ and $A_y u,A_x v \in \mathcal{V}^0_{x\times y},$ there exist the identities:
\ben
&&\langle f,D_y u\rangle_{vc} + [d_y f,u]_{ew} = 0,\label{SumPart5}\\
&&\langle f,D_x v\rangle_{vc} + [d_x f,v]_{ns} = 0.\label{SumPart6}
\een
\end{lem}


\subsection{Spatial discretization}
With the notation above, we are ready to explain the spatial discretization of the CHNS on staggered grids. To write it in a clear manner, we write the spatial discretization in the following component form
\begin{subequations}\label{model-CF}
\begin{align}
&\label{model-CF1} \rho  \partial_t u + \frac{\rho}{2} \Big( u \partial_x u + \partial_x (u u) +  v \partial_y u + \partial_y (v u)\Big)  = -\partial_x p +  \eta \partial_x \partial_x u - \phi \partial_x \mu,\\
&\label{model-CF2} \rho  \partial_t  v +\frac{\rho}{2} \Big( u \partial_x v + \partial_x (u v) + v \partial_y v + \partial_y (v v)\Big)  = -\partial_y p +  \eta \partial_y \partial_y  v - \phi \partial_y \mu,\\
&\label{model-CF3} \partial_x u + \partial_y v  =  0,\\
&\label{model-CF4} \partial_t \phi + \partial_x (\phi u) + \partial_y (\phi v)  = \partial_x \Big( M(\phi) \partial_x \mu \Big)  + \partial_y \Big( M(\phi) \partial_y \mu \Big),\\
&\label{model-CF5} \mu = -\gamma \varepsilon \Delta \phi +\frac{\gamma}{\varepsilon} f'(\phi),
\end{align}
\end{subequations}
where $\bu=(u, v)$. Then, we apply the finite difference spatial discretization on a staggered grid on the system \eqref{model-CF}.

\begin{scheme}[Semi-discrete Scheme] \label{sch:semi-discrete}
Applying staggered-grid finite differences in space to the system \eqref{model-CF} with boundary conditions \eqref{eq:CHNS-boundary}, we obtain a semi-discrete scheme as follows:
\begin{subequations}\label{semi-discrete}
\begin{align}
&\label{semi1} \Big\{\rho \frac{d}{dt} u +  \frac{\rho}{2} \Big(u D_x (a_x  u) + A_x \big(d_x(u u)\big) + a_y (A_x v D_y u) + d_y (A_y u A_x v)\Big)  \\
&\nonumber \qquad =  -D_x p + \eta  D_x  d_x u  - A_x \phi D_x \mu \Big\}_{i+\frac{1}{2},j},\\
&\label{semi2} \Big\{\rho \frac{d}{dt} v + \frac{\rho}{2} \Big(a_x (A_y u D_x v) + d_x \big(A_y u A_x v \big) + v D_y(a_y v) + A_y \big(d_y (v v)\big)\Big)  \\
&\nonumber \qquad =  -D_y p +  + \eta  D_y d_y v - A_y \phi D_y \mu\Big\}_{i,j+\frac{1}{2}},\\
&\label{semi3} \Big\{d_x u + d_y v  =  0\Big\}_{i,j},\\
&\label{semi4} \Big\{\frac{d}{dt} \phi + d_x (A_x\phi u) + d_y (A_y\phi v)  = d_x \big(M(A_x\phi) D_x \mu\big) + d_y \big(M(A_y\phi) D_y \mu\big)\Big\}_{i,j},\\
&\label{semi5} \Big\{\mu = -\gamma \varepsilon \Delta_h\phi  +\frac{\gamma}{\varepsilon} f'(\phi) \Big\}_{i,j},
\end{align}
\end{subequations}
where $u, v, \phi, \mu$ satisfy the discrete boundary conditions \eqref{dis-BD1}-\eqref{dis-BD6}, and $i=1,\ldots,N_x-1,j=1,\ldots,N_y$ for Eq. \eqref{semi1}, $i=1,\ldots,N_x,j=1,\ldots,N_y-1$ for Eq. \eqref{semi2}, $i=1,\ldots,N_x,j=1,\ldots,N_y$ for Eqs. \eqref{semi3}-\eqref{semi5}.
\end{scheme}

We can easily show that the semi-discrete scheme \ref{sch:semi-discrete} have the following two properties.
\begin{thm}\label{t-semiDis-EDL}
The semi-discrete scheme \ref{sch:semi-discrete} preserves the discrete mass conservation law given by
\beq\label{semiDis-MCL}
\frac{d}{dt} (\phi,1)_2 = 0.
\eeq
\end{thm}

\begin{proof}
This can be easily verified. IN fact, we can compute the discrete inner product of \eqref{semi4} with constant function $1$, and use \eqref{dis-BD} and Lemma \ref{lem-1}. Then, this leads us to \eqref{semiDis-MCL}.
\end{proof}

\begin{thm}
The semi-discrete scheme \ref{sch:semi-discrete} preserves the  discrete energy dissipation law 
\beq
\label{semiDis-EDL} 
\frac{d}{dt}E_h + \Big(\eta, (d_x u)^2+(d_y v)^2\Big)_2  + \big[M(A_x\phi), (D_x \mu)^2\big]_{ew} + \big[M(A_y\phi), (D_y \mu)^2\big]_{ns} = 0,
\eeq
where $E_h$ is the discrete energy functional defined as
\beq\label{semiDis-energy}
E_h = \frac{1}{2}\|\bu\|_2^2
+\frac{\gamma \varepsilon}{2}\|\nabla\phi\|_2^2+\frac{\gamma}{\varepsilon} (f(\phi), 1)_2.
\eeq
\end{thm}

\begin{proof}

Noticing that $u \in \mathcal{E}^{ew0}_{x\times y}$ and using Lemma \ref{lem-1}, we have
\beq\label{semi-r-v1-1}
\big[u D_x (a_x u) + A_x \big(d_x(u u)\big), u\big]_{ew} = -\big(a_x u ,d_x(u u )\big)_{2} + (d_x(u u),a_x u)_{2} = 0.
\eeq
Eq. \eqref{dis-BD} implies that $ u \in \mathcal{E}^{ew0}_{x\times \overline{y}},~ v \in \mathcal{E}^{ns0}_{\overline{x}\times y},$ and thus $A_x v D_y u, A_y u A_x v \in \mathcal{V}^{0}_{x\times y}.$ According to Lemma \ref{lem-2}, we then have
\beq\label{semi-r-v1-2}
 \Big[a_y (A_x v D_y u) + d_y (A_y u A_x v), u\Big]_{ew} = \langle A_x v D_y u,A_y u\rangle_{vc}-\langle A_y u A_x v,D_y u\rangle_{vc}=0.
 \eeq
Similarly, we can deduce
\ben
&&\Big[a_x (A_y u D_x v) + d_x \big(A_y u A_x v\big),v\Big]_{ns} = 0,~ \Big[v D_y(a_y v) + A_y \big(d_y (v v)\big),v_2\Big]_{ns} = 0,\label{semi-r-v2}\\
&& [D_x p,u]_{ew}+[D_y p,v]_{ns} = -(p,d_x u +d_y v)_2 = 0,\label{semi-r-p}\\
&& \Big[\eta D_x\big( d_x u\big),u\Big]_{ew}+\Big[\eta  D_y\big( d_y v\big),v\Big]_{ns} = -\Big(\eta, (d_x u)^2+(d_y v)^2\Big)_2,\\
&& [A_x \phi D_x\mu,u]_{ew}+[A_y \phi D_y\mu,v]_{ns} = -\big(\mu,d_x(A_x \phi u)+d_y(A_y \phi v)\big)_2,\label{semi-r-mu-phi-bv}\\
&& \nonumber \Big(d_x \big(M(A_x\phi) D_x \mu\big) + d_y \big(M(A_y\phi) D_y \mu\big),\mu\Big)_2 = \\
&& \quad - \big[M(A_x\phi), (D_x \mu)^2\big]_{ew} - \big[M(A_y\phi), (D_y \mu)^2\big]_{ns}.\label{semi-r-lap-mu}
\een
Computing the discrete inner product of \eqref{semi1}, \eqref{semi2} and \eqref{semi4} with $u,$ $v$ and $\mu,$ respectively, and using \eqref{semi-r-v1-1}-\eqref{semi-r-lap-mu}, we have
\begin{multline}\label{semi-bv+phi-inner}
 [u,{u}_t]_{ew}+[v,{v}_t]_{ns}+(\mu,\phi_t)_2 =\\
  - \Big(\eta, (d_x u)^2+(d_y v)^2\Big)_2  - \big[M(A_x\phi), (D_x \mu)^2\big]_{ew} - \big[M(A_y\phi), (D_y \mu)^2\big]_{ns}. 
\end{multline}
In addition, it is not hard to calculate
\begin{align*}
\frac{d}{dt}E_h &= [u,{u}_t]_{ew}+[v,{v}_t]_{ns}
+\gamma \varepsilon ([D_x\phi,D_x\phi_t]_{ew}+[D_y\phi,D_y\phi_t]_{ns}) + (\frac{\gamma}{\varepsilon}f(\phi), \phi_t)_2\\
&= [u,{u}_t]_{ew}+[v,{v}_t]_{ns}
- (\gamma \varepsilon \Delta_h \phi,\phi_t)_2+(\frac{\gamma}{\varepsilon} f(\phi),\phi_t)_2\\
&= [u,{u}_t]_{ew}+[v,{v}_t]_{ns}+(\mu,\phi_t)_2.
\end{align*}

Adding the two equations above, will lead us to \eqref{semiDis-EDL}. This completes the proof.
\end{proof}

\section{Decoupled Time Discretization} \label{sect:time}
With the semi-discrete scheme \ref{sch:semi-discrete}, we are ready to introduce the temporal discretization. Notice the fact the spatial discretization and the temporal discretization are independent.  To simply our notations, we apply the temporal discretization directly on the continuous CHNS system \eqref{eq:CHNS}, instead of on the semi-discrete scheme in \eqref{semi1}-\eqref{semi5}. However, we emphasize by using the same time discretization on the semi-discrete scheme \ref{sch:semi-discrete}, the full discrete scheme will be immediately obtained.

\subsection{Notations for temporal discretization}
To better explain the proposed numerical algorithms, we introduce some notations for the temporal discretization. 
Consider the time domain $t\in [0, T]$. We discretize it into equally distanced intervals $0=t_0 < t_1 < t_2 < \cdots < t_N = T$, with $\delta t = \frac{T}{N}$, and $t_i = i \delta t$, $i=0,1,2\cdots, N$. 
Following the notations in our previous work, we introduce
\ben
&& \phi^{n+\frac{1}{2}} = \frac{1}{2}( \phi^{n+1} + \phi^n), \quad  \overline{\phi}^{n+\frac{1}{2}} = \frac{1}{2}( 3 \phi^n - \phi^{n-1}),  \\
&& \overline{\phi}^{n+\frac{1}{4}} = \frac{5}{4} \phi^n - \frac{1}{4} \phi^{n-1}, \quad \overline{ \phi}^{n+\frac{3}{4}} = \frac{7}{4} \phi^n - \frac{3}{4} \phi^{n-1}, \\
&& \tilde{\bu}^{n+\frac{1}{2}} = \frac{1}{2} ( 3 \bu^n - \bu^{n-1}),  \quad  \overline{\bu}^{n+\frac{1}{4}} = \frac{5}{4}  \bu^n -\frac{1}{4} \bu^{n-1},  \quad  \overline{\bu}^{n+\frac{3}{4}} = \frac{7}{4} \bu^n - \frac{3}{4} \bu^{n-1}.
\een

For any $f, g \in [L^2(\Omega)]^d$ with $d$ the vector dimension, we denote the inner product and induced $L^2$ norm as
\beq
(f, g)  = \int_\Omega \sum_{i=1}^d f_i g_i d\bx, \quad \| f \| = \sqrt{ (f, f)}.
\eeq 

Recall the bulk potential in this paper \eqref{eq:bulk_potential}. For the semi-implicit discretization of $f'$ in the interval $[t_n, t_{n+1}]$, we denote it as $\frac{\delta f}{\delta(\phi^n, \phi^{n+1})}$. 
We use the classical difference quotient 
\beq \label{eq:f_prime_case_1}
\frac{\delta f}{\delta(\phi^n, \phi^{n+1})} =  \Big[\frac{ (\phi^n)^2 + (\phi^{n+1})^2}{2} - 1\Big] \frac{ \phi^n + \phi^{n+1}}{2}.
\eeq 
It has the property that
\beq \label{eq:bulk_potential_requirement}
f(\phi^{n+1}) - f(\phi^n) = (\phi^{n+1} - \phi^n) \frac{\delta f}{\delta(\phi^n, \phi^{n+1})}.
\eeq 

We emphasis that the property in \eqref{eq:bulk_potential_requirement} is a rather strong requirement.
 In general, we could require a weaker property
\beq
f(\phi^{n+1}) - f(\phi^n) \leq (\phi^{n+1} - \phi^n) \frac{\delta f}{\delta(\phi^n, \phi^{n+1})}.
\eeq 
The advantage of the choice in \eqref{eq:f_prime_case_1} is that the requirement in \eqref{eq:bulk_potential_requirement} is automatically satisfied. Meanwhile, the major disadvantage is that it is nonlinear, such that each time step, a nonlinear problem has to be solved. 

\begin{rem}
Note that the choice of $\frac{\delta f}{\delta(\phi^n, \phi^{n+1})}$  is not unique.   If we can assume
\beq
\max_{\phi \in \mathbb{R}} |f'(\phi) | \leq L,
\eeq 
with $L$ a constant, we can utilize the semi-implicit stabilized discretization to linearize it as
\beq \label{eq:f_prime_case_2}
\frac{\delta f}{\delta(\phi^n, \phi^{n+1})} = f'(\frac{3}{2} \phi^n-\frac{1}{2}\phi^{n-1}) - A \delta t \Delta (\phi^{n+1} - \phi^n) + B (\phi^{n+1} - 2\phi^n + \phi^{n-1}).
\eeq 
where $A$ and $B$ are stabilization constants \cite{Wang&YuJSC2018}. It can be shown that the scheme will be energy stable, given $A$ and $B$ is big enough.
\end{rem}

\begin{rem}
Or as a simple case, we may use the linearized approximation
\beq \label{eq:f_prime_case_3}
\frac{\delta f}{\delta(\phi^n, \phi^{n+1})} =f'(\frac{3}{2} \phi^n-\frac{1}{2}\phi^{n-1})  + C\delta t(\phi^{n+1}-\phi^n),
\eeq 
with $C$ a stabilization constant.  
\end{rem}

\begin{rem}
The possible choices in \eqref{eq:f_prime_case_1}, \eqref{eq:f_prime_case_2} and \eqref{eq:f_prime_case_3} all have their advantages and disadvantages. Once can even further simply this by introducing auxiliary variables to result in linear systems.  However, the major focus of this paper is to design decoupled numerical schemes, such that the CHNS system can be easily solved.
\end{rem}

\begin{scheme}[Second-Order Splitting Scheme]   \label{Scheme:sch-general}
To solve the CHNS system of \eqref{eq:CHNS-gradient-flow} in the time interval $[t_n, t_{n+1}]$, we use the  (second-order) Strang-Marchuk operator splitting method for $\cG_a$ and $\cG_s$. Then each time marching step will require the following three sub-steps.
\begin{itemize}
\item Step 1: In the interval $[t_n, t_{n+\half}]$ , solve the problem
$$\Lambda \partial_t \Psi = \cG_s \frac{\delta E}{\delta \Psi}, \mbox{ with } \Psi(t=t_n) = \Psi^n,$$
and get $\Psi_{\star} = \Psi(t=t_{n+\half})$.
\item Step 2: In the interval $[t_n, t_{n+1}]$, solve the problem
$$
\Lambda \partial_t \Psi = \cG_a \frac{\delta E}{\delta \Psi} \mbox{ with } \Psi(t=t_n) =\Psi_{\star},
$$
and get $\Psi_{\star \star} = \Psi(t=t_{n+1})$.
\item Step 3: In the interval $[t_{n+\half}, t_{n+1}]$, solve the problem
$$
\Lambda \partial_t \Psi = \cG_s \frac{\delta E}{\delta \Psi}, \mbox{   with } \Psi(t=t_{n+\half}) = \Psi_{\star \star}.$$
and get $\Psi^{n+1} = \Psi(t_{n+1})$.
\end{itemize}
\end{scheme}

The scheme above is a second-order operator splitting algorithm. Given specific operators $\cG_s$ and $\cG_a$, the proposed scheme \ref{Scheme:sch-general} can be specified. In the rest of this section, we will discuss several variants of decoupled numerical schemes.

\subsection{Second-order decoupled time-marching scheme for the CHNS model}
If we plug the mobility operator splitting formula in \eqref{eq:Splitting-Case1}  into the general scheme \ref{Scheme:sch-general}, we obtain the second-order numerical schemes below.
\begin{scheme} \label{Scheme:sch-1}
Given $(\bu^n, \phi^n)$, $(\bu^{n-1}, \phi^{n-1})$ and $\bu^n \cdot \bn=0$, we can obtain $(\bu^{n+1}, \phi^{n+1})$ in the following three steps:
\begin{itemize}
\item Step 1: in $[t_n, t_{n+\frac{1}{2}}]$, we solve $(\bu_{\star}, \phi_\star)$ via the following two decoupled sub-steps.
\begin{itemize}
\item Step 1.1: solve $\bu_\star$ via
\begin{subequations} \label{eq:sch1-s1}
\begin{align}
&\label{eq:sch1-s1-1} \rho \frac{\bu_\star - \bu^n}{\delta t/2}  + B(\bar{\bu}^{n+\frac{1}{4}}, \frac{\bu_\star + \bu^n}{2})= -\nabla p + \eta \Delta \frac{\bu_\star + \bu^n}{2} , \\
&\label{eq:sch1-s1-2} \nabla \cdot \frac{\bu_\star + \bu^n}{2} = 0, \\
&\label{eq:sch1-s1-3} \bu_\star =0, \quad \mbox{ on } \partial \Omega.
\end{align}
\end{subequations}
\item Step 1.2: solve $\phi_\star$ via
\begin{subequations} \label{eq:sch1-s2}
\begin{align}
&  \label{eq:sch1-s2-s1} \frac{\phi_\star - \phi^n}{\delta t/2} =  \nabla (M(\overline{\phi}^{n+\frac{1}{4}}) \nabla \mu^{n+\frac{1}{4}}),  \\
& \label{eq:sch1-s2-s2} \mu^{n+\frac{1}{4}} = -\frac{\gamma \varepsilon}{2}(\Delta \phi_\star + \Delta \phi^n) + \frac{\delta f}{\delta (\phi_\star, \phi^n)},\\
& \label{eq:sch1-s2-s3} \nabla \mu^{n+\frac{1}{4}} \cdot \bn =0, \quad \nabla \phi_\star \cdot \bn = 0, \quad \mbox{ on } \partial \Omega.
\end{align}
\end{subequations}
\end{itemize}

\item  Step 2, In $[t_n, t_{n+1}]$, we solve $(\bu_{\star\star}, \phi_{\star\star})$ via two sub-steps.
\begin{itemize}
\item Step 2.1, Solve $(\phi_{\star \star}, p)$ via the following system
\begin{subequations} \label{eq:sch1-s2X}
\begin{align}
&\frac{1}{\delta t}( \phi_{\star\star} - \phi_\star) + \nabla \cdot ( \bu_\star \overline{\phi}^{n+\frac{1}{2}}) = \nabla \cdot ( \frac{\delta t}{2 \rho}  (\overline{\phi}^{n+\frac{1}{2}})^2  \nabla \mu^{n+\frac{1}{2}}) 
+ \frac{\delta t}{2\rho} \nabla \cdot (\overline{\phi}^{n+\frac{1}{2}} \nabla p), \\
&\mu^{n+\frac{1}{2}} = -\frac{\gamma \varepsilon}{2}(\Delta \phi_{\star\star} + \Delta \phi_{\star}) +\frac{\gamma}{\varepsilon} \frac{\delta f}{\delta (\phi_{\star \star}, \phi_{\star})},\\
&-\Delta p - \nabla \cdot ( \overline{\phi}^{n+\frac{1}{2}} \nabla \mu^{n+\frac{1}{2}}) = 0, \\
&\nabla \mu^{n+\frac{1}{2}} \cdot \bn =0, \quad \nabla \phi_{\star \star}  \cdot \bn = 0,  \quad \nabla p \cdot \bn = 0, \quad \mbox{ on } \partial \Omega.
\end{align}
\end{subequations}
\item Step 2.2, Update $\bu_{\star \star}$ via
\beq  \label{eq:sch1-s2Y}
\bu_{\star \star} = \bu_\star - \frac{\delta t}{\rho} (\nabla p + \overline{\phi}^{n+\frac{1}{2}} \nabla \mu^{n+\frac{1}{2}}),
\eeq 
with $\mu^{n+\frac{1}{2}}$ defined in \eqref{eq:sch1-s2X}.
\end{itemize}

\item Step 3,  In $[t_{n+\frac{1}{2}}, t_{n+1}]$, we solve $(\bu^{n+1}, \phi^{n+1})$ via the following decoupled two sub-steps.
\begin{itemize}
\item Step 3.1, Solve $\bu^{n+1}$ via
\begin{subequations} \label{eq:sch1-s3}
\begin{align}
\label{eq:sch1-s3-1}&\rho \frac{\bu^{n+1}- \bu_{\star \star}}{\delta t/2}  + B(\bar{\bu}^{n+\frac{3}{4}}, \frac{\bu^{n+1} + \bu_{\star\star}}{2})= -\nabla p + \eta \Delta \frac{\bu_{\star \star} + \bu^{n+1}}{2} , \\
\label{eq:sch1-s3-2}&\nabla \cdot \frac{\bu^{n+1} + \bu_{\star\star}}{2}= 0, \\
\label{eq:sch1-s3-3}&\bu^{n+1} =0, \quad \mbox{ on } \partial \Omega.
\end{align}
\end{subequations}
\item Step 3.2, Solve $\phi^{n+1}$ via
\begin{subequations}
\begin{align}
&\label{eq:sch1-s3-4} \frac{\phi^{n+1}  - \phi_{\star \star}}{\delta t/2} =  \nabla \cdot( M(\overline{\phi}^{n+\frac{3}{4}}) \nabla  \mu^{n+\frac{3}{4}}),  \\
&\label{eq:sch1-s3-5} \mu^{n+\frac{3}{4}} = -\frac{\gamma \varepsilon}{2}(\Delta \phi^{n+1} + \Delta \phi_{\star\star}) +\frac{\gamma}{\varepsilon} \frac{\delta f}{\delta (\phi^{n+1}, \phi_{\star\star})},\\
&\label{eq:sch1-s3-6} \nabla \mu^{n+\frac{3}{4}} \cdot \bn =0, \quad \nabla \phi^{n+1} \cdot \bn = 0, \quad \mbox{ on } \partial \Omega.
\end{align}
\end{subequations}
\end{itemize}

\end{itemize}
\end{scheme}

The scheme above is second-order accurate in time. Moreover, in each step, only problems with smaller sizes need to be solved. Here are several remarks.

\begin{rem}
Note that Step 2 in \eqref{eq:sch1-s2X}-\eqref{eq:sch1-s2Y} comes from  the problem
\begin{subequations}\label{sch1:step2}
\begin{align}
\label{sch1:step2-1}  &\rho \frac{\bu_{\star \star}  - \bu_\star}{\delta t} = -\nabla p - \bar{\phi}^{n+\frac{1}{2}} \nabla \mu^{n+\frac{1}{2}}, \\
\label{sch1:step2-2} &\mu^{n+\frac{1}{2}} = -\frac{\gamma \varepsilon}{2}(\Delta \phi_{\star\star} + \Delta \phi_{\star}) + \frac{\gamma}{\varepsilon} \frac{\delta f}{\delta (\phi_{\star \star}, \phi_{\star})},\\
\label{sch1:step2-3} &\nabla \cdot \frac{ \bu_{\star \star} + \bu_\star}{2}= 0, \\
\label{sch1:step2-4}& \frac{\phi_{\star \star} - \phi_\star }{\delta t} = -\nabla \cdot ( \frac{ \bu_{\star \star} + \bu_\star}{2} \bar{\phi}^{n+\frac{1}{2}}), \\
\label{sch1:step2-5} &\nabla \mu^{n+\frac{1}{2}} \cdot \bn =0, \quad \nabla \phi_{\star\star} \cdot \bn = 0, \quad \bu_{\star\star} =0, \quad \mbox{ on } \partial \Omega.
\end{align}
\end{subequations}
Notice  $\bu_{\star \star}$ and $\phi_{\star\star}$ can be decoupled, by realizing the following equality
$$
\frac{\bu_{\star \star} + \bu_{\star}}{2} = \bu_{\star} - \frac{\delta t}{2\rho} (\nabla p + \overline{\phi}^{n+\frac{1}{2}} \nabla \mu^{n+\frac{1}{2}}).
$$ 
Then we rewrite the problem in Step 2 as two decoupled sub-steps in \eqref{eq:sch1-s2X}-\eqref{eq:sch1-s2Y}.
The coupled system of $(\phi_{\star \star}, p)$ is significantly reduced compared with the original coupled CHNS system in \eqref{eq:CHNS}. Also, unlike the coupling in \eqref{eq:CHNS}, its dimension complexity does not increase with the dimension of the problem.
\end{rem}

\begin{rem} \label{rem:Relaxation}
Furthermore, we can even relax the in-compressibility constraint in Step 2. Then the problem of \eqref{sch1:step2} in Step 2 is reduced to
\beq
\left\{
\bea{l}
\rho \frac{\bu_{\star \star}  - \bu_\star}{\delta t} =  - \bar{\phi}^{n+\frac{1}{2}} \nabla \mu^{n+\frac{1}{2}}, \\
\mu^{n+\frac{1}{2}} = -\frac{\gamma \varepsilon}{2}(\Delta \phi_{\star\star} + \Delta \phi_{\star}) + \frac{\gamma}{\varepsilon} \frac{\delta f}{\delta (\phi_{\star \star}, \phi_{\star})},\\
\frac{\phi_{\star \star} - \phi_\star }{\delta t} = -\nabla \cdot ( \frac{ \bu_{\star \star} + \bu_\star}{2} \bar{\phi}^{n+\frac{1}{2}}), \\
\nabla \mu^{n+\frac{1}{2}} \cdot \bn =0, \quad \nabla \phi_{\star\star} \cdot \bn = 0, \quad \bu_{\star\star} =0, \quad \mbox{ on } \partial \Omega.
\eea 
\right.
\eeq 
Hence, Step 2 can be solved by the following two sub-steps:
\begin{itemize}
\item Step 2.1, solve $\phi_{\star \star}$ via
\beq
\left\{
\bea{l}
\frac{1}{\delta t}( \phi_{\star \star} - \phi_\star) + \nabla \cdot ( \bu_\star \overline{\phi}^{n+\frac{1}{2}}) = \nabla \cdot (\frac{\delta t}{2 \rho}  (\overline{\phi}^{n+\frac{1}{2}})^2 \nabla \mu^{n+\frac{1}{2}}), \\
\mu^{n+\frac{1}{2}} = -\frac{\gamma \varepsilon}{2}(\Delta \phi_{\star\star} + \Delta \phi_{\star}) + \frac{\gamma}{\varepsilon} \frac{\delta f}{\delta (\phi_{\star \star}, \phi_{\star})},\\
\nabla \mu^{n+\frac{1}{2}} \cdot \bn =0, \quad \nabla \phi_{\star\star} \cdot \bn = 0, \quad \mbox{ on } \partial \Omega.
\eea 
\right.
\eeq 
\item Step 2.2, update $\bu_{\star \star}$ via
\beq
\bu_{\star \star} = \bu_\star - \frac{\delta t}{\rho} \overline{\phi}^{n+\frac{1}{2}} \nabla \mu^{n+\frac{1}{2}}.
\eeq 
\end{itemize}
Since scheme \ref{Scheme:sch-1} is already easy to solve, we don't attempt this relaxation strategy in this paper.  Interested readers are encouraged to further explore it.
\end{rem}

\begin{thm}[Energy Stability]  \label{thm:sch1}
The proposed scheme \ref{Scheme:sch-1} is energy stable, in the sense that
\begin{multline}
\cE(\bu^{n+1}, \phi^{n+1}) - \cE(\bu^n , \phi^n) \leq  - \frac{\delta t}{2}  \Big[ \eta \| \nabla \frac{\bu_\star + \bu^n}{2}\|^2 +\eta  \| \nabla \frac{ \bu^{n+1} + \bu_{\star \star}}{2}\|^2 \\
 +  \| \sqrt{M(\overline{\phi}^{n+\frac{1}{4}})} \nabla \mu^{n+\frac{1}{4}}\|^2 +  \| \sqrt{M(\overline{\phi}^{n+\frac{3}{4}})} \nabla \mu^{n+\frac{3}{4}} \|^2 \Big], 
\end{multline}
where the energy is defined as
\beq
\cE(\bu, \phi) = E(\bu) + F(\phi), \quad E(\bu) = \int_\Omega \frac{\rho}{2} | \bu|^2 d\bx, \quad F(\phi) =\gamma \int_\Omega  \frac{\varepsilon}{2} |\nabla \phi|^2 + \frac{1}{\varepsilon} f(\phi) d\bx.
\eeq 
\end{thm}

\begin{proof}
If we take inner product of \eqref{eq:sch1-s1-1} with $\frac{\delta t}{2} \frac{\bu_\star + \bu^n}{2}$ and apply the boundary conditions in \eqref{eq:sch1-s1-3}, we will have
\beq \label{eq:proof-tmp1}
E(\bu_\star) - E(\bu^n) = -\frac{\delta t}{2} (\frac{\bu_\star + \bu^n}{2}, \nabla p) -\frac{\delta t}{2}   \eta \| \nabla \frac{\bu_\star + \bu^n}{2}\|^2.
\eeq 
If we take inner product of \eqref{eq:sch1-s1-2} with $\frac{\delta t}{2} p$, we have
\beq \label{eq:proof-tmp2}
\frac{\delta t}{2} (p, \nabla \cdot  \frac{\bu_\star + \bu^n}{2}) = 0.
\eeq 
Then, adding the two equations in \eqref{eq:proof-tmp1} and \eqref{eq:proof-tmp2} above, we get
\beq
E(\bu_\star) - E(\bu^n) = - \frac{\delta t}{2} \eta \| \nabla \frac{\bu_\star + \bu^n}{2}\|^2,
\eeq 
by noticing the boundary condition \eqref{eq:sch1-s1-3} and $\bu^n \cdot \bn =0$.

Similarly, if we take inner product of \eqref{eq:sch1-s2-s1} with $\frac{\delta t}{2} \mu^{n+\frac{1}{4}}$, and inner product of  \eqref{eq:sch1-s2-s2} with $\phi_\star - \phi^n$, apply the boundary condition in \eqref{eq:sch1-s2-s3}, and use the constraint \eqref{eq:bulk_potential_requirement}, we have
\beq
F(\phi_\star) - F(\phi^n) =- \frac{\delta t}{2}  \| \sqrt{M(\overline{\phi}^{n+\frac{1}{4}})} \nabla \mu^{n+\frac{1}{4}} \|^2.
\eeq 
The two equations above give us
\beq \label{eq:tmp-add1}
\cE(\bu_\star , \phi_\star) - \cE(\bu^n, \phi^n) = -\frac{\delta t}{2} \Big[ \eta  \| \nabla \frac{\bu_\star + \bu^n}{2}\|^2 +  \| \sqrt{M(\overline{\phi}^{n+\frac{1}{4}})} \nabla \mu^{n+\frac{1}{4}} \|^2 \Big].
\eeq 

Given that \eqref{eq:sch1-s2X}-\eqref{eq:sch1-s2Y} is equivalent to \eqref{sch1:step2}. We take inner product of \eqref{sch1:step2-1} with $\delta t\frac{ \bu_{\star\star} + \bu_\star}{2}$, inner product of \eqref{sch1:step2-2} with $(\phi_{\star\star} - \phi_\star)$, inner product of \eqref{sch1:step2-3} with $p$ and inner product of \eqref{sch1:step2-4} with $\delta t \mu^{n+\frac{1}{2}}$, and utilize the boundary conditions in \eqref{sch1:step2-5}, we will get
\beq \label{eq:tmp-add2}
\cE(\bu_{\star \star} , \phi_{\star \star}) - \cE(\bu_\star, \phi_\star) = 0.
\eeq 

In a similar manner, we take inner product of \eqref{eq:sch1-s3-1} with $\frac{\delta t}{2} \frac{\bu^{n+1} + \bu_{\star \star}}{2}$,  \eqref{eq:sch1-s3-2} with $p$,  we will get
\beq \label{eq:proof1-s3-1}
E(\bu^{n+1}) - E(\bu_{\star \star})  = - \frac{\delta t}{2} \eta \| \nabla \frac{ \bu^{n+1} + \bu_{\star \star}}{2}\|^2 .
\eeq 
And if we take inner product of  \eqref{eq:sch1-s3-4} with $\frac{\delta t}{2} \mu^{n+\frac{3}{4}}$ and \eqref{eq:sch1-s3-5} with $\phi^{n+1} - \phi_{\star \star}$, we will obtain
\beq \label{eq:proof1-s3-2}
F(\phi^{n+1} ) - F(\phi_{\star \star}) = -\frac{\delta t}{2}  \|\sqrt{M(\overline{\phi}^{n+\frac{3}{4}})} \nabla \mu^{n+\frac{3}{4}} \|^2.
\eeq 
Adding \eqref{eq:proof1-s3-1} and \eqref{eq:proof1-s3-2}, we have
\beq \label{eq:tmp-add3}
\cE(\bu^{n+1}, \phi^{n+1}) - \cE(\bu_{\star \star}, \phi_{\star \star}) = -\frac{\delta t}{2} \Big[ \eta \| \nabla \frac{ \bu^{n+1} + \bu_{\star \star}}{2}\|^2 +  \|\sqrt{M(\overline{\phi}^{n+\frac{3}{4}})} \nabla \mu^{n+\frac{3}{4}} \|^2  \Big].
\eeq 

Overall, if we add the three equations in \eqref{eq:tmp-add1}, \eqref{eq:tmp-add2} and \eqref{eq:tmp-add3} together, the discrete energy law is obtained as
\begin{multline}
\cE(\bu^{n+1}, \phi^{n+1}) - \cE(\bu^n , \phi^n) \leq  - \frac{\delta t}{2}  \Big[ \eta \| \nabla \frac{\bu_\star + \bu^n}{2}\|^2 +\eta  \| \nabla \frac{ \bu^{n+1} + \bu_{\star \star}}{2}\|^2\\
  +  \| \sqrt{M(\overline{\phi}^{n+\frac{1}{4}})} \nabla \mu^{n+\frac{1}{4}}\|^2 + \|\sqrt{M(\overline{\phi}^{n+\frac{3}{4}})} \nabla \mu^{n+\frac{3}{4}} \|^2 \Big].
\end{multline}
\end{proof}

\subsection{An alternative  second-order decoupled numerical scheme based on a different operator splitting strategy}

Notice that the splitting of the mobility operator $\cG= \cG_a + \cG_s$ is not unique. Different splitting strategies will lead to different numerical algorithms. In particular, we can also introduce the following splitting
\beq \label{eq:Splitting-Case2}
\cG_a = 
\begin{pmatrix}
0 & -\phi \nabla \bullet  \\
-\nabla \cdot (\bullet \phi) & \nabla \cdot ( M(\phi) \nabla \bullet )\\
\end{pmatrix}, 
\quad 
\cG_s =
\begin{pmatrix}
\eta \Delta \bullet - B(\bu, \bullet ) & 0 \\
0 & 0
\end{pmatrix}.
\eeq 

In the meanwhile, if we plug in the splitting operators defined in \eqref{eq:Splitting-Case2}, we obtain the following second-order operator splitting scheme.

\begin{scheme} \label{Scheme:sch-2}
Given $(\bu^n, \phi^n)$, $(\bu^{n-1}, \phi^{n-1})$ and $\bu^n \cdot \bn=0$, we can obtain $(\bu^{n+1}, \phi^{n+1})$ in the following three steps:
\begin{itemize}
\item Step 1: In $[t_n, t_{n+\frac{1}{2}}]$, we set $\phi_\star = \phi^n$, and  solve $\bu_{\star}$ via 
\begin{subequations}  \label{eq:sch2-s1}
\begin{align}
&\label{eq:sch2-s1-1}\rho \frac{\bu_\star - \bu^n}{\delta t/2}  + B(\bar{\bu}^{n+\frac{1}{4}}, \frac{\bu_\star + \bu^n}{2})= -\nabla p + \eta \Delta \frac{\bu_\star + \bu^n}{2} , \\
&\label{eq:sch2-s1-2}\nabla \cdot \frac{\bu_\star +\bu^n}{2} = 0, \\
&\label{eq:sch2-s1-3}\bu_\star =0, \quad \mbox{ on } \partial \Omega.
\end{align}
\end{subequations}
\item Step 2, In $[t_n, t_{n+1}]$, we solve $(\bu_{\star\star}, \phi_{\star\star})$ via the following two sub-steps.
\begin{itemize}
\item Step 2.1, Solve $(\phi_{\star \star}, p)$ via the following system
\begin{subequations} \label{eq:sch2-s2X}
\begin{align}
&\frac{1}{\delta t}( \phi_{\star\star} - \phi_\star) + \nabla \cdot ( \bu_\star \overline{\phi}^{n+\frac{1}{2}}) = \\
& \nonumber  \nabla \cdot ((M(\overline{\phi}^{n+\frac{1}{2}}) + \frac{\delta t}{2 \rho}  (\overline{\phi}^{n+\frac{1}{2}})^2)  \nabla \mu^{n+\frac{1}{2}}) 
+ \frac{\delta t}{2\rho} \nabla \cdot (\overline{\phi}^{n+\frac{1}{2}} \nabla p), \\
&\mu^{n+\frac{1}{2}} = -\frac{\gamma \varepsilon}{2}(\Delta \phi_{\star\star} + \Delta \phi_{\star}) +\frac{\gamma}{\varepsilon} \frac{\delta f}{\delta (\phi_{\star \star}, \phi_{\star})},\\
&-\Delta p - \nabla \cdot ( \overline{\phi}^{n+\frac{1}{2}} \nabla \mu^{n+\frac{1}{2}}) = 0, \\
&\nabla \mu^{n+\frac{1}{2}} \cdot \bn =0, \quad \nabla \phi^{n+1} \cdot \bn = 0, \quad  \nabla p \cdot \bn = 0, \quad \mbox{ on } \partial \Omega.
\end{align}
\end{subequations}
\item Step 2.2, Update $\bu_{\star \star}$ via
\beq  \label{eq:sch2-s2Y}
\bu_{\star \star} = \bu_\star - \frac{\delta t}{\rho} (\nabla p + \overline{\phi}^{n+\frac{1}{2}} \nabla \mu^{n+\frac{1}{2}}).
\eeq
\end{itemize}

\item Step 3, In $[t_{n+\frac{1}{2}}, t_{n+1}]$, we set $\phi^{n+1} = \phi_{\star \star}$, and solve $\bu^{n+1}$ via
\begin{subequations} \label{eq:sch2-s3}
\begin{align}
&\label{eq:sch2-s3-1} \rho \frac{\bu^{n+1}- \bu_{\star \star}}{\delta t/2}  + B(\bar{\bu}^{n+\frac{3}{4}}, \frac{\bu^{n+1} + \bu_{\star\star}}{2})= -\nabla p + \eta \Delta \frac{\bu_{\star \star} + \bu^{n+1}}{2} , \\
&\label{eq:sch2-s3-2}\nabla \cdot \bu^{n+1}= 0, \\
&\label{eq:sch2-s3-3} \bu^{n+1} =0, \quad \mbox{ on } \partial \Omega.
\end{align}
\end{subequations}
\end{itemize}
\end{scheme}

The scheme \ref{Scheme:sch-2} is computationally efficient than the scheme \ref{Scheme:sch-1} since the phase-field equation only needs to be solved once.  In the scheme \ref{Scheme:sch-1}, the phase-field equations have to be solved three times.

\begin{rem}
Notice that   \eqref{eq:sch2-s2X}-\eqref{eq:sch2-s2Y} is derived from the problem
\begin{subequations} \label{eq:sch2-s2}
\begin{align} 
&\label{eq:sch2-s2-1} \rho \frac{\bu_{\star \star}  - \bu_\star}{\delta t} = -\nabla p - \bar{\phi}^{n+\frac{1}{2}} \nabla \mu^{n+\frac{1}{2}}, \\
&\label{eq:sch2-s2-2} \nabla \cdot \bu_{\star \star}= 0, \\
&\label{eq:sch2-s2-3} \frac{\phi_{\star \star} - \phi_\star }{\delta t} = -\nabla \cdot ( \frac{ \bu_{\star \star} + \bu_\star}{2} \bar{\phi}^{n+\frac{1}{2}}) + \nabla \cdot ( M(\overline{\phi}^{n+\frac{1}{2}}) \nabla  \mu^{n+\frac{1}{2}}), \\
&\label{eq:sch2-s2-4} \mu^{n+\frac{1}{2}} = -\frac{\gamma \varepsilon}{2}(\Delta \phi_{\star\star} + \Delta \phi_{\star}) +\frac{\gamma}{\varepsilon} \frac{\delta f}{\delta (\phi_{\star \star}, \phi_{\star})},\\
&\label{eq:sch2-s2-5} \nabla \mu^{n+\frac{1}{2}} \cdot \bn =0, \quad \nabla \phi_{\star\star} \cdot \bn = 0, \quad \bu_{\star\star} =0, \quad \mbox{ on } \partial \Omega.
\end{align}
\end{subequations}
\end{rem}

\begin{rem}
Similarly, as discussed in Remark \ref{rem:Relaxation},  we can introduce relaxation in Step 2. As a minor modification, we do not necessarily need to restrict the in-compressibility during the operator splitting. With that in mind, we can obtain a relaxed version of Step 2, such that a simplified numerical scheme can be proposed. This idea will not be further elaborated due to space limitation.
\end{rem}

\begin{thm}  \label{thm:sch2}
The proposed scheme \ref{Scheme:sch-2} is energy stable, in the sense that
\begin{multline}
\cE(\bu^{n+1}, \phi^{n+1}) - \cE(\bu^n , \phi^n) \leq \\
  - \frac{\delta t}{2}  \Big[ \eta \| \nabla \frac{\bu_\star + \bu^n}{2}\|^2 +\eta  \| \nabla \frac{ \bu^{n+1} + \bu_{\star \star}}{2}\|^2  + 2  \| \sqrt{M(\overline{\phi}^{n+\frac{1}{2}})} \nabla \mu^{n+\frac{1}{2}} \|^2 \Big], 
\end{multline}
where the energy is defined as
\beq
\cE(\bu, \phi) = E(\bu) + F(\phi), \quad E(\bu) = \int_\Omega \frac{\rho}{2} | \bu|^2 d\bx, \quad F(\phi) =\gamma \int_\Omega  \frac{\varepsilon}{2} |\nabla \phi|^2 + \frac{1}{\varepsilon} f(\phi) d\bx.
\eeq 
\end{thm}

\begin{proof}
The proof is similar to the proof in Theorem \ref{thm:sch1}. Here we only show the major steps. If we take inner product of \eqref{eq:sch2-s1-1} with $\frac{\delta t}{2} \frac{ \bu_{\star} + \bu^n}{2}$,  we will have
\beq \label{eq:sch2-tmpA}
E(\bu_\star) - E(\bu^n) = -\frac{\delta t}{2} (\nabla p, \frac{\bu_\star + \bu^n}{2}) -\frac{\delta t}{2} \eta \| \nabla \frac{\bu_\star + \bu^n}{2}\|^2.
\eeq  
If we take the inner product of \eqref{eq:sch2-s1-2} with $\frac{\delta t}{2} p$, we have
\beq \label{eq:sch2-tmpB}
\frac{\delta t}{2} (p , \nabla \cdot \frac{ \bu_\star + \bu^n}{2}) = 0.
\eeq 
Adding the equations \eqref{eq:sch2-tmpA} and \eqref{eq:sch2-tmpB} above, we get
\beq \label{eq:sch2-tmp1}
E(\bu_\star)  - E(\bu^n) = -\frac{\delta t}{2} \eta \| \nabla \frac{\bu_\star + \bu^n}{2}\|^2.
\eeq 
Given the notations $\phi_\star = \phi^n$,  \eqref{eq:sch2-tmp1} is equivalent to
\beq \label{eq:sch2-E1}
\cE(\bu_\star, \phi_\star) - \cE(\bu^n, \phi^n) = -\frac{\delta t}{2} \eta \| \nabla \frac{\bu_\star + \bu^n}{2}\|^2.
\eeq  
Similarly, if we take inner product of  \eqref{eq:sch2-s2-1} with $\delta t \frac{\bu_{\star \star} + \bu_\star}{2}$, we will have
\beq \label{eq:sch2-s2-tmp1}
E(\bu_{\star \star}) - E(\bu_{\star}) = -\frac{\delta t}{2} (\frac{\bu_{\star \star} + \bu_{\star}}{2}, \nabla p) -  \delta t ( \frac{\bu_{\star \star} + \bu_{\star}}{2}, \overline{\phi}^{n+\frac{1}{2}} \nabla \mu^{n+\frac{1}{2}}).
\eeq 
If we take inner product of \eqref{eq:sch2-s2-2} with $\frac{\delta t}{2} p$, we have
\beq\label{eq:sch2-s2-tmp2}
\frac{\delta t}{2} (p, \nabla \cdot  \frac{\bu_{\star \star} + \bu_{\star}}{2}) =0.
\eeq 
If we take inner product of \eqref{eq:sch2-s2-3} with $\delta t \mu^{n+\frac{1}{2}}$,  and \eqref{eq:sch2-s2-4} with $\delta t (\phi_{\star \star} - \phi_\star)$, we have
\beq \label{eq:sch2-s2-tmp3}
F(\phi_{\star \star}) - F(\phi_\star) = -\delta t (\nabla \cdot ( \frac{\bu_{\star \star} + \bu_{\star}}{2} \overline{\phi}^{n+\frac{1}{2}}), \mu^{n+\frac{1}{2}}) - \delta t  \| \sqrt{M(\overline{\phi}^{n+\frac{1}{2}})}\nabla \mu^{n+\frac{1}{2}} \|^2.
\eeq 
Adding the equations \eqref{eq:sch2-s2-tmp1}, \eqref{eq:sch2-s2-tmp2} and \eqref{eq:sch2-s2-tmp3}  above, we have
\beq \label{eq:sch2-E2}
\cE(\bu_{\star \star}, \phi_{\star \star}) - \cE(\bu_{\star}, \phi_{\star}) = - \delta t  \|\sqrt{M(\overline{\phi}^{n+\frac{1}{2}})} \nabla \mu^{n+\frac{1}{2}} \|^2.
\eeq 

Similarly, for Step 3, if we take inner product of \eqref{eq:sch2-s3-1} with $\frac{\delta t}{2} \frac{ \bu^{n+1} + \bu_{\star \star}}{2}$,  and take inner product of \eqref{eq:sch2-s3-2} with $\delta p$, we have
\beq \label{eq:sch2-E3}
\cE(\bu^{n+1}, \phi^{n+1}) - \cE(\bu_{\star \star}, \phi_{\star\star}) = - \frac{\delta t}{2} \eta \| \nabla \frac{ \bu^{n+1} + \bu_{\star \star}}{2}\|^2.
\eeq 

Adding the equations \eqref{eq:sch2-E1}, \eqref{eq:sch2-E2} and \eqref{eq:sch2-E3} together, we finally obtain the energy dissipation law
\begin{multline}
\cE(\bu^{n+1}, \phi^{n+1}) - \cE(\bu^n , \phi^n) \leq  \\
 - \frac{\delta t}{2}  \Big[ \eta \| \nabla \frac{\bu_\star + \bu^n}{2}\|^2 +\eta  \| \nabla \frac{ \bu^{n+1} + \bu_{\star \star}}{2}\|^2  + 2  \| \sqrt{M(\overline{\phi}^{n+\frac{1}{2}})}\nabla \mu^{n+\frac{1}{2}} \|^2 \Big].
\end{multline}

\end{proof}

\subsection{Strategies to solve the Navier-Stokes equation}
This sub-section further discusses how the Navier-Stoke portion of the proposed scheme in Step 1 and Step 3 can be solved appropriately. 
Recall the 
Crank-Nicolson (CN) type  scheme for the Navier-Stokes equation
\beq \label{eq:NS-part}
\left\{
\bea{l}
\rho \frac{\bu^{n+1} - \bu^n}{\delta t} + \rho B(\overline{\bu}^{n+\frac{1}{2}}, \bu^{n+\frac{1}{2}})  =  - \nabla p^{n+\frac{1}{2}} + \eta \Delta \bu^{n+\frac{1}{2}} +f^{n+\frac{1}{2}}, \\
\nabla \cdot \bu^{n+\frac{1}{2}}  =  0.
\eea 
\right.
\eeq 
Here $f$ is the external force term.
This paper introduces two strategies: (1) the preconditioner method; (2) the velocity projection method.

\subsubsection{Projection method as a preconditioner for the Naver-Stokes equation}

First of all, we discuss the precondition strategy. 
Notice $p^{n+\frac{1}{2}}$ is a Lagrangian multiplier, so it is not advisable to do time marching for $p$, saying $p^{n+\frac{1}{2}}=\frac{1}{2}(p^n + p^{n+1})$. Instead, we treat $p^{n+\frac{1}{2}}$ sa a variable and solve it directly.  Therefore, the Navier-Stokes portion in \eqref{eq:NS-part} can be written as
\beq
\left[
\bea{ll}
\frac{2\rho}{\delta t}- \eta \Delta + \rho B^{n+1} &  \nabla  \\
\nabla \cdot   &  0
\eea 
\right]
\left(
\bea{ll}
\bu^{n+\frac{1}{2}} \\
p^{n+\frac{1}{2}}
\eea
\right)
=
\left[
\bea{ll}
\frac{2\rho}{\delta t}+ \rho B^n &  0  \\
0  &  0
\eea 
\right]
\left(
\bea{ll}
\bu^n \\
0
\eea
\right)
+
\left[
\bea{l}
f^{n+\frac{1}{2}} \\
0
\eea 
\right].
\eeq 
Here $B^{n+1}$ and $B^n$ are schematically to represent the implicit and explicit parts in the convection operator $B(\overline{\bu}^{n+\frac{1}{2}}, \bu^{n+\frac{1}{2}})$. To solve the system above, one need an efficient preconditioner.  In this paper, we utilize the idea in \cite{GriffithJCP2009} for solving scheme \ref{Scheme:sch-1} and scheme \ref{Scheme:sch-2}. Here we briefly explain how the preconditioner can be constructed.

Recall the decoupled projection scheme for the Navier-Stokes equation in two steps:
\begin{itemize}
\item Step 1, solve the intermediate velocity field via
\beq
2\rho  \frac{\tilde{\bu}^{n+\frac{1}{2}}  - \bu^n}{\delta t}   + \rho B(\overline{\bu}^{n+\frac{1}{2}}, \tilde{\bu}^{n+\frac{1}{2}})=\eta \Delta \tilde{\bu}^{n+\frac{1}{2}} + f^{n+\frac{1}{2}},
\eeq 
where $\tilde{\bu}^{n+\frac{1}{2}} = \frac{1}{2}( \tilde{\bu}^{n+1}+\bu^n)$.
\item Step 2, solve the velocity field via the projection
\beq
\bea{rcl}
&& 2\rho \frac{ \bu^{n+\frac{1}{2}}- \tilde{\bu}^{n+\frac{1}{2}}}{\delta t} = - \nabla \psi, \\
&&  \nabla \cdot \bu^{n+1} = 0.
\eea
\eeq 
\end{itemize}
We can further rewrite Step 1 as 
\beq
\left[
\bea{ll}
\frac{2\rho}{\delta t}-\eta \Delta + \rho B^{n+1} &  0  \\
0   &  \bI
\eea 
\right]
\left(
\bea{ll}
\tilde{\bu}^{n+\frac{1}{2}} \\
0
\eea
\right)
=
\left[
\bea{ll}
\frac{2\rho}{\delta t} + \rho B^n &  0  \\
0   &  0
\eea 
\right]
\left(
\bea{ll}
\bu^n \\
0
\eea
\right)
+
\left[
\bea{l}
f^{n+\frac{1}{2}} \\
0
\eea 
\right],
\eeq 
For Step 2, it can be decomposed into
\begin{subequations} 
\begin{align} 
&-\Delta \psi = - \frac{2\rho}{\delta t} \nabla \cdot  \tilde{\bu}^{n+\frac{1}{2}}, \quad \\
& \bu^{n+\frac{1}{2}} = \tilde{\bu}^{n+\frac{1}{2}} - \frac{\delta t}{2\rho} \nabla \psi, \quad \\
& \nabla p^{n+\frac{1}{2}} = \nabla \psi - \frac{\delta t}{\rho} \frac{\eta}{2} \Delta \nabla \psi. 
\end{align}
\end{subequations} 
Then the corresponding operator forms are given as
\beq
\left[
\bea{ll}
\bI &  0  \\
0 &  -\Delta 
\eea 
\right]
\left(
\bea{ll}
\tilde{\bu}^{n+\frac{1}{2}} \\
\psi
\eea
\right)
=
\left[
\bea{ll}
\bI &  0  \\
-\frac{2\rho}{\delta t}\nabla \cdot    &  -\frac{\rho}{\delta t}
\eea 
\right]
\left(
\bea{ll}
\tilde{\bu}^{n+\frac{1}{2}} \\
0
\eea
\right),
\eeq

\beq
\left(
\bea{ll}
\bu^{n+1} \\
p^{n+\frac{1}{2}}
\eea
\right)
=
\left[
\bea{ll}
\bI &  -\frac{\delta t}{ 2\rho} \nabla   \\
0  &  \bI - \frac{\delta t}{\rho} \frac{\eta}{2} \Delta
\eea 
\right]
\left(
\bea{ll}
\tilde{\bu}^{n+1} \\
\psi
\eea
\right).
\eeq

Therefore, the two-step procedure can be written in an operator form as 
\beq
\mathcal{A} 
\left(
\bea{ll}
\bu^{n+\frac{1}{2}} \\
p^{n+\frac{1}{2}}
\eea
\right)
=\left[
\bea{ll}
\frac{2\rho}{\delta t} + \rho B^{n} & 0  \\
0  &  0
\eea 
\right]
\left(
\bea{ll}
\bu^n \\
0
\eea
\right) +
\left[
\bea{l}
f^{n+\frac{1}{2}} \\
0
\eea 
\right].
\eeq 
with the inverse of the linear operator given by
\beq \label{eq:A-inverse}
\mathcal{A}^{-1} = 
\left[
\bea{ll}
\bI &  -\frac{\delta t}{2 \rho} \nabla   \\
0  &  \bI - \frac{\delta t}{\rho}\frac{\eta}{2} \Delta 
\eea 
\right]
\left[
\bea{ll}
\bI &  0  \\
0   &   -\Delta ^{-1}
\eea 
\right]
\left[
\bea{ll}
\bI &  0  \\
-\frac{2\rho}{\delta t}\nabla \cdot    &  -\frac{\rho}{\delta t}
\eea 
\right] 
\left[
\bea{ll}
\Big( \frac{2\rho}{\delta t}-\eta \Delta + \rho B^{n+1} \Big)^{-1} &  0  \\
0   &  \bI
\eea 
\right].
\eeq

In other words, the operator in \eqref{eq:A-inverse} is a good preconditioner. Since the inverse of $B$ is non-trivial, we use the following operator 
\beq
\left[
\bea{ll}
\bI &  -\frac{\delta t}{2 \rho} \nabla   \\
0  &  \bI - \frac{\delta t}{\rho}\frac{\eta}{2} \Delta 
\eea 
\right]
\left[
\bea{ll}
\bI &  0  \\
0   &   -\Delta ^{-1}
\eea 
\right]
\left[
\bea{ll}
\bI &  0  \\
-\frac{2\rho}{\delta t}\nabla \cdot    &  -\frac{\rho}{\delta t}
\eea 
\right] 
\left[
\bea{ll}
\Big( \frac{2\rho}{\delta t}-\eta \Delta  \Big)^{-1} &  0  \\
0   &  \bI
\eea 
\right].
\eeq 
as the preconditioner for solving the Navier-Stokes equations in Step 1 and Step 3 of schemes \ref{Scheme:sch-1} and \ref{Scheme:sch-2}. And it turns out to be effective numerically.

\subsubsection{Velocity projection to decouple the Navier-Stokes equation}
In our second approach, we can further decouple the velocity and pressure fields in Step 1 and Step 3, using the classical velocity projection method. And the energy stability still holds. Specifically, for the CN scheme of the NS equation, we can instead, introduce the projection strategy, by solving several Poisson-type equation sequentially. The corresponding energy law can also be derived. The equation \eqref{eq:NS-part} can be approximated by the the following two steps with second-order accuracy.
\begin{itemize}
\item Step 1: solve the intermediate velocity field
\beq
\rho \frac{\tilde{\bu}^{n+1} - \bu^n}{\delta t} + \rho B(\overline{\bu}^{n+\frac{1}{2}}, \tilde{\bu}^{n+\frac{1}{2}})  =  - \nabla p^n + \eta \Delta \tilde{\bu}^{n+\frac{1}{2}} +f^{n+\frac{1}{2}}. \\
\eeq 
\item Step 2: update the velocity field
\beq
\bea{l}
\frac{\rho}{\delta t} (\bu^{n+1} - \tilde{\bu}^{n+1}) = -\frac{1}{2} \nabla (p^{n+1} - p^n). \\
\nabla \cdot \bu^{n+1} = 0.
\eea 
\eeq 
\end{itemize}

This leads to a new variant of second-order decoupled numerical schemes for the CHNS system in \eqref{eq:CHNS}-\eqref{eq:CHNS-boundary}.

\begin{scheme} \label{Scheme:sch-3}
Given $(\bu^n, \phi^n)$ and $(\bu^{n-1}, \phi^{n-1})$, we can obtain $(\bu^{n+1}, \phi^{n+1})$ in the following three steps:
\begin{itemize}
\item Step 1: in $[t_n, t_{n+\frac{1}{2}}]$, we set $\phi_\star = \phi^n$ and solve $\bu_{\star}$ via 

\begin{itemize}
\item Step 1.1, get $\tilde{\bu}_\star$ from
\begin{subequations}   \label{eq:sch3-s1X}
\begin{align}
&\label{eq:sch3-s1-1} \rho \frac{\tilde{\bu}_\star - \bu^n}{\delta t/2}  + B(\bar{\bu}^{n+\frac{1}{4}}, \frac{\bu_\star + \bu^n}{2})= -\nabla p^n + \eta \Delta \frac{\tilde{\bu}_\star + \bu^n}{2}, \\
&\label{eq:sch3-s1-2} \tilde{\bu}_\star =0, \quad \mbox{ on } \partial \Omega.
\end{align}
\end{subequations}
\item Step 1.2, get $p_\star$ from
\begin{subequations}    \label{eq:sch3-s1Y}
\begin{align}
& \label{eq:sch3-s1-3} \Delta p_\star = \frac{4\rho}{\delta t} \nabla \cdot \tilde{\bu}_\star + \Delta p^n, \\
& \label{eq:sch3-s1-4} \nabla p_\star \cdot \bn = 0, \mbox{ on } \partial \Omega.
\end{align}
\end{subequations}
\item Step 1.3, get $\bu_\star$  from
\beq \label{eq:sch3-s1Z}
\bu_\star = \tilde{\bu}_\star - \frac{\delta t}{4\rho}\nabla (p_\star - p^n).
\eeq 
\end{itemize}

\item Step 2, in $[t_n, t_{n+1}]$, we solve $(\bu_{\star\star}, \phi_{\star\star})$ via two sub-steps.
\begin{itemize}
\item Step 2.1, solve $(\phi_{\star \star}, p)$ via the following system
\begin{subequations} \label{eq:sch3-s2X}
\begin{align}
&\frac{\phi_{\star \star} - \phi_\star}{\delta t}+ \nabla \cdot ( \bu_\star \overline{\phi}^{n+\frac{1}{2}}) = \\
& \nonumber \nabla \cdot ((M(\overline{\phi}^{n+\frac{1}{2}}) + \frac{\delta t (\overline{\phi}^{n+\frac{1}{2}})^2}{2 \rho}  )  \nabla \mu^{n+\frac{1}{2}}) 
+ \nabla \cdot ( \frac{\delta t \overline{\phi}^{n+\frac{1}{2}} }{2\rho}\nabla p), \\
&\mu^{n+\frac{1}{2}} = -\frac{\gamma \varepsilon}{2}(\Delta \phi_{\star\star} + \Delta \phi_{\star}) +\frac{\gamma}{\varepsilon} \frac{\delta f}{\delta (\phi_{\star \star}, \phi_{\star})},\\
&-\Delta p - \nabla \cdot ( \overline{\phi}^{n+\frac{1}{2}} \nabla \mu^{n+\frac{1}{2}}) = 0, \\
&\nabla \mu^{n+\frac{3}{4}} \cdot \bn =0, \quad \nabla \phi_{\star \star} \cdot \bn = 0,  \quad  \nabla p \cdot \bn = 0, \quad \mbox{ on } \partial \Omega.
\end{align}
\end{subequations}
\item Step 2.2, update $\bu_{\star \star}$ via
\beq  \label{eq:sch3-s2Y}
\bu_{\star \star} = \bu_\star - \frac{\delta t}{\rho} (\nabla p + \overline{\phi}^{n+\frac{1}{2}} \nabla \mu^{n+\frac{1}{2}}).
\eeq
\end{itemize}

\item  Step 3, in $[t_{n+\frac{1}{2}}, t_{n+1}]$, we set $\phi^{n+1} =\phi_{\star \star}$, and solve $\bu^{n+1}$ via the follow three sub-steps.
\begin{itemize}
\item Step 3.1,  get $\tilde{\bu}^{n+1}$ from
\begin{subequations}   \label{eq:sch3-s3X}
\begin{align}
& \label{eq:sch3-s3-1} \rho \frac{ \tilde{\bu}^{n+1}- \bu_{\star \star}}{\delta t/2}  + B(\bar{\bu}^{n+\frac{3}{4}}, \frac{\tilde{\bu}^{n+1} + \bu_{\star\star}}{2})= -\nabla p_\star + \eta \Delta \frac{\bu_{\star \star} + \tilde{\bu}^{n+1}}{2} , \\
& \label{eq:sch3-s3-2} \tilde{\bu}^{n+1}= 0\quad \mbox{ on } \partial \Omega.
\end{align}
\end{subequations}
\item Step 3.2, get $p^{n+1}$ from
\begin{subequations}   \label{eq:sch3-s3Y}
\begin{align}
& \label{eq:sch3-s3-3} \Delta p^{n+1} = \frac{4\rho}{\delta t} \nabla \cdot \tilde{\bu}^{n+1} + \Delta p_\star, \\
& \label{eq:sch3-s3-4} \nabla p^{n+1} \cdot \bn = 0, \mbox{ on } \partial \Omega.
\end{align}
\end{subequations}
\item Step 3.3, get $\bu^{n+1}$ from
\beq \label{eq:sch3-s3Z}
\bu^{n+1} = \tilde{\bu}^{n+1} - \frac{\delta t}{4 \rho}\nabla (p^{n+1} - p_\star).
\eeq 
\end{itemize}
\end{itemize}
\end{scheme}

\begin{rem}
Note the decoupled steps in \eqref{eq:sch3-s1-3}-\eqref{eq:sch3-s1Z} are from the general scheme
\beq \label{eq:sch3-s1-New}
\frac{\rho}{\delta t}(\bu_\star -  \tilde{\bu}_\star) = - \frac{1}{2} (\nabla p_\star  -\nabla p^n). 
\eeq 
Similarly, the steps in \eqref{eq:sch3-s3-3}-\eqref{eq:sch3-s3Z} are from the general scheme
\beq \label{eq:sch3-s3-New}
\frac{\rho}{\delta t}( \bu^{n+1} - \tilde{\bu}_{\star \star}) = - \frac{1}{2} (\nabla p^{n+1}  -\nabla p_\star). 
\eeq
\end{rem}


\begin{rem}
We further comment on Step 2. Based on the operator splitting strategy, in $[t_n, t_{n+1}]$, we solve $(\bu_{\star\star}, \phi_{\star\star})$ via
\begin{subequations}   \label{eq:sch3-s2}
\begin{align}
& \label{eq:sch3-s2-1} \rho \frac{\bu_{\star \star}  - \bu_\star}{\delta t} = -\nabla p - \bar{\phi}^{n+\frac{1}{2}} \nabla \mu^{n+\frac{1}{2}}, \\
& \label{eq:sch3-s2-2} \mu^{n+\frac{1}{2}} = -\frac{\gamma \varepsilon}{2} (\Delta \phi_{\star\star} + \Delta \phi_{\star}) + \frac{\gamma}{\varepsilon} \frac{\delta f}{\delta (\phi_{\star \star}, \phi_{\star})},\\
& \label{eq:sch3-s2-3} \nabla \cdot \frac{\bu_{\star \star} + \bu_\star}{2}= 0, \\
& \label{eq:sch3-s2-4}\frac{\phi_{\star \star} - \phi_\star }{\delta t} = -\nabla \cdot ( \frac{ \bu_{\star \star} + \bu_\star}{2} \bar{\phi}^{n+\frac{1}{2}}) + \nabla \cdot ( M(\overline{\phi}^{n+\frac{1}{2}}) \nabla  \mu^{n+\frac{1}{2}}), \\
& \label{eq:sch3-s2-5} \nabla \mu^{n+\frac{1}{2}} \cdot \bn =0, \quad \nabla \phi_{\star\star} \cdot \bn = 0, \quad \bu_{\star\star} =0, \quad \mbox{ on } \partial \Omega.
\end{align}
\end{subequations}
Decoupling the velocity field from the phase variables, we can simplify \eqref{eq:sch3-s2} into  the equations \eqref{eq:sch3-s2X}-\eqref{eq:sch3-s2Y}.
\end{rem}

Besides, we still have the energy stability of this scheme as below.

\begin{thm}
The decoupled scheme \ref{Scheme:sch-3} is energy stable in the sense that
\begin{multline}
\hat{\cE}(\bu^{n+1}, p^{n+1},  \phi^{n+1}) - \hat{\cE}(\bu^n ,  p^n, \phi^n) \leq  \\
- \frac{\delta t}{2}  \Big[ \eta \| \nabla \frac{\bu_\star + \bu^n}{2}\|^2 +\eta  \| \nabla \frac{ \bu^{n+1} + \bu_{\star \star}}{2}\|^2  + 2 \|\sqrt{M(\overline{\phi}^{n+\frac{1}{2}})} \nabla \mu^{n+\frac{1}{2}} \|^2 \Big], 
\end{multline}
where the modified energy is defined as
\beq
\hat{\cE}(\bu, p, \phi) = E(\bu) + \tilde{E}(\nabla p ) + F(\phi), \\
\eeq
with the three terms specified as
\beq
E(\bu) = \int_\Omega \frac{\rho}{2} | \bu |^2 d\bx,    \quad
\tilde{E}(\nabla p) = \int_\Omega \frac{\delta t^2}{8\rho } \| \nabla p\|^2 d\bx, \quad 
F(\phi) =\gamma \int_\Omega  \frac{\varepsilon}{2} |\nabla \phi|^2 + \frac{1}{\varepsilon} f(\phi) d\bx.
\eeq
\end{thm}

\begin{proof}
Since the proof of this theorem is a little different from proofs of the other two theorems in previous sections, we will illustrate the details. 
First of all, we take the inner product of \eqref{eq:sch3-s1-1} with $\frac{\delta t}{2} \frac{ \tilde{\bu}_\star + \bu^n}{2}$, we obtain
\beq \label{eq:tmp-st1}
\frac{\rho}{2}( \| \tilde{\bu}_\star \|^2 - \| \bu^n\|^2) =- \frac{\delta t}{2} (\nabla p^n, \frac{ \tilde{\bu}_\star + \bu^n}{2}) - \frac{\delta t}{2} \eta \| \nabla \frac{ \tilde{\bu}_\star + \bu^n}{2} \|^2.
\eeq 

If we take the inner product of \eqref{eq:sch3-s1-New} with $\delta t \bu_\star$, we have
\beq \label{eq:tmp-st2}
\frac{\rho}{2} (\| \bu_\star \|^2 - \| \tilde{\bu}_\star\|^2 + \| \bu_\star - \tilde{\bu}_\star\|^2 ) = 0.
\eeq 

We can rewrite \eqref{eq:sch3-s1-New} as 
\beq
\frac{\rho}{2\delta t} ( \bu_\star + \bu^n - 2 \frac{ \bu_\star + \tilde{\bu}_\star}{2}) + \frac{1}{2} (\nabla p_\star - \nabla p^n ) = 0.
\eeq 
and test it with $\frac{\delta t^2 }{4}\nabla p^n$. It yields
\beq \label{eq:tmp-st3}
- \frac{\delta t}{2} (\frac{ \tilde{\bu}_\star + \bu^n}{2}, \nabla p^n)  + \frac{\delta t^2}{8\rho } (\| \nabla p_\star \|^2 - \| \nabla p^n \|^2 - \| \nabla (p_\star - p^n) \|^2 ) = 0.
\eeq 

Taking the inner product of \eqref{eq:sch3-s1-New} with itself, we can have
\beq \label{eq:tmp-st4}
\frac{\delta t^2}{8\rho } \| \nabla (p_\star - p^n) \|^2 - \frac{\rho }{2} \| \bu_\star - \tilde{\bu}_\star\|^2 = 0.
\eeq 

Adding the equations \eqref{eq:tmp-st1},\eqref{eq:tmp-st2}, \eqref{eq:tmp-st3} and \eqref{eq:tmp-st4} together will give us
\beq \label{eq:sch3-E1}
(\frac{\rho}{2} \| \bu_\star\|^2 + \frac{\delta t^2}{\rho } \| \nabla p_\star\|^2 ) - (\frac{\rho}{2} \| \bu^n \|^2 + \frac{\delta t^2}{8\rho } \| \nabla p^n \|^2) = - \frac{\delta t}{2} \eta \| \nabla \frac{ \tilde{\bu}_\star + \bu^n}{2} \|^2.
\eeq 

Take inner product of \eqref{eq:sch3-s2-1} with $\delta t \frac{ \bu_{\star \star} + \bu_\star}{2}$, we have
\beq \label{eq:sch3-step2-A}
\frac{\rho}{2}( \| \bu_{\star \star} \| ^2 - \| \bu_\star\|^2) = -\delta t  (\frac{ \bu_{\star \star} + \bu_\star}{2} \overline{\phi}^{n+\frac{1}{2}} \nabla \mu^{n+\frac{1}{2}}).
\eeq 

If we take inner product of \eqref{eq:sch3-s2-3} with $\delta t \mu^{n+\frac{1}{2}}$,  and \eqref{eq:sch3-s2-4} with $\delta t (\phi_{\star \star} - \phi_\star)$, we have
\beq \label{eq:sch3-step2-B}
F(\phi_{\star \star}) - F(\phi_\star) = -\delta t (\nabla \cdot ( \frac{\bu_{\star \star} + \bu_{\star}}{2} \overline{\phi}^{n+\frac{1}{2}}), \mu^{n+\frac{1}{2}}) - \delta t  \|\sqrt{M(\overline{\phi}^{n+\frac{1}{2}})} \nabla \mu^{n+\frac{1}{2}} \|^2.
\eeq 
Adding the equations \eqref{eq:sch3-step2-A} and \eqref{eq:sch3-step2-B} above, we have
\beq  \label{eq:sch3-E2}
\frac{\rho}{2}( \| \bu_{\star \star} \| ^2 - \| \bu_\star\|^2)  + F(\phi_{\star \star}) - F(\phi_\star)  = - \delta t  \| \sqrt{M(\overline{\phi}^{n+\frac{1}{2}})} \nabla \mu^{n+\frac{1}{2}} \|^2.
\eeq 

Similarly, from \eqref{eq:sch3-s3-1}-\eqref{eq:sch3-s3-4}, we will obtain
\beq \label{eq:sch3-E3}
(\frac{\rho}{2} \| \bu^{n+1}]\|^2 + \frac{\delta t^2}{8} \| \nabla p^{n+1}]\|^2 ) - (\frac{\rho}{2} \| \bu_{\star \star} \|^2 + \frac{\delta t^2}{8} \| \nabla p_\star  \|^2) = - \frac{\delta t}{2} \eta \| \nabla \frac{ \tilde{\bu}^{n+1} + \bu_{\star \star}}{2} \|^2.
\eeq 

Adding the equations \eqref{eq:sch3-E1}, \eqref{eq:sch3-E2} and \eqref{eq:sch3-E3}, it lead to
\begin{multline}
(\frac{\rho}{2} \| \bu^{n+1}\|^2 + \frac{\delta t^2}{8\rho } \| \nabla p^{n+1}]\|^2  + F(\phi^{n+1})) - (\frac{\rho}{2} \| \bu^n \|^2 + \frac{\delta t^2}{8\rho} \| \nabla p^n \|^2 +F(\phi^n))  \\
= - \frac{\delta t}{2}  \Big[  \eta \| \nabla \frac{ \tilde{\bu}^{n+1} + \bu_{\star \star}}{2} \|^2 + \eta \| \nabla \frac{ \tilde{\bu}_{\star \star} + \bu^n}{2} \|^2  + \| \sqrt{M(\overline{\phi}^{n+\frac{1}{2}})} \nabla \mu^{n+\frac{1}{2}} \|^2 \Big].
\end{multline}
\end{proof}

\begin{rem}
We note that the modified free energy $\hat{\cE}$ is a second-order perturbation of the original energy. This type of perturbation is preferred than the modified free energy with auxiliary variables, since its errors are explicitly written as $\tilde{E}(\nabla p)$, which is $O(\delta t^2)$.
\end{rem}

\section{Numerical Results} \label{sect:results}

Note that the central finite difference on staggered grids for the spatial discretization and the operator splitting finite difference method in uniform time meshes for the temporal discretization are independent. Thus, the full discrete schemes for the CHNS system can be easily obtained by literately combine the spatial discretization in Section \ref{sect:space} and temporal discretization in Section \ref{sect:time}. Here we won't elaborate on the details for simplicity.

Once the full discrete schemes are obtained, they are implemented.  In this section, we conduct several numerical experiments with the proposed schemes. In particular, the time mesh refinements are presented to demonstrate the second order temporal accuracy of the proposed schemes. And some benchmark examples are calculated to illustrate the effectiveness of the proposed decoupled schemes.

\subsection{Time step refinement tests}
First of all, we verify that the proposed numerical algorithms are second-order in time indeed. Consider a square domain $\Omega =[0, 1]^2$. We fix the parameters
$\rho = 1$,  $\eta = 1.0$, $\lambda =0.01$, $\varepsilon = 0.01$, $\gamma=0.01$. The initial profiles are chosen as
$$
\bu(t=0) = 0, \quad \phi(t=0) = \tanh \frac{ R - \sqrt{ (x-0.5L_x)^2 +2* (y-0.5L_y)^2}}{2\varepsilon}.
$$
We use the meshes $N_x=N_y=128$, and numerical solutions with various time steps are calculated. Since the true solution is unknown, we follow our previous procedure by calculating the errors at $T=0.2$ between two numerical solutions with adjacent time step sizes. It is known the order of errors approximates the order of the temporal accuracy for the numerical schemes. The three proposed numerical schemes \ref{Scheme:sch-1}, \ref{Scheme:sch-2}, \ref{Scheme:sch-3} are tested with the numerical results summarized in Figure \ref{fig:error-1}, \ref{fig:error-2} and \ref{fig:error-3}, respectively. It can be easily seen that all the three proposed schemes reach second-order accuracy in time when the time step is small.

\begin{figure}[H]
\center 
\includegraphics[width=0.9\textwidth]{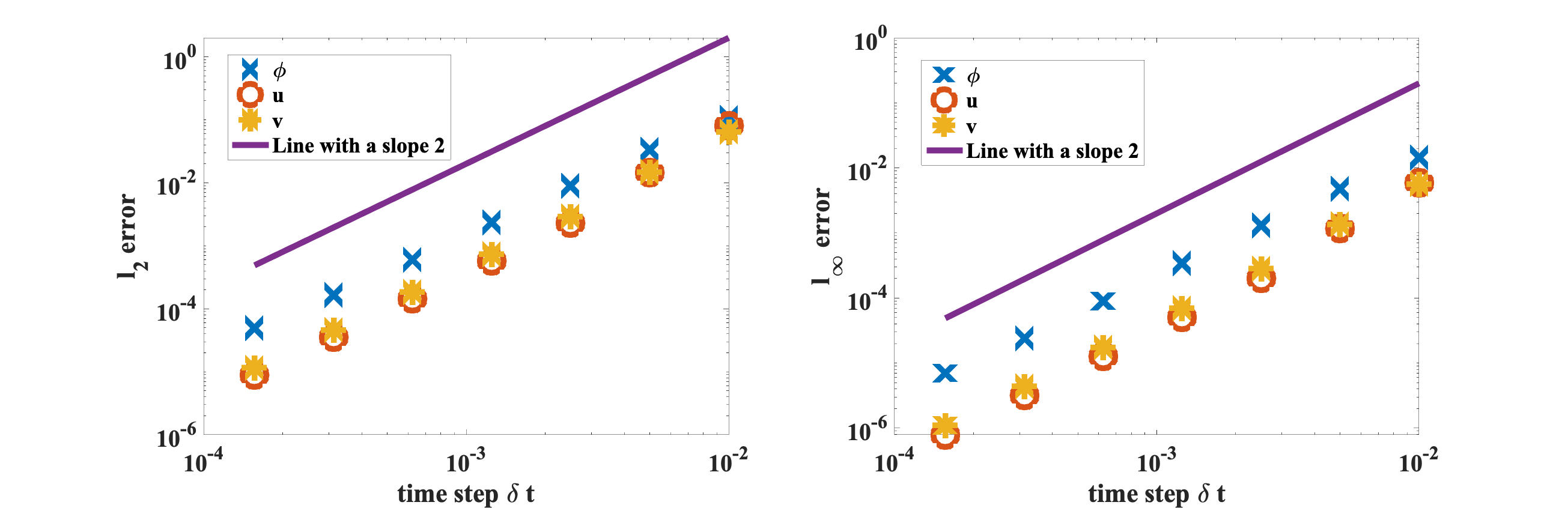}
\caption{Temporal mesh refinement test for scheme \ref{Scheme:sch-1}. Here both the $L^2$ errors and $L^\infty$ errors using different time step sizes are shown. We observe that the scheme \ref{Scheme:sch-1} provides 2nd-order temporal accuracy.}
\label{fig:error-1}
\end{figure}

\begin{figure}[H]
\center
\includegraphics[width=0.9\textwidth]{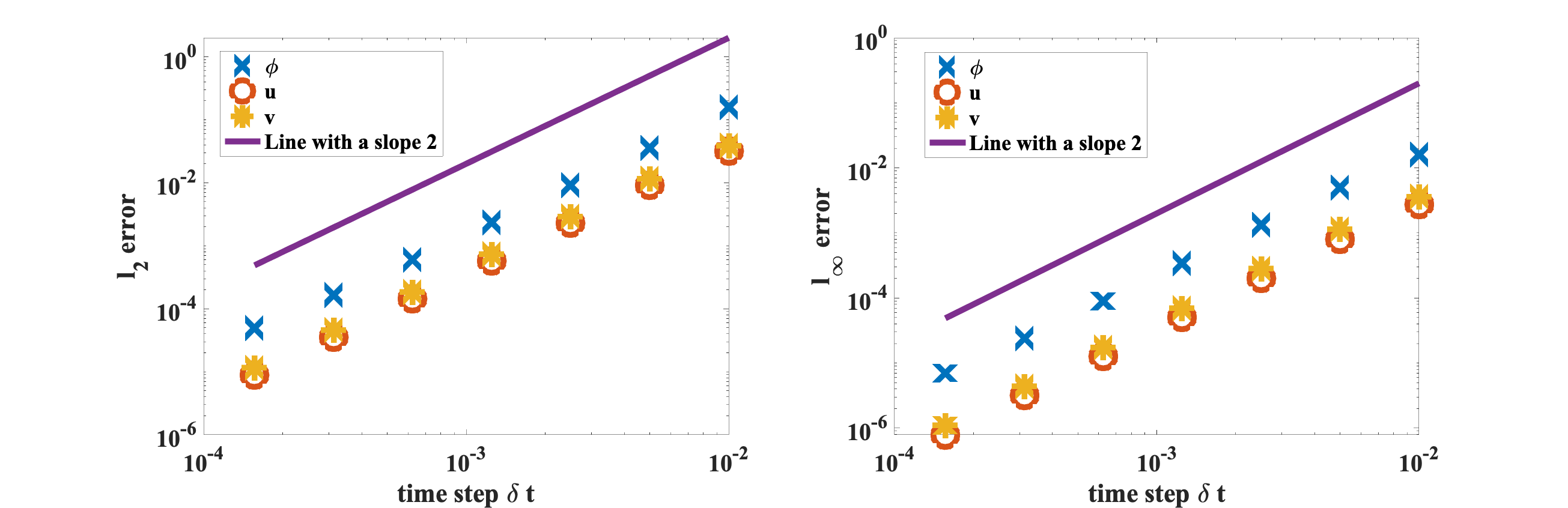}
\caption{Temporal mesh refinement test for scheme \ref{Scheme:sch-2}. Here both the $L^2$ errors and $L^\infty$ errors using different time step sizes are shown. We observe that the scheme \ref{Scheme:sch-2} provides 2nd-order temporal accuracy.}
\label{fig:error-2}
\end{figure}

\begin{figure}[H]
\center
\includegraphics[width=0.9\textwidth]{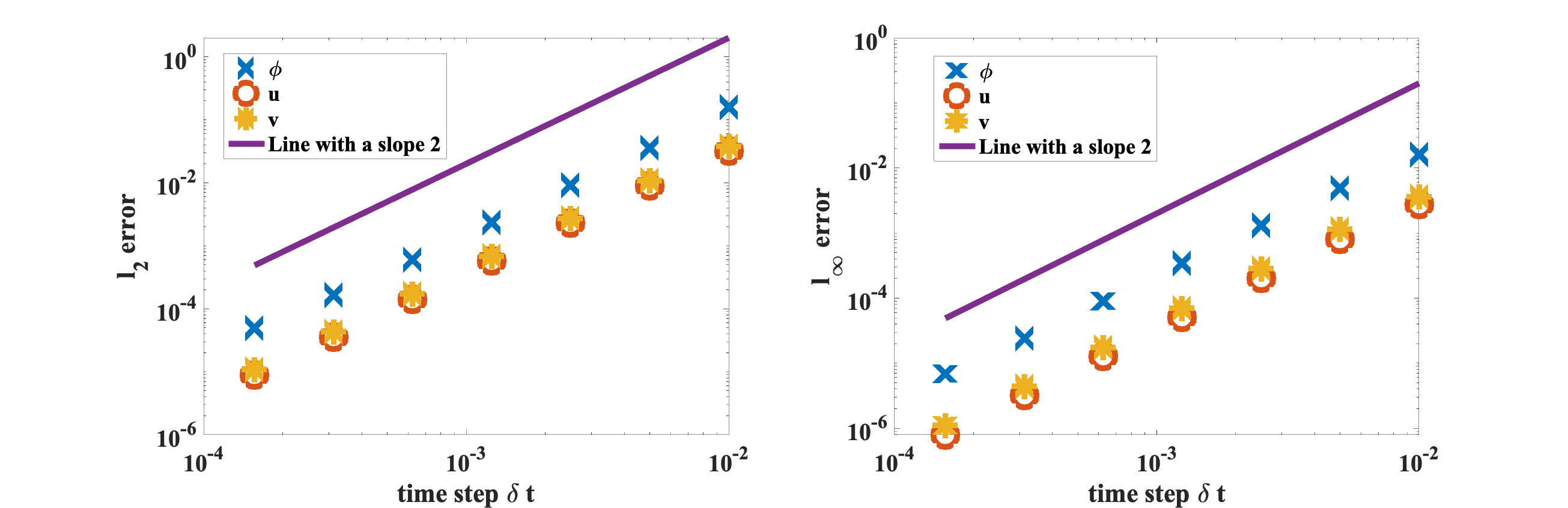}
\caption{Temporal mesh refinement test for scheme \ref{Scheme:sch-3}. Here both the $L^2$ errors and $L^\infty$ errors using different time step sizes are shown. We observe that the scheme \ref{Scheme:sch-3} provides 2nd-order temporal accuracy.}
\label{fig:error-3}
\end{figure}

In addition, it seems that three schemes provide similar numerical results and accuracy, given they are using the same time steps. This is not surprise since three schemes only differ slightly, and the major numerical errors resulted from these three schemes are due to the operator splitting. Acknowledging this fact,  we use scheme \ref{Scheme:sch-3} for all numerical simulations in the rest of this paper.

\subsection{Bubbles merging driven by surface tension}
Next, we conduct several numerical simulations on the merging of the two bubbles that is driven by surface tension. We follow the similar set up as in \cite{Guo&Lin&LowngrubJCP2014}. Consider the domain $\Omega=[0, L_x] \times [0, L_y]$, with $L_x=L_y=1$.  The initial conditions are provided as $\bu(x, y, t=0) = \mathbf{0}$, and
$$
\phi(x, y, t=0) = 1- \tanh \frac{ -r + \sqrt{ (x-x_a)^2 + (y-y_a)^2}}{2\varepsilon}  - \tanh \frac{ -r + \sqrt{ (x-x_b)^2 + (y-y_b)^2}}{2\varepsilon}.
$$
with $x_a = 0.5L_x - \frac{r}{\sqrt{2}}$, $y_a = 0.5L_y + \frac{r}{\sqrt{2}}$,  $x_b = 0.5L_x + \frac{r}{\sqrt{2}}$, $y_b = 0.5L_y - \frac{r}{\sqrt{2}}$, $r=0.15$. The parameters are chosen as $\rho=1$, $\eta = 1$, $\lambda = 0.01$, $\gamma=0.01$, $\varepsilon=0.01$, and we use various viscosity $\eta$, and uniorm meshes $N_x=N_y=128$.

The numerical results with viscosity $\eta = 0.01$ are shown in Figure \ref{fig:ETA2}, where the profiles of the phase variable $\phi$ at various times are shown. It can be observed that the two drops merge into a single drop.

\begin{figure}[H]
\center
\subfigure[$\phi$ at $t=0,0.2,0.4,0.8, 1.0$]{
\includegraphics[width=0.19\textwidth]{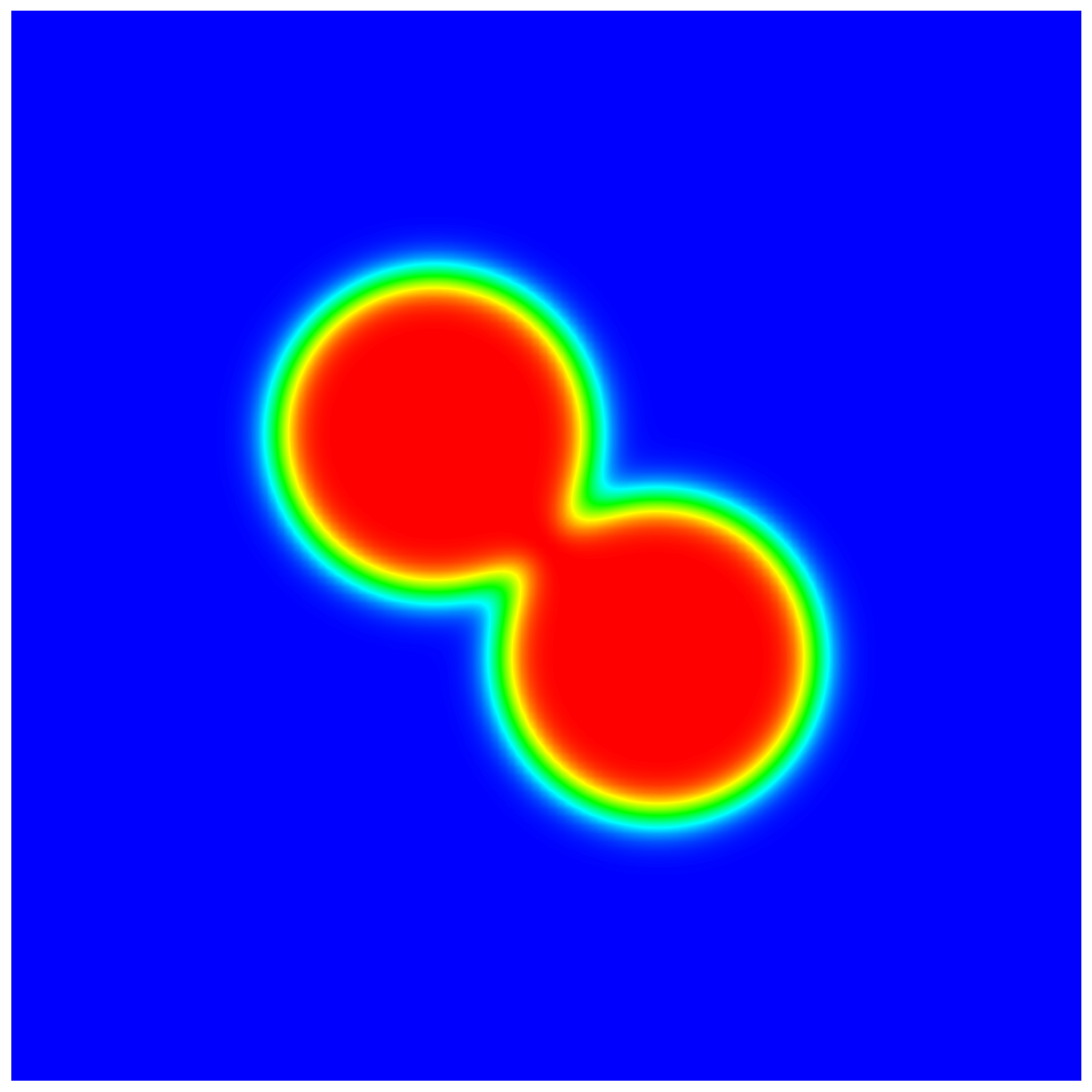}
\includegraphics[width=0.19\textwidth]{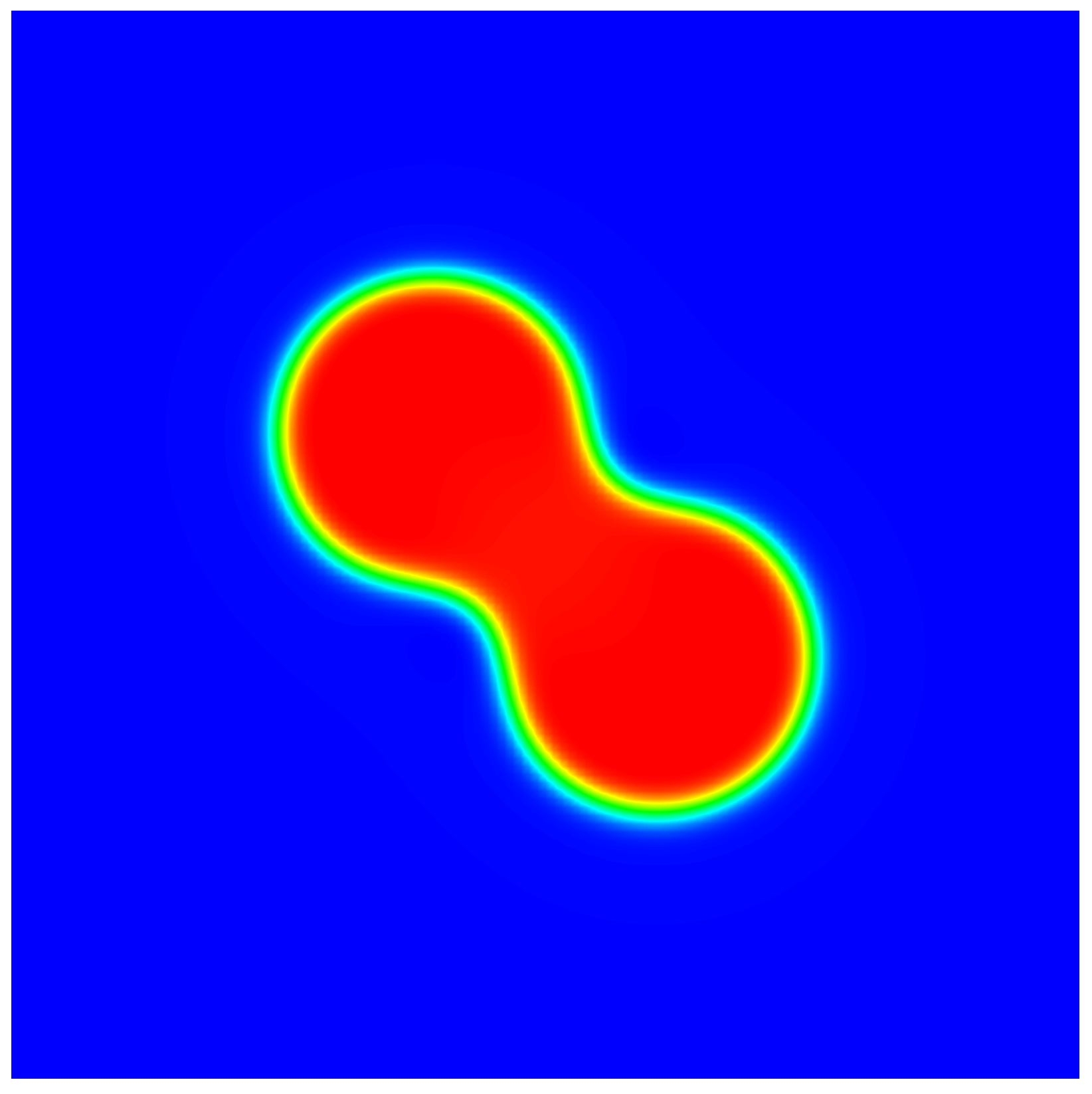}
\includegraphics[width=0.19\textwidth]{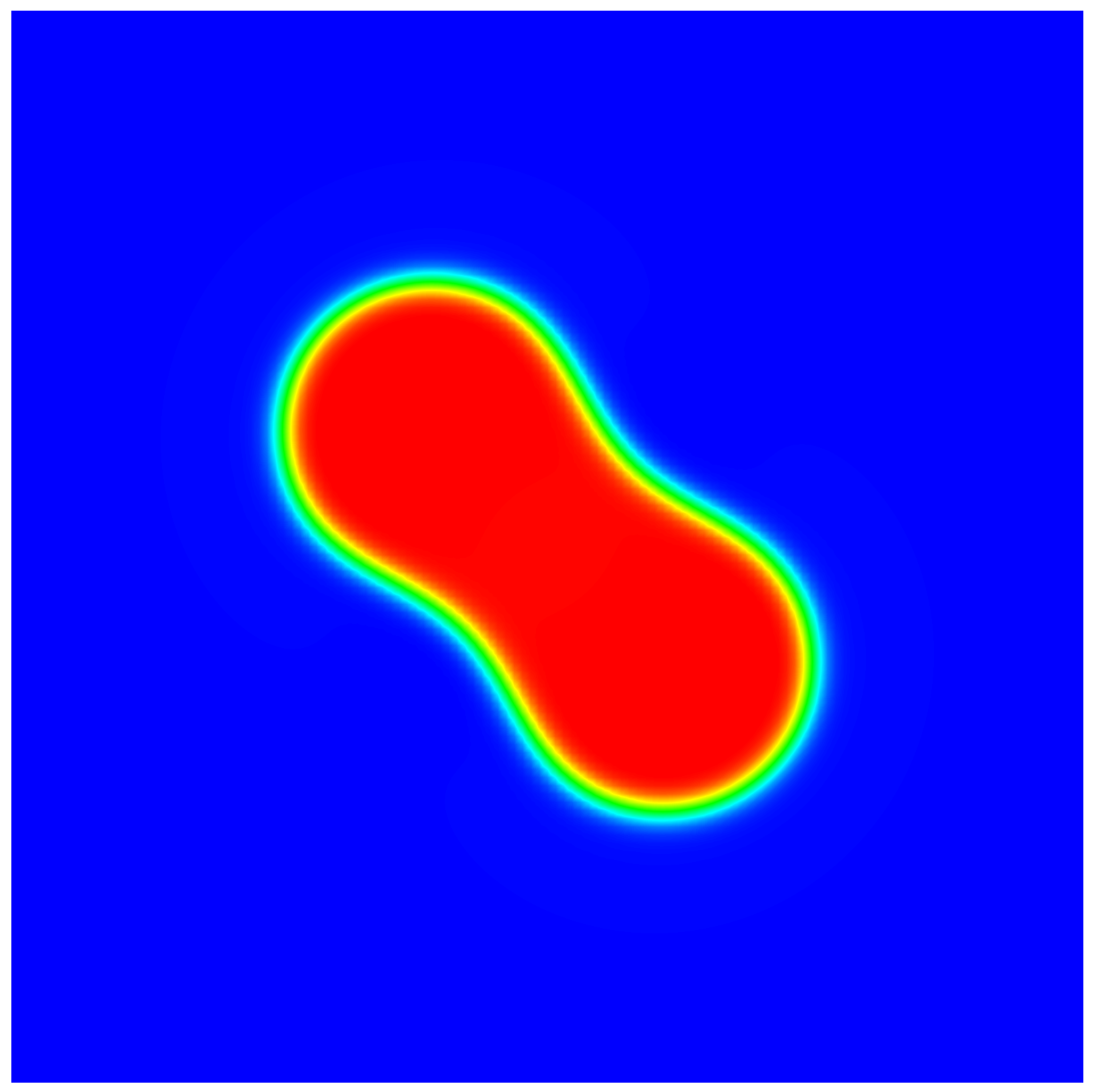}
\includegraphics[width=0.19\textwidth]{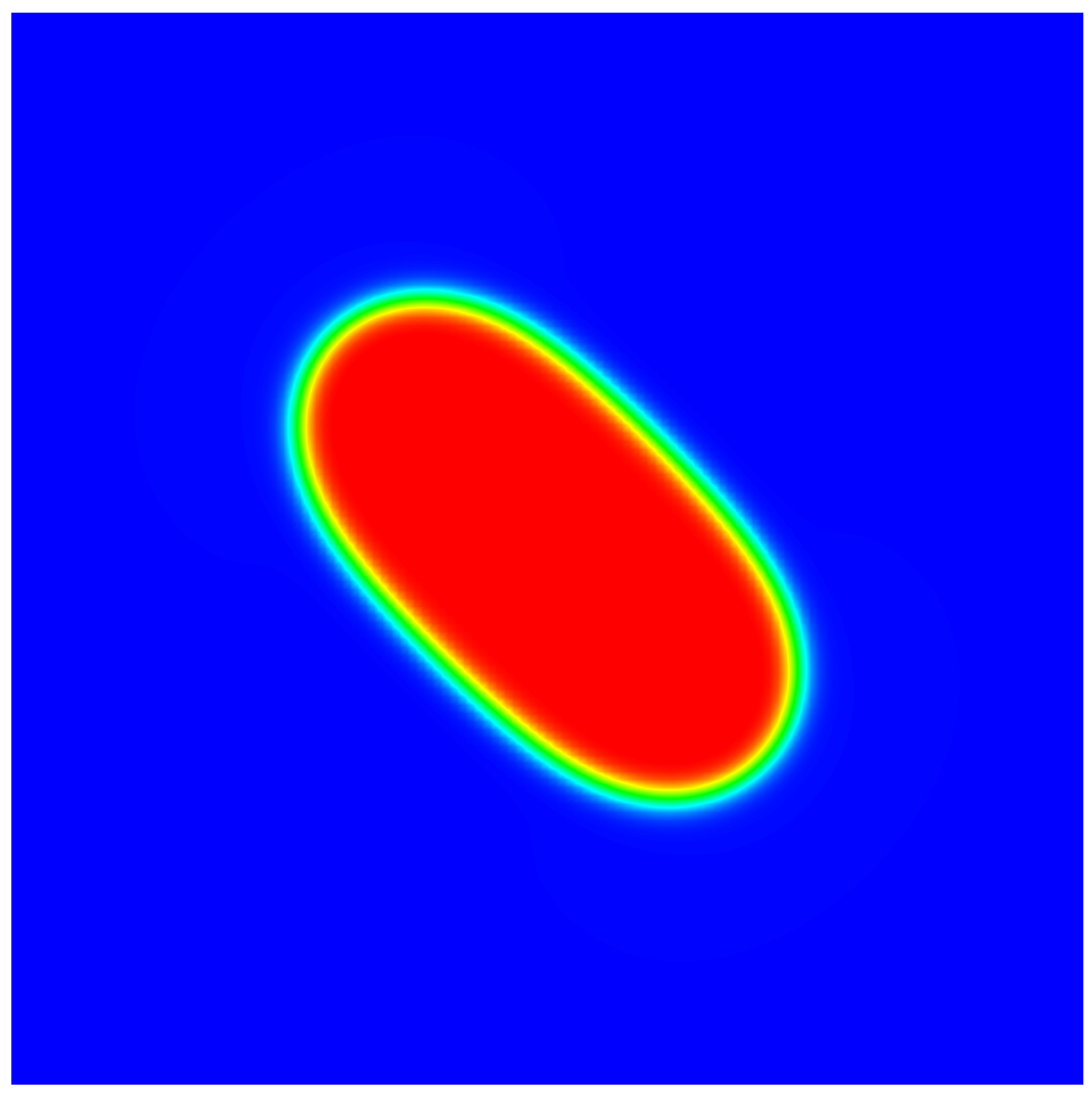}
\includegraphics[width=0.19\textwidth]{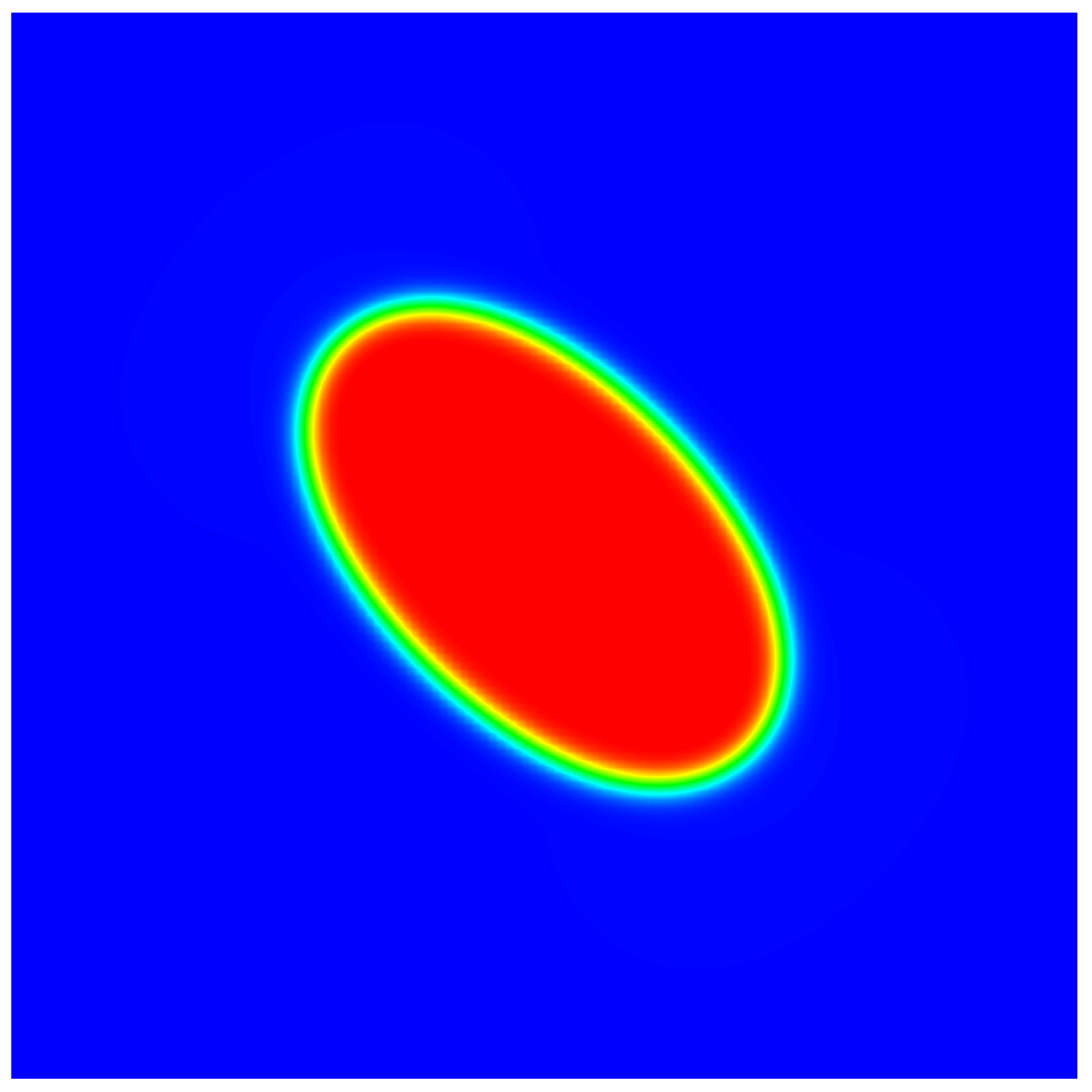}}

\subfigure[$\phi$ at $t=1.2, 1.5, 2, 3, 3.2$]{
\includegraphics[width=0.19\textwidth]{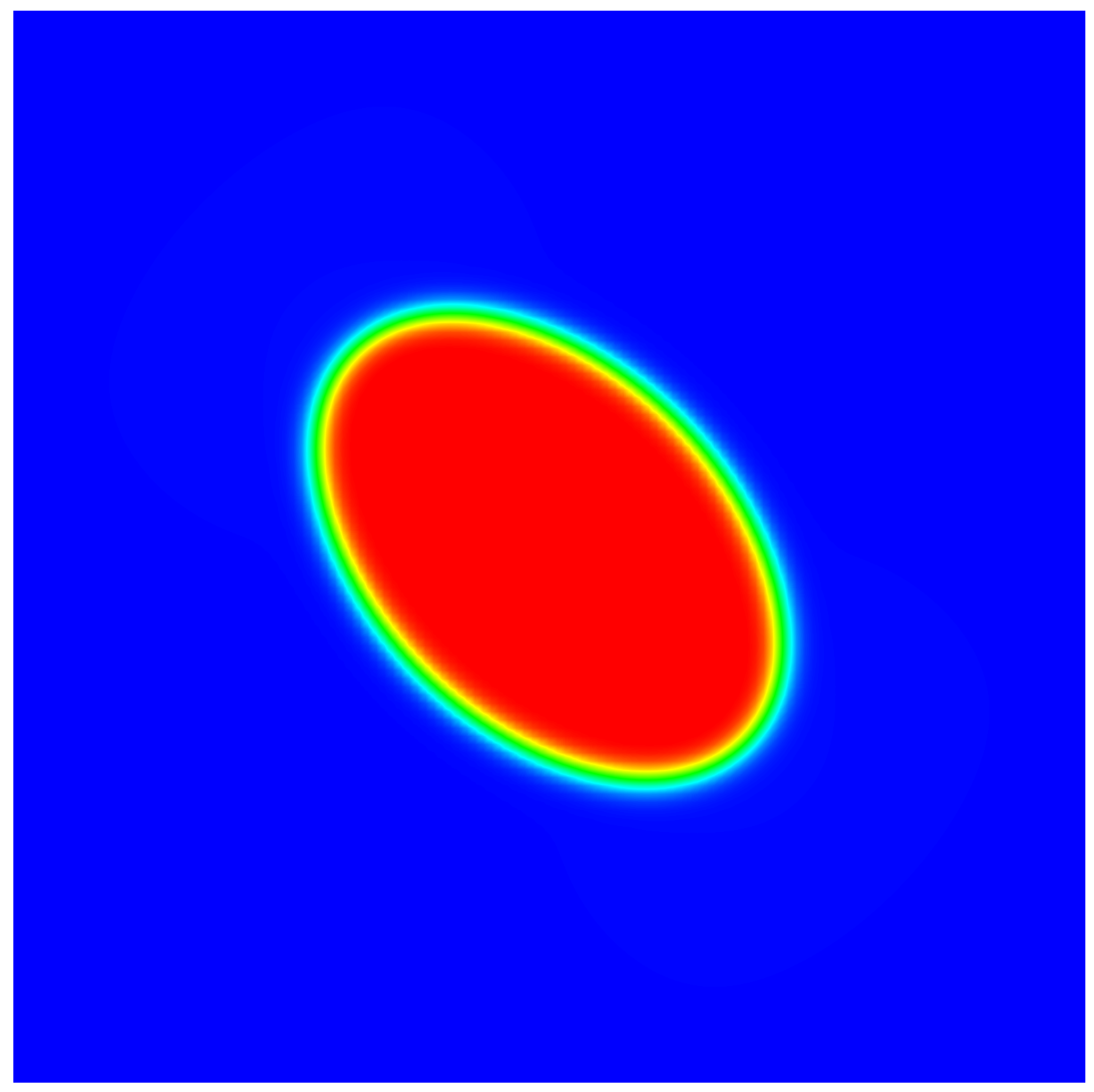}
\includegraphics[width=0.19\textwidth]{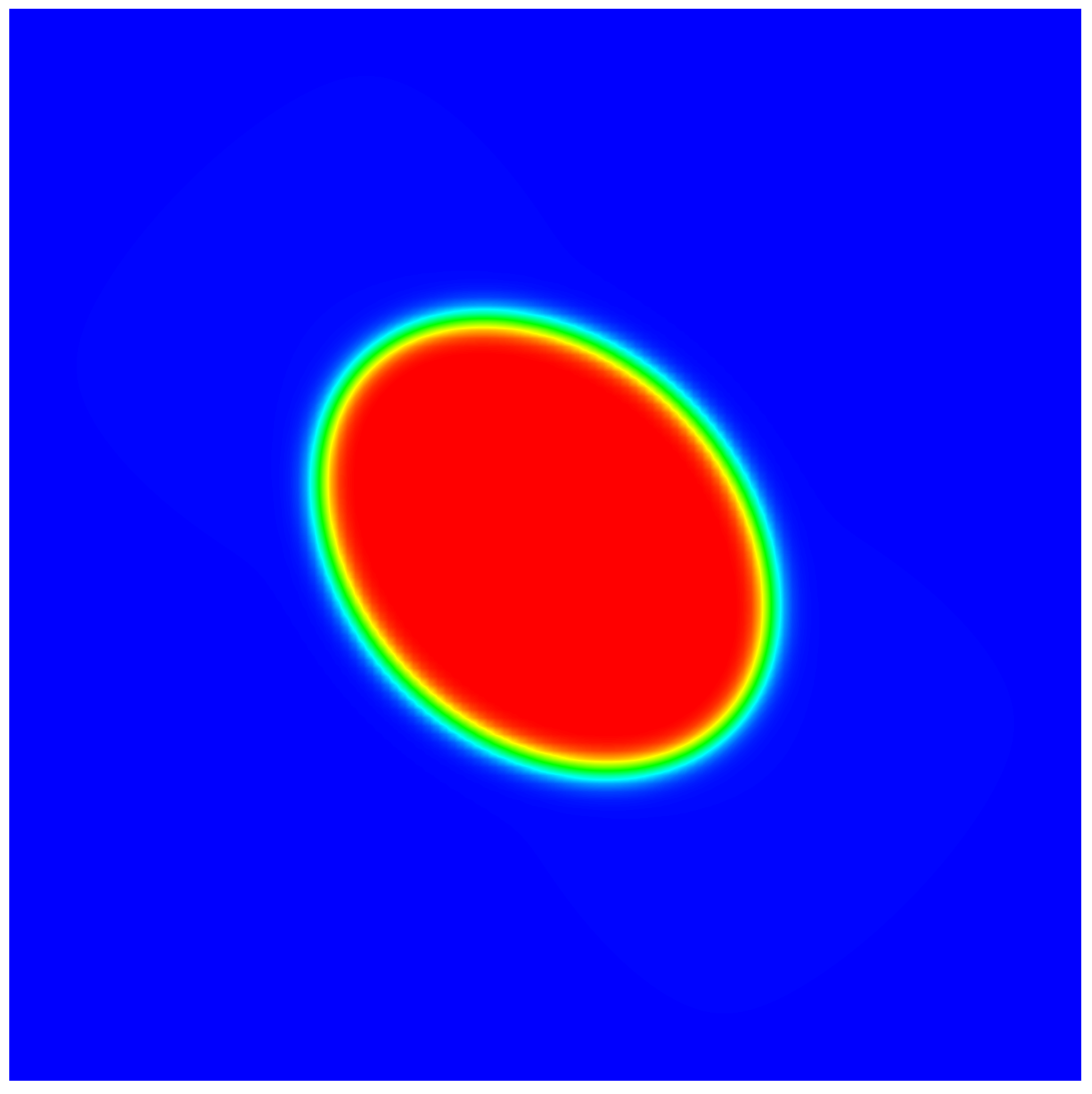}
\includegraphics[width=0.19\textwidth]{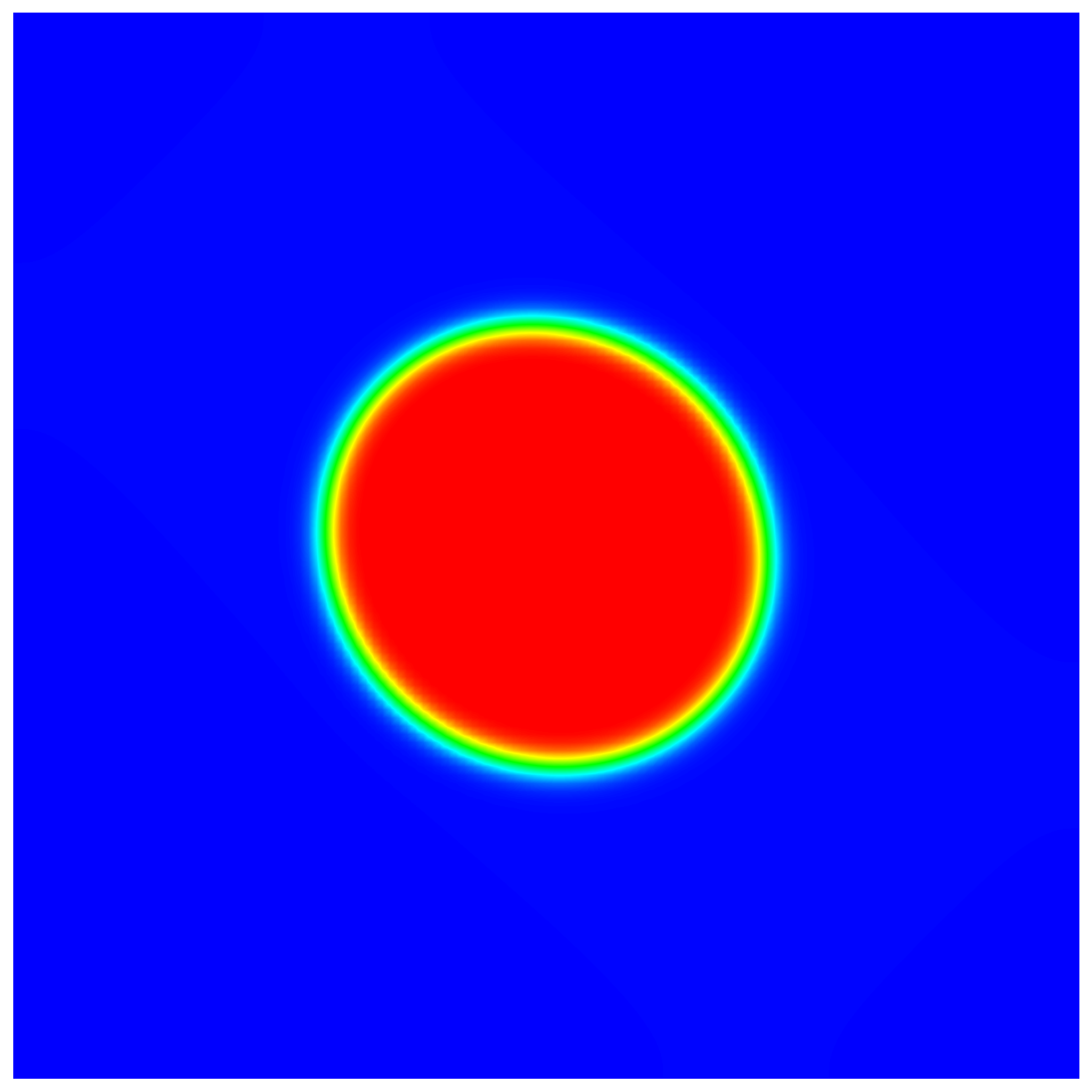}
\includegraphics[width=0.19\textwidth]{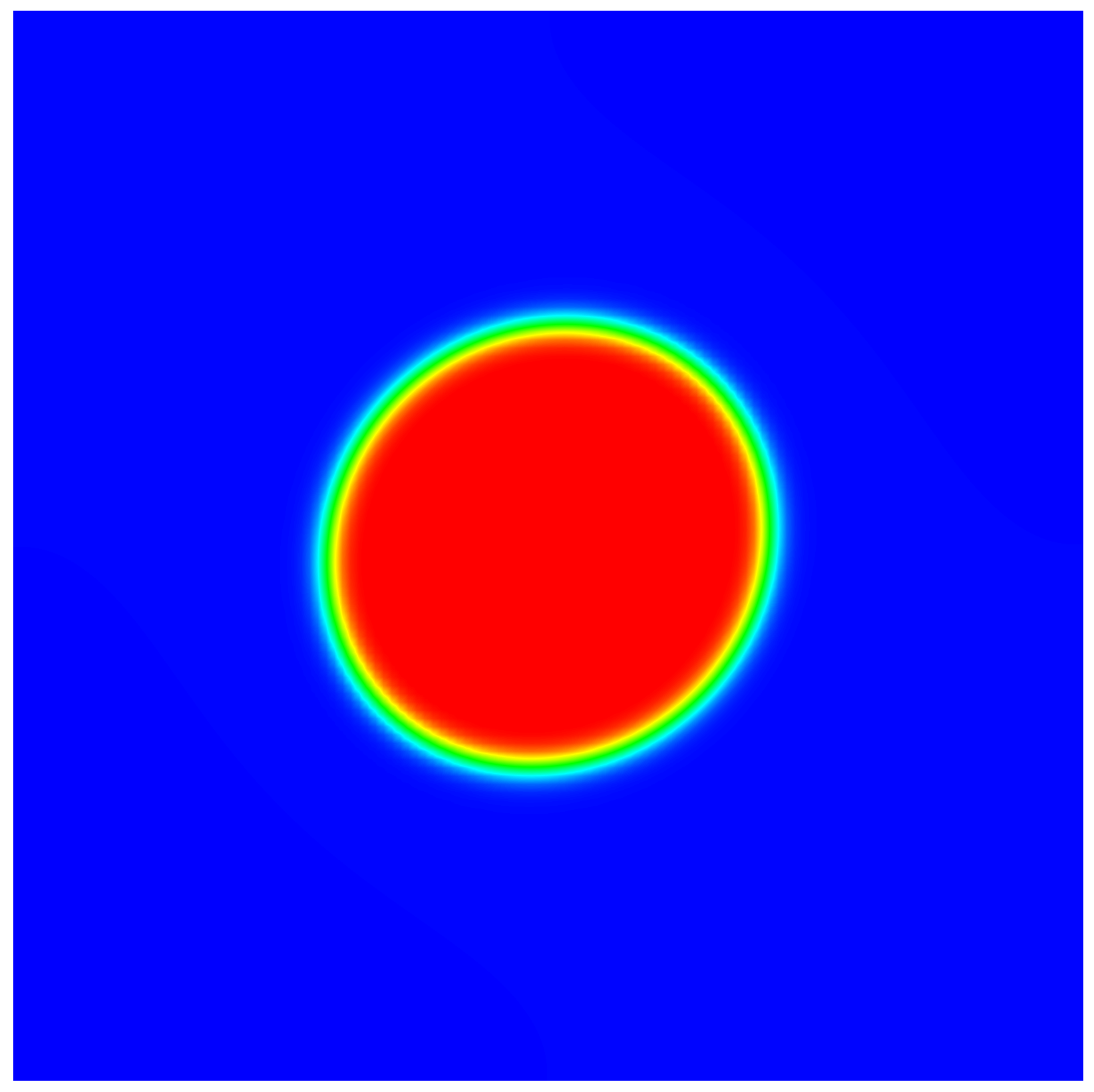}
\includegraphics[width=0.19\textwidth]{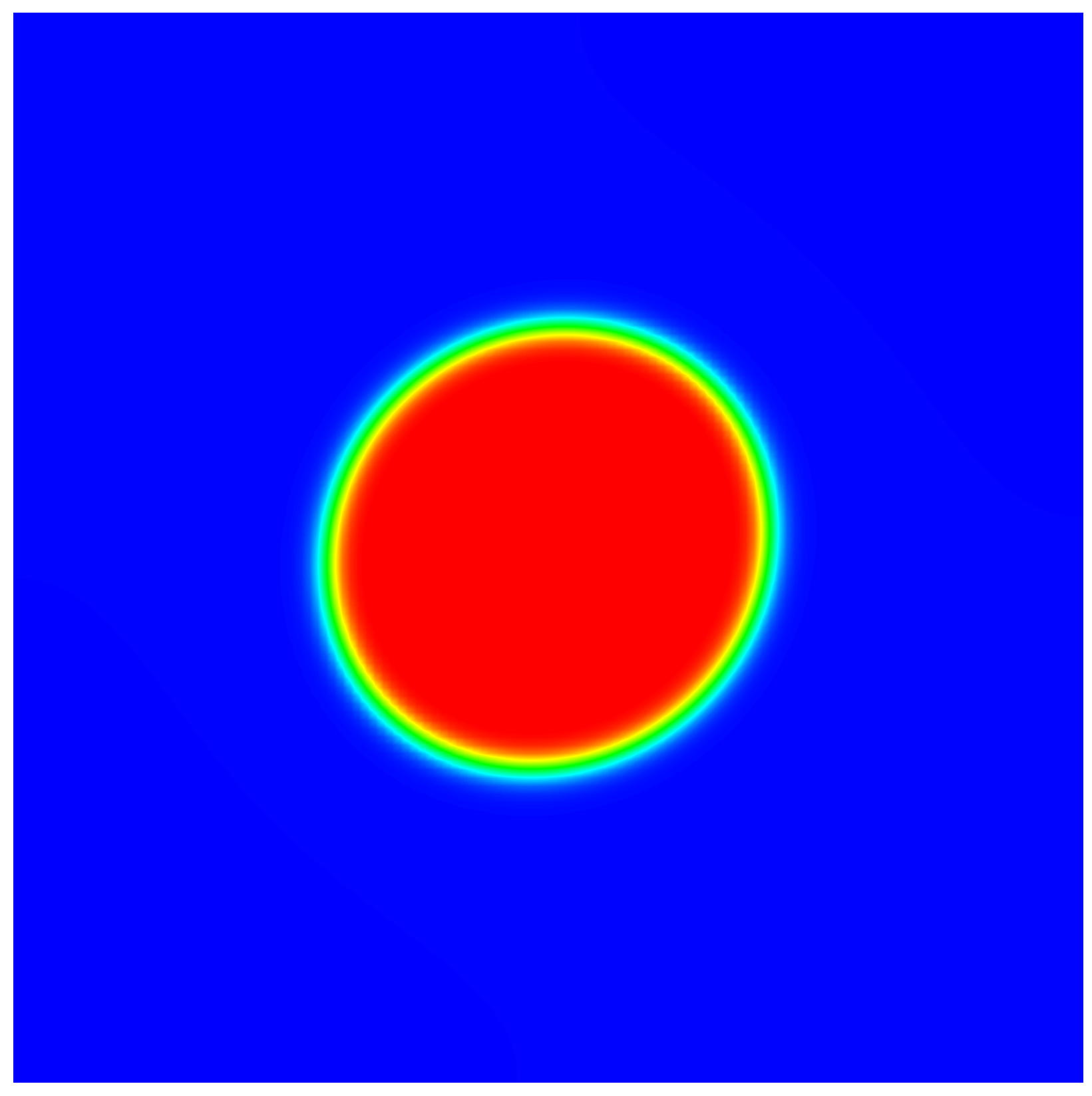}}

\subfigure[$\phi$ at $t=3.4, 4.6, 4.8, 5, 10$]{
\includegraphics[width=0.19\textwidth]{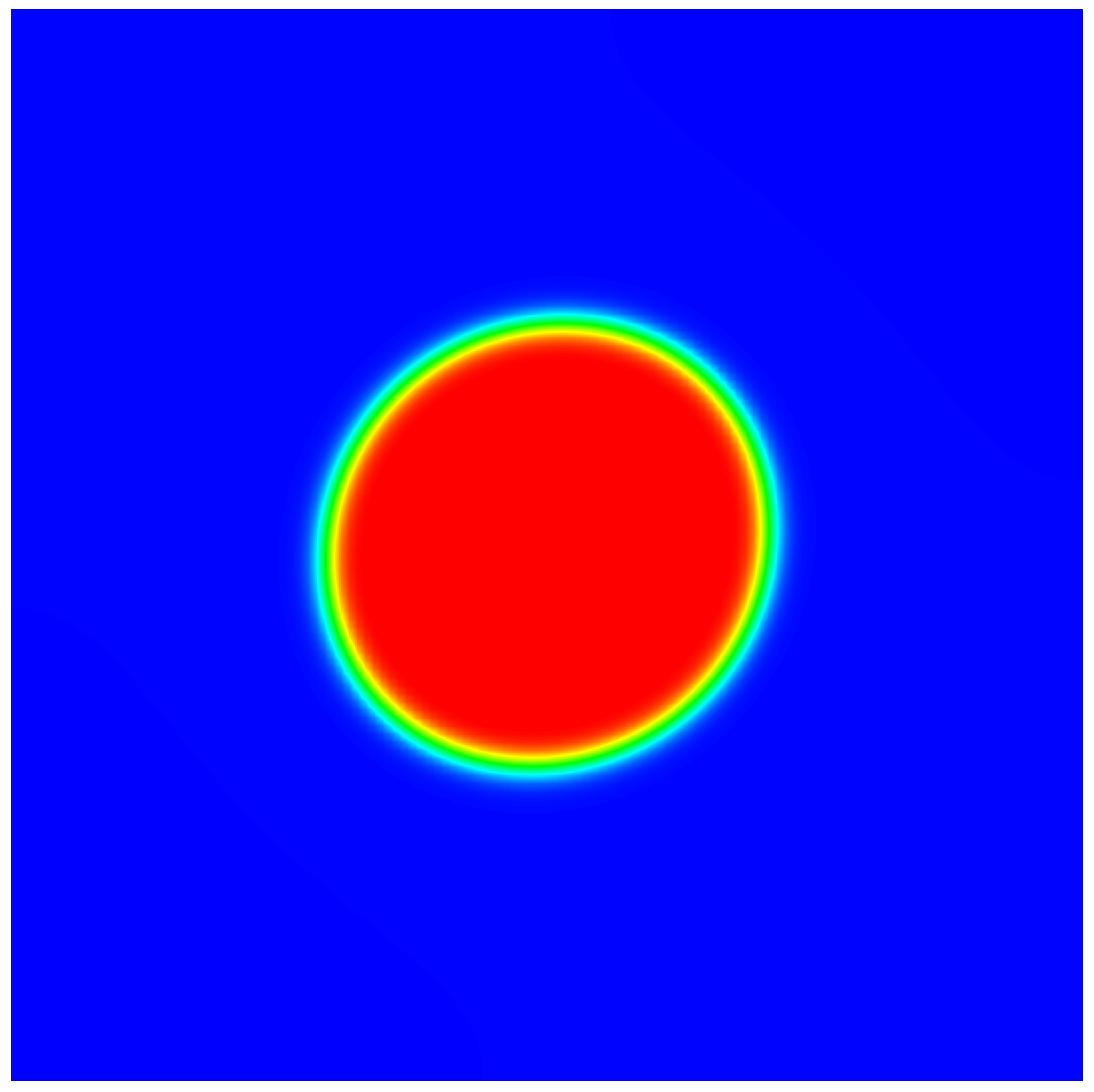}
\includegraphics[width=0.19\textwidth]{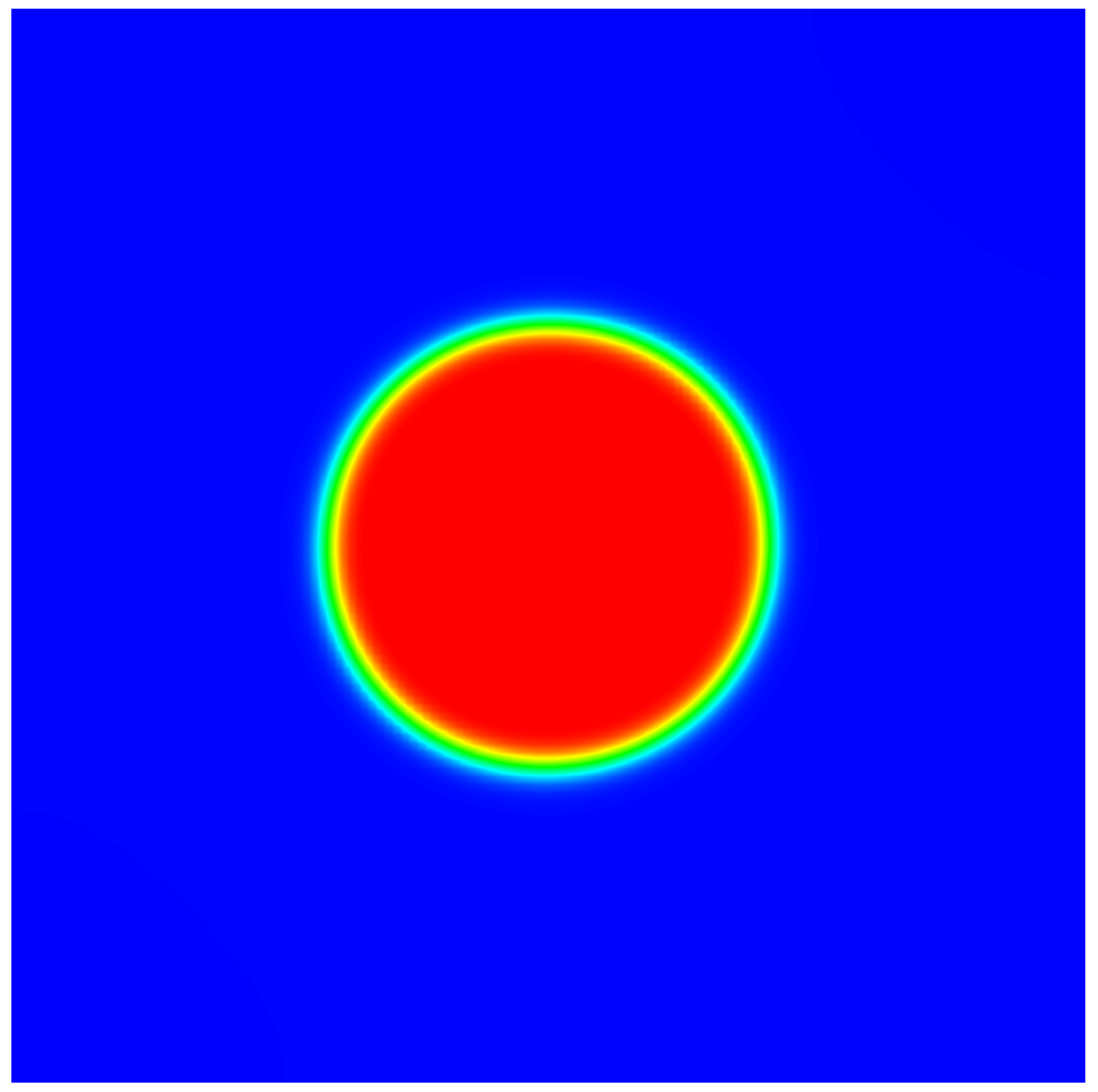}
\includegraphics[width=0.19\textwidth]{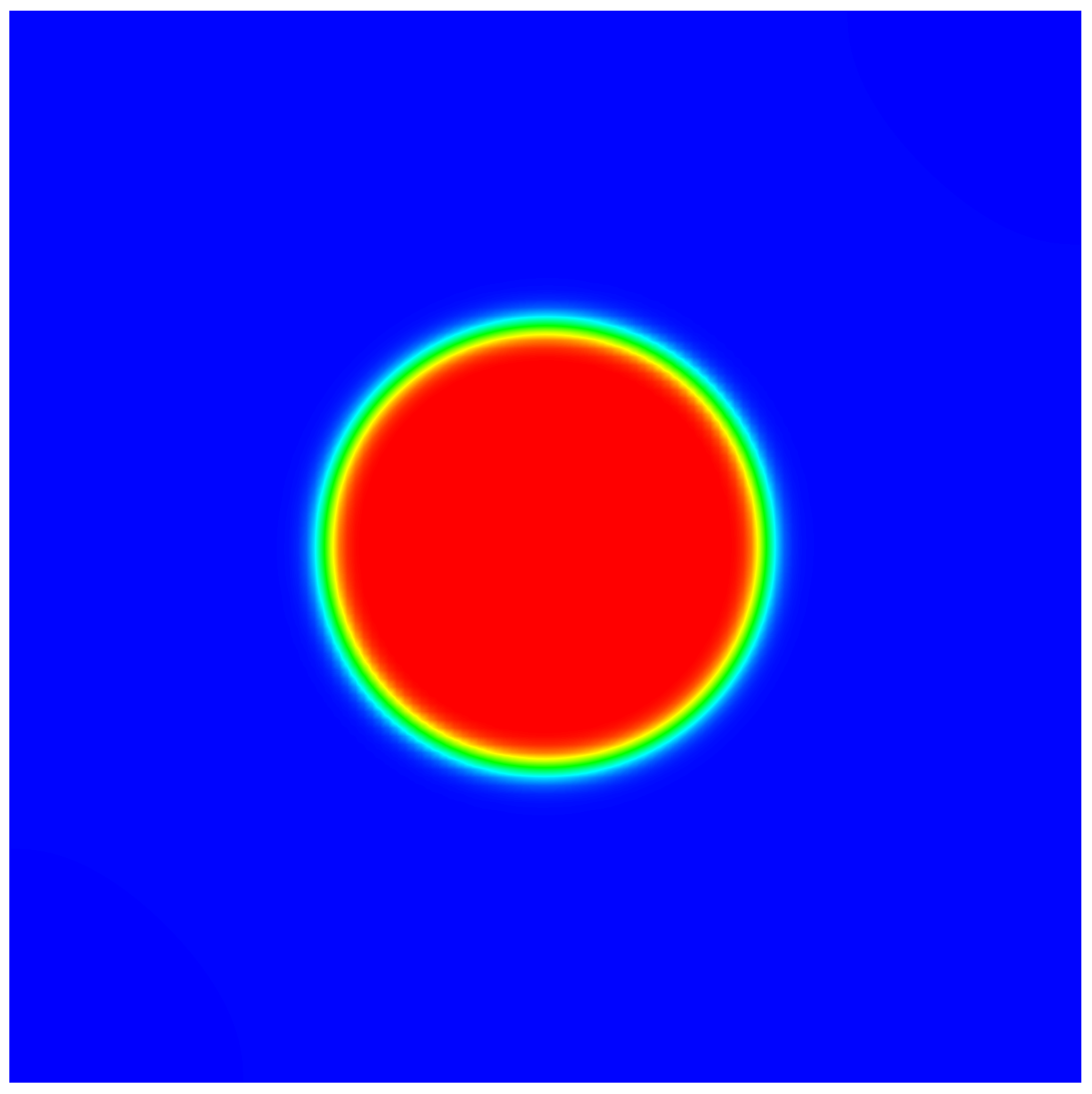}
\includegraphics[width=0.19\textwidth]{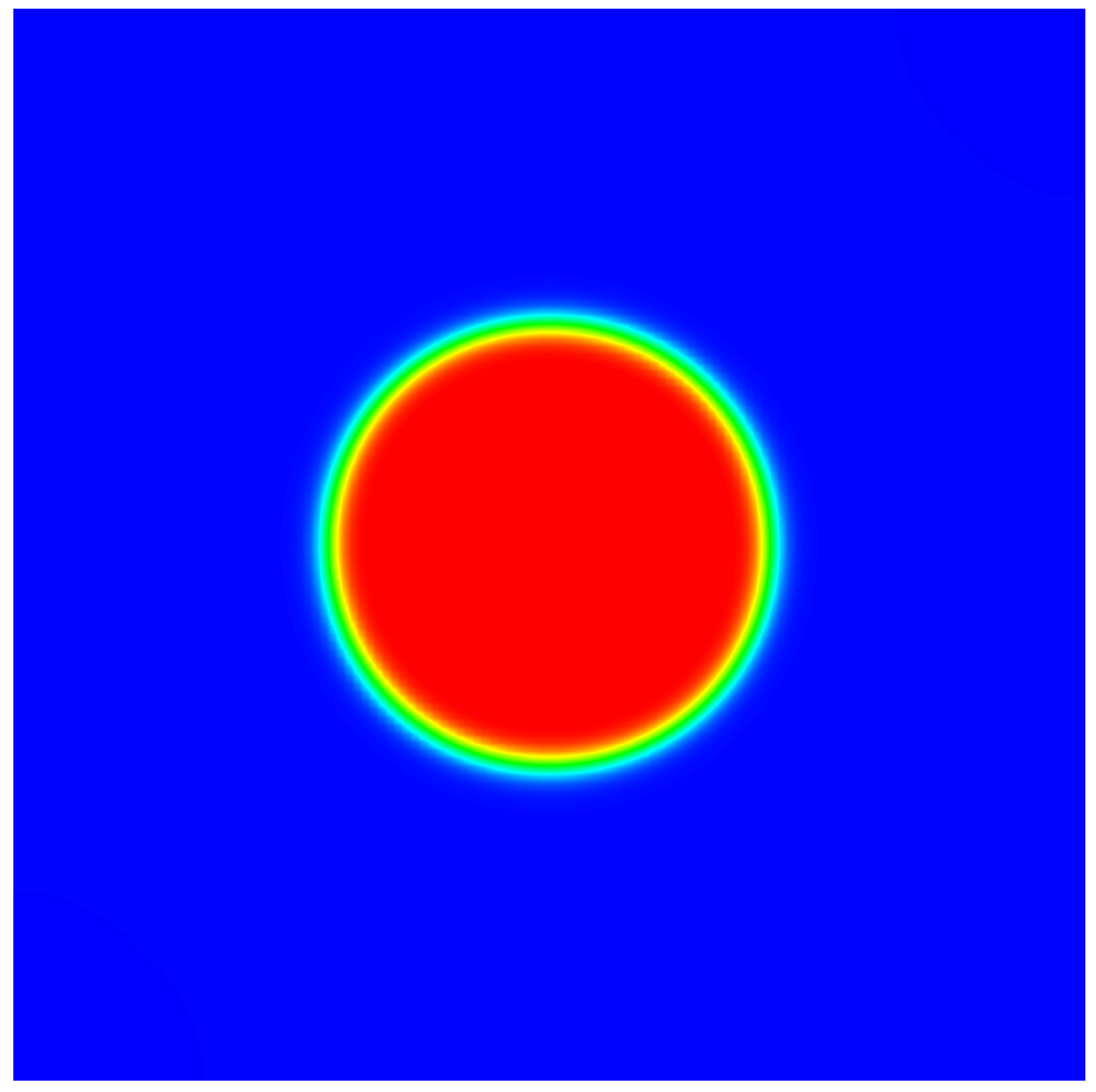}
\includegraphics[width=0.19\textwidth]{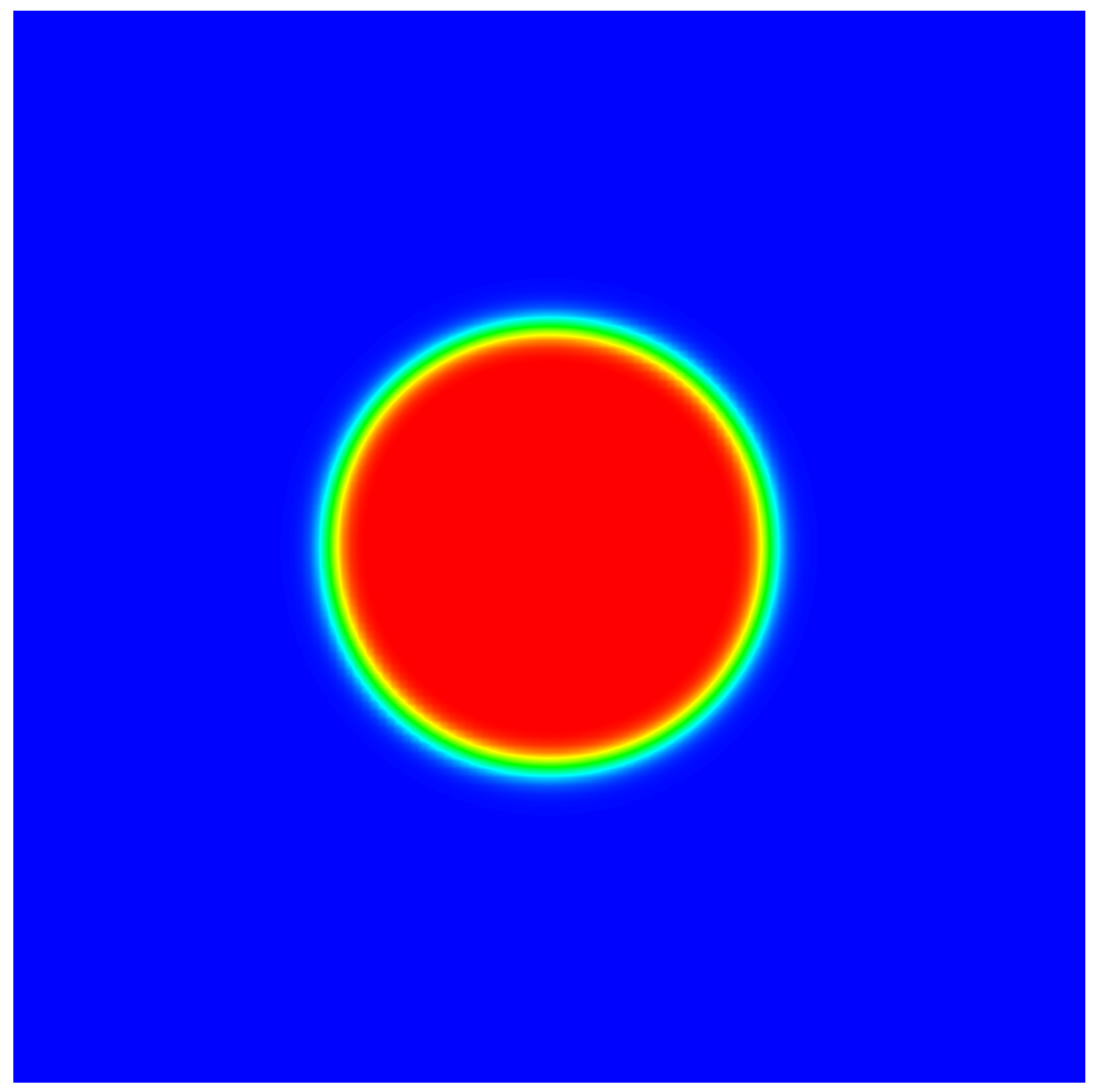}
}
\caption{The dynamics of bubble merging under hydrodynamic environment with viscosity $\eta = 0.01$. In this figure,  the profiles of the phase variable $\phi$ are shown at time $t=0$, $0.2$, $0.4$, $0.8$, $1.0$, $1.2$, $1.5$, $2.0$, $3.0$, $3.2$, $3.4$, $4.6$, $4.8$, $5$ and $10$.} 
\label{fig:ETA2}
\end{figure}

As a comparison, another simulation with the same parameters and other settings, except a smaller viscosity $\eta =0.001$. The profiles of the phase variables at the same times as the previous example are summarized in Figure \ref{fig:ETA3}. Though the two bubbles eventually merged into a single one and the round drop reaches a steady state, the dynamics between Figure \ref{fig:ETA2} and \ref{fig:ETA3} are dramatically different.  Mainly, when the viscosity is small, the effect of inertia is not negligible anymore. In  Figure \ref{fig:ETA3}, we observe that the drop squeezed and then stretched,  showing back-and-forth damping oscillations (due to the exchanges between the kinetic energy and the Helmholtz free energy). And it eventually stabilizes as a round drop. In other words, the hydrodynamics or the Navier-Stokes equation shall not be ignored when the inertia has noticeable effects on the dynamics.

\begin{figure}[H]
\center
\subfigure[$\phi$ at $t=0,0.2,0.4,0.8, 1.0$]{
\includegraphics[width=0.19\textwidth]{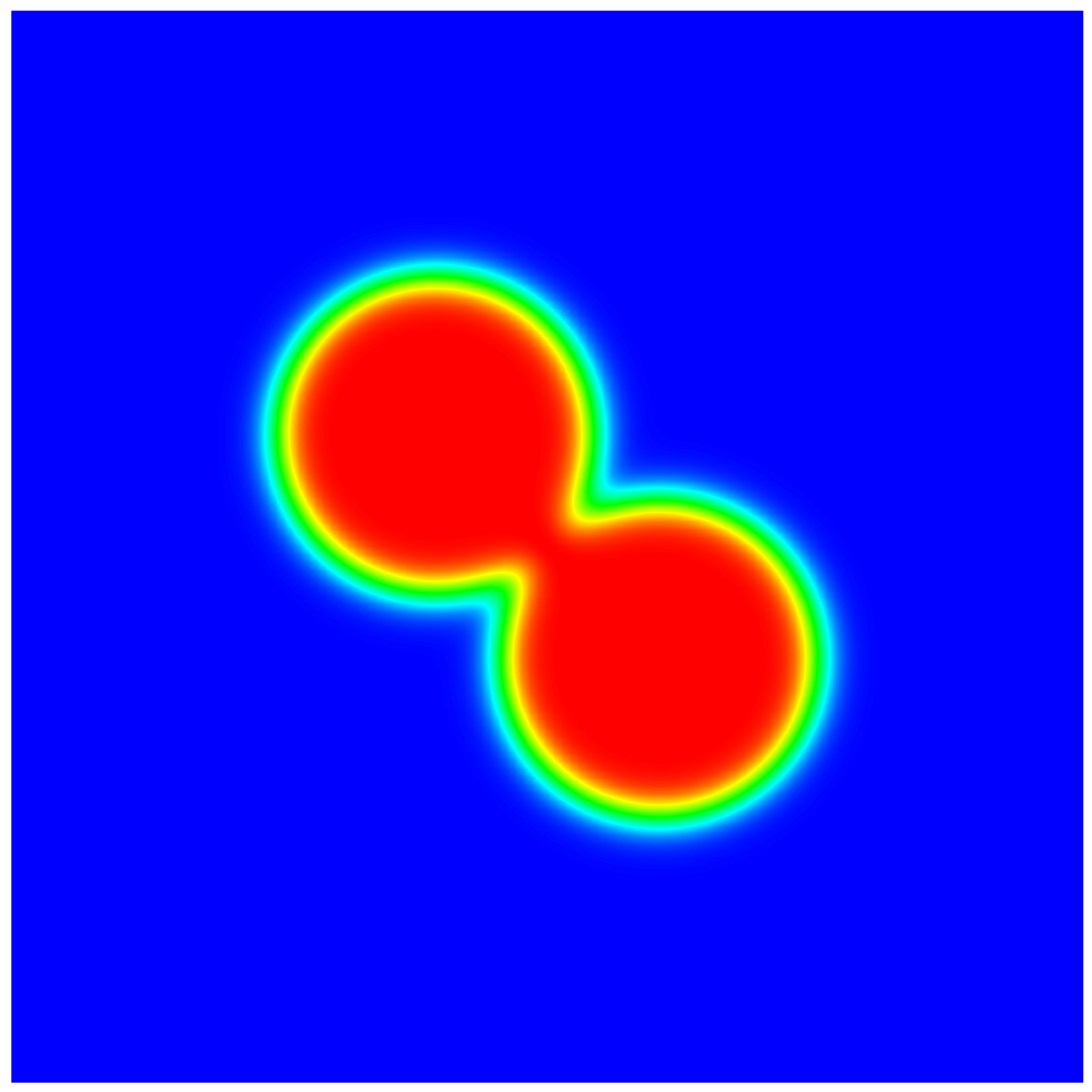}
\includegraphics[width=0.19\textwidth]{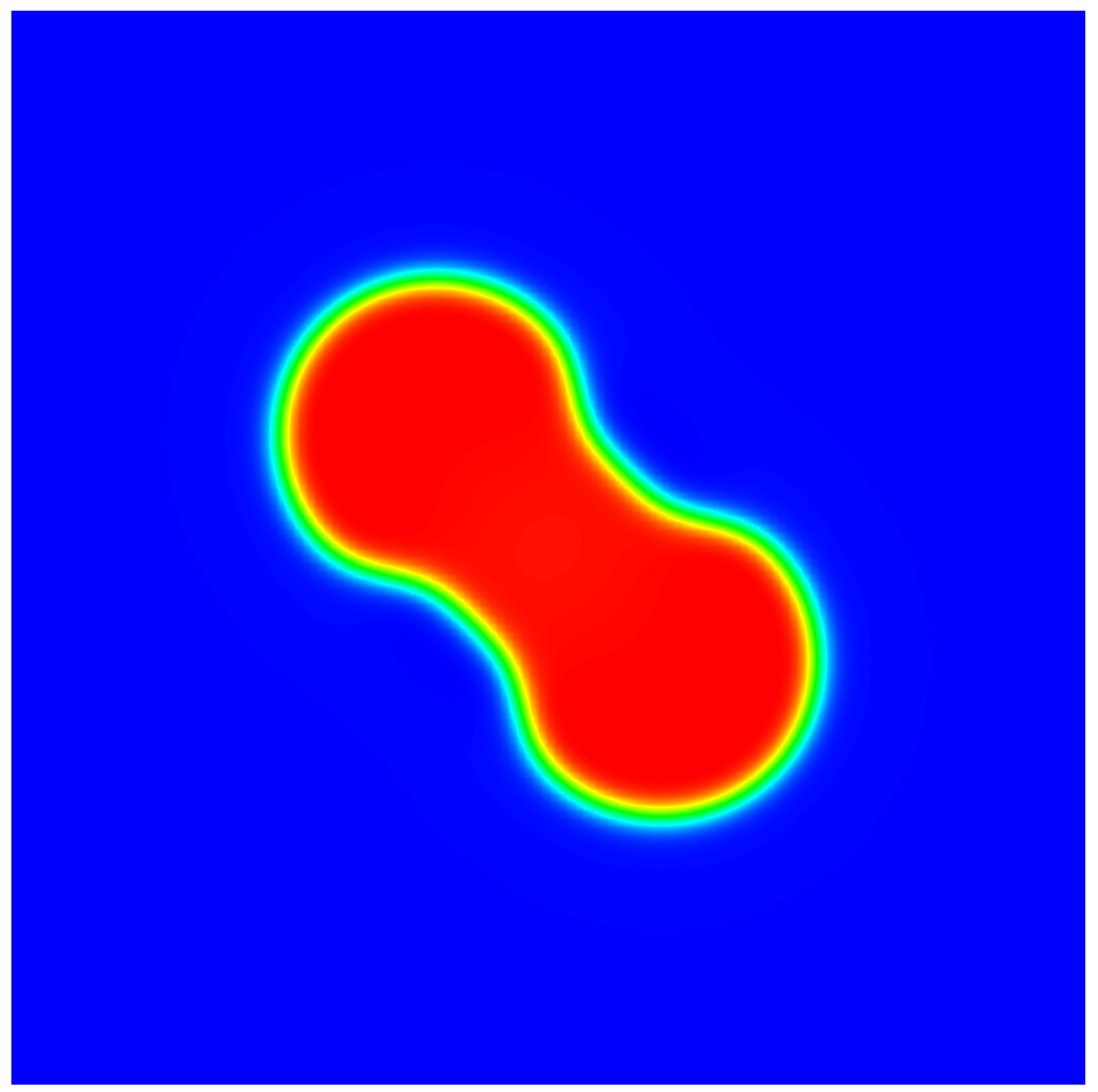}
\includegraphics[width=0.19\textwidth]{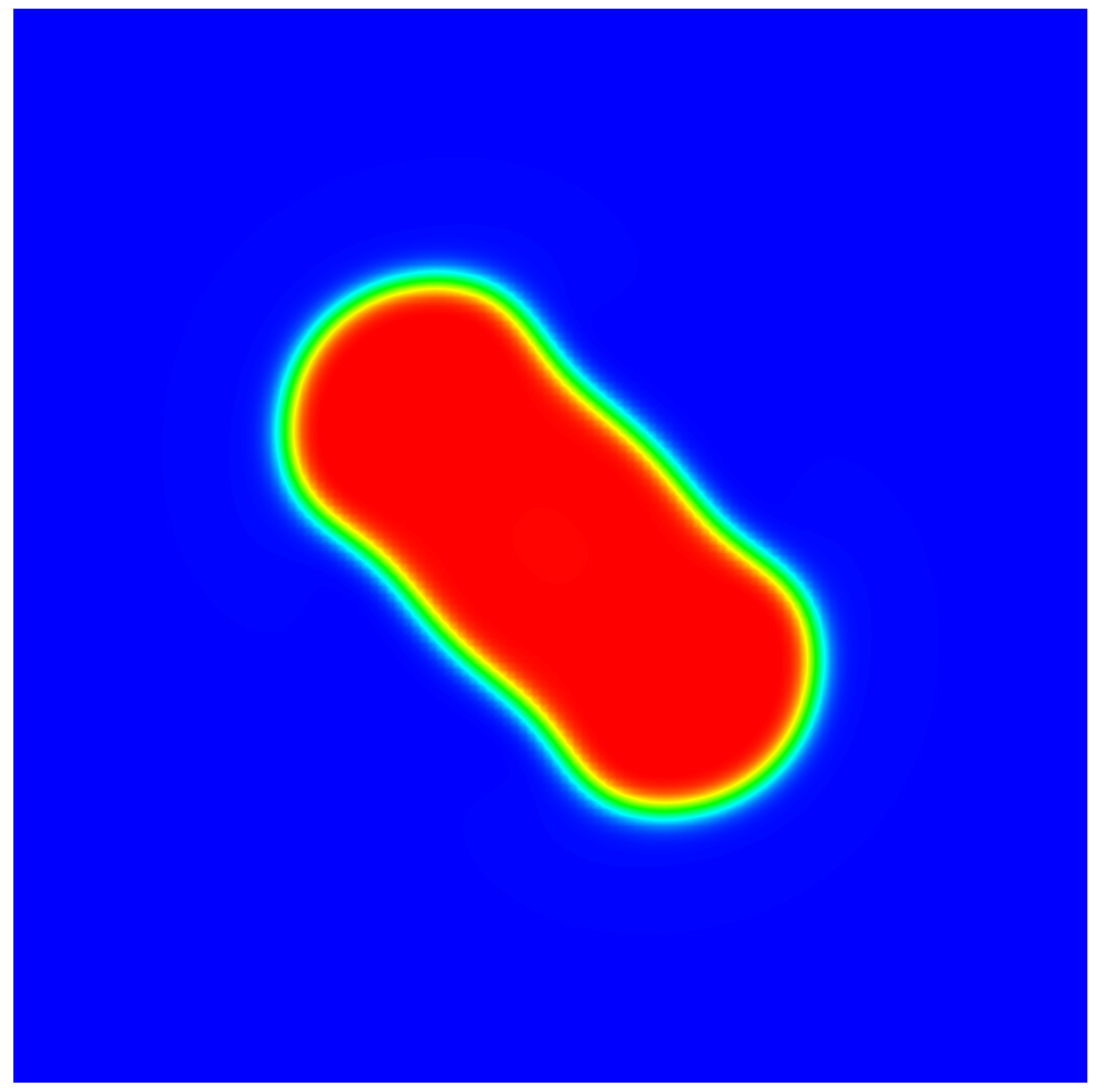}
\includegraphics[width=0.19\textwidth]{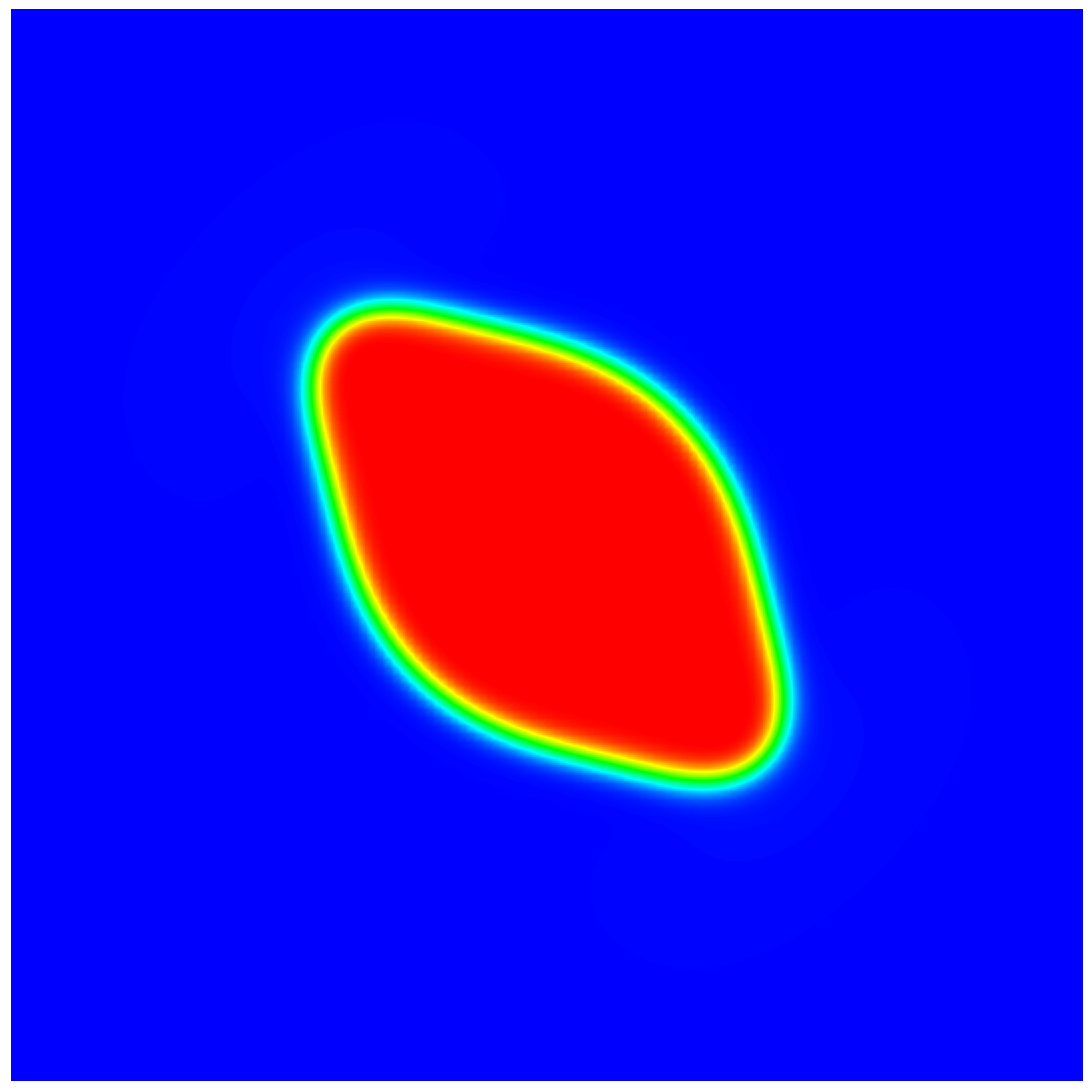}
\includegraphics[width=0.19\textwidth]{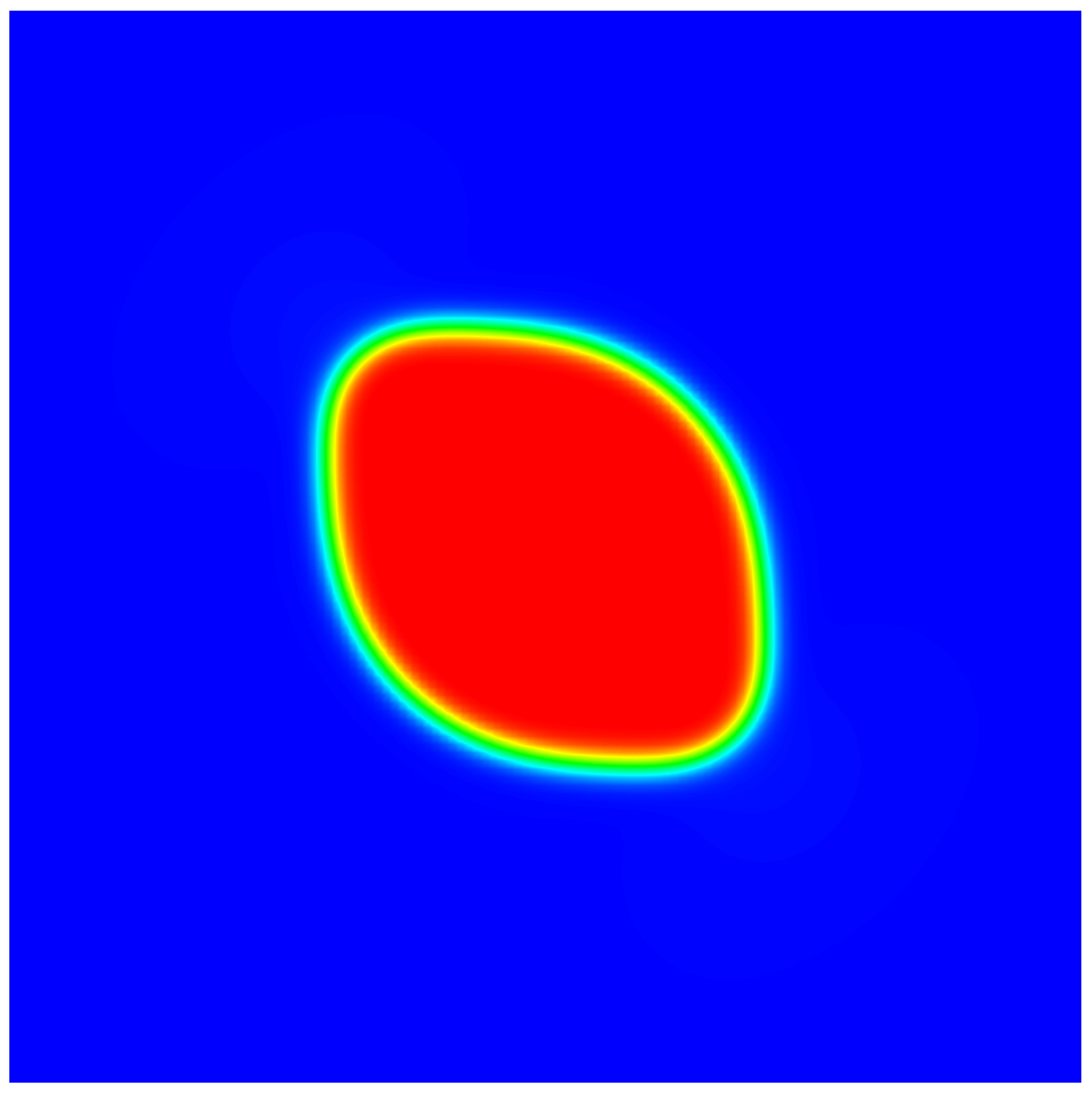}
}

\subfigure[$\phi$ at $t=1.2, 1.5, 2, 3, 3.2$]{
\includegraphics[width=0.19\textwidth]{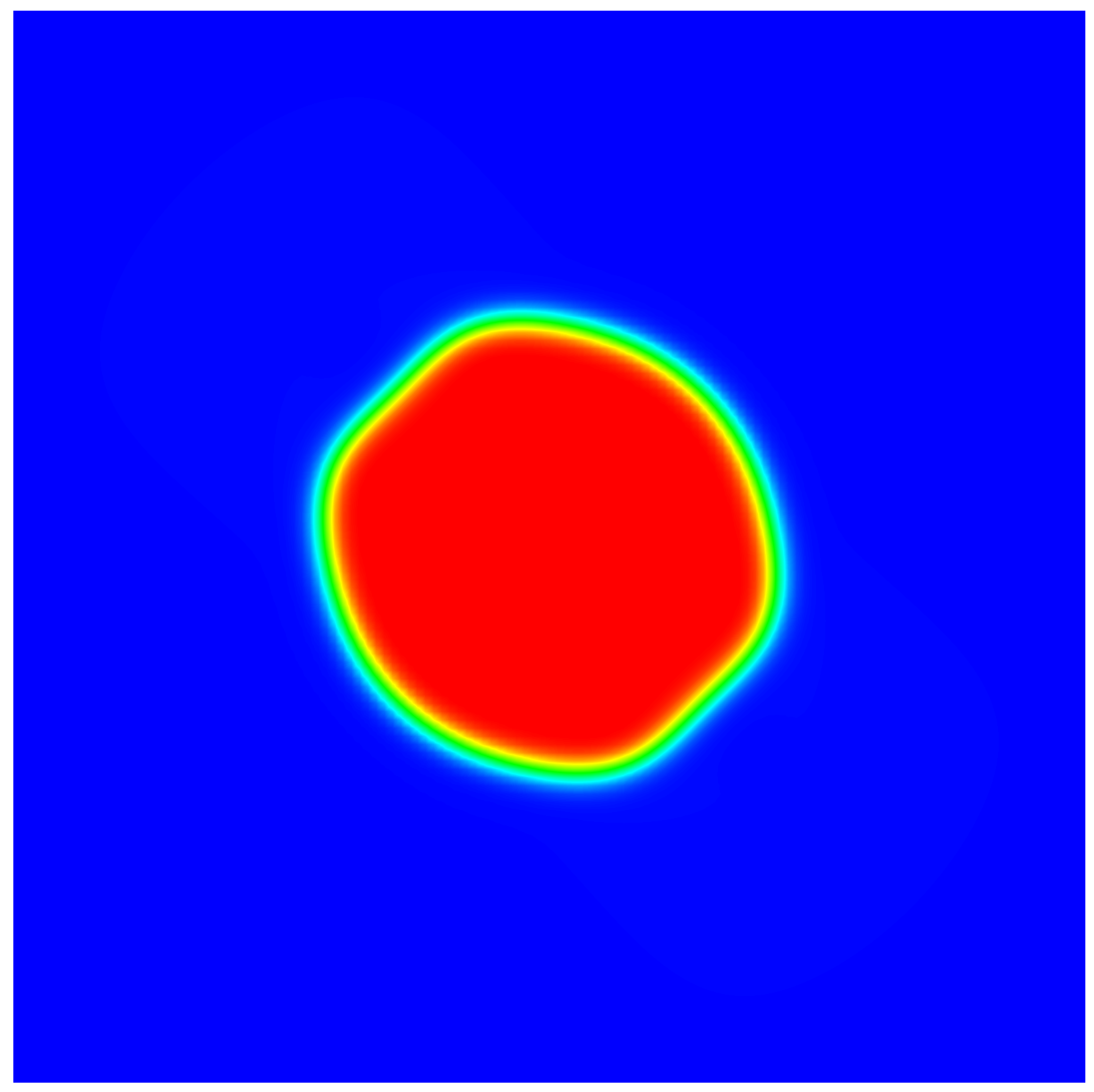}
\includegraphics[width=0.19\textwidth]{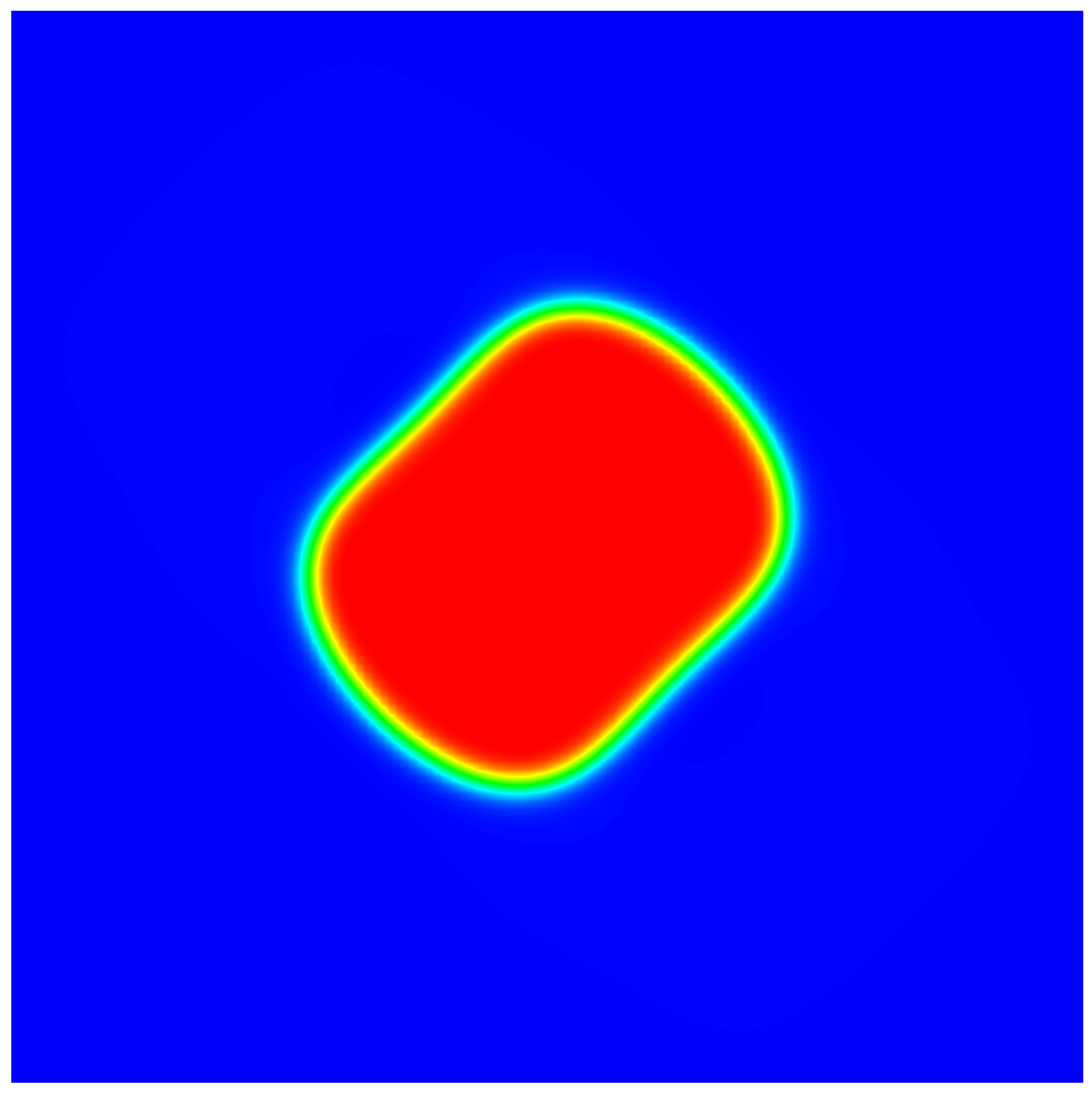}
\includegraphics[width=0.19\textwidth]{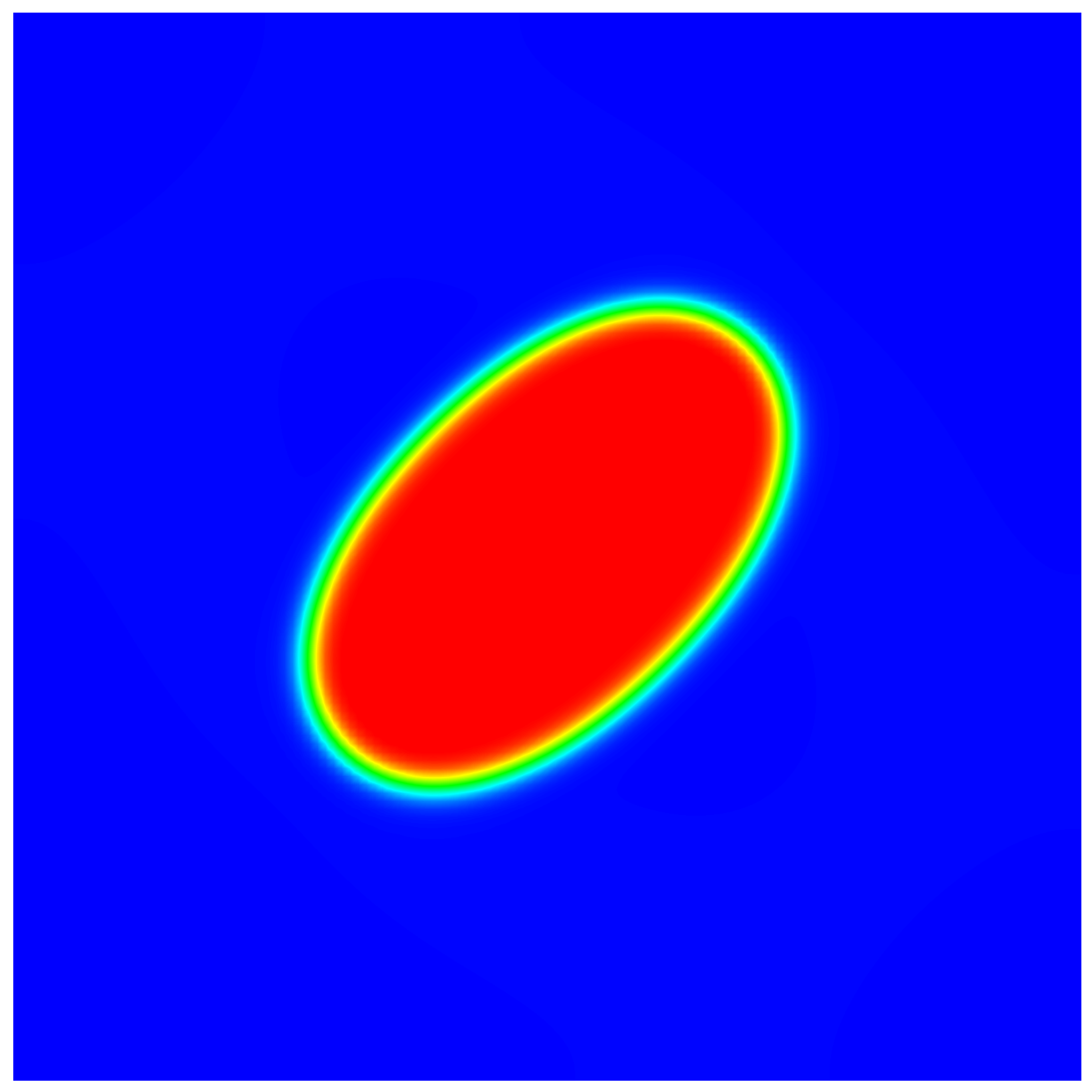}
\includegraphics[width=0.19\textwidth]{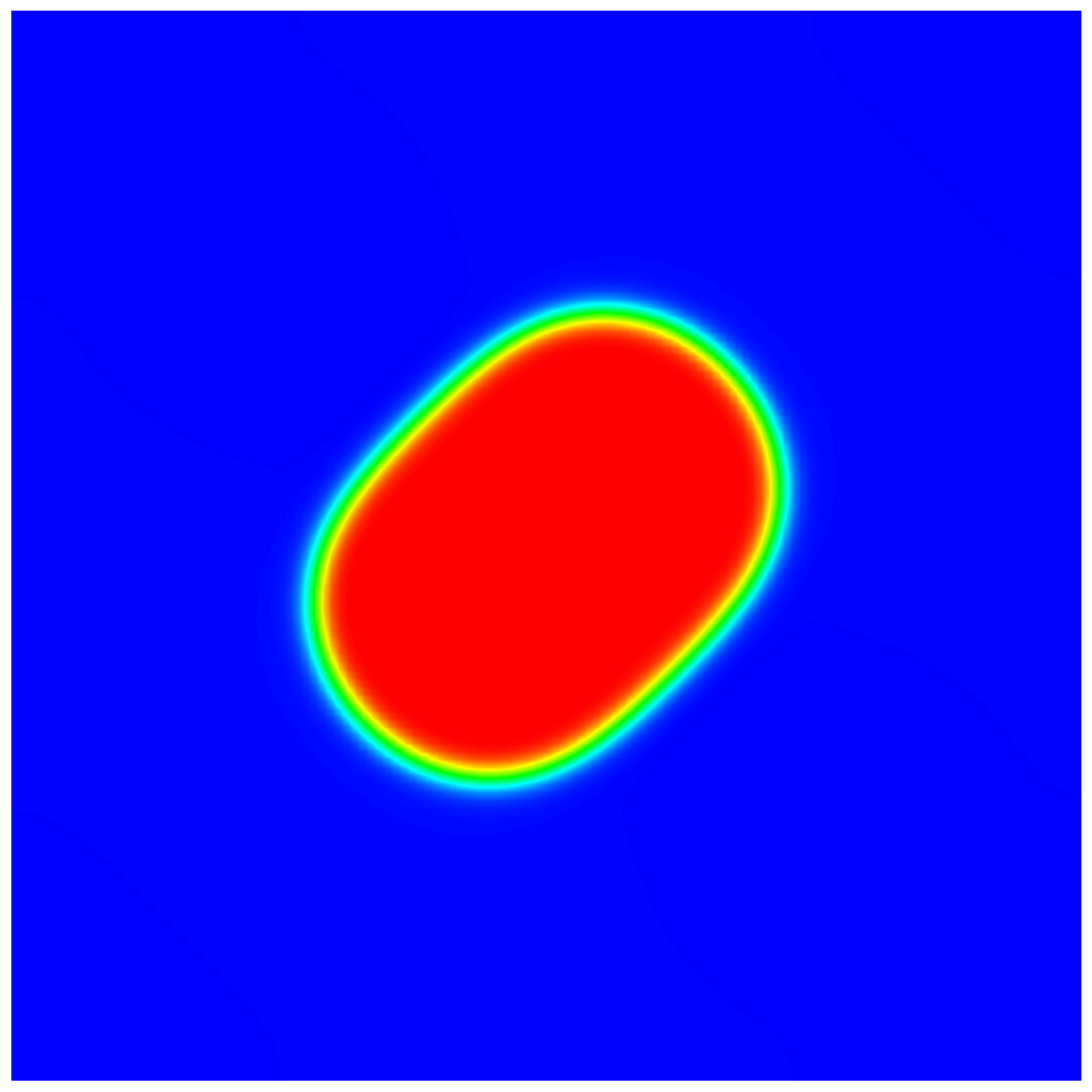}
\includegraphics[width=0.19\textwidth]{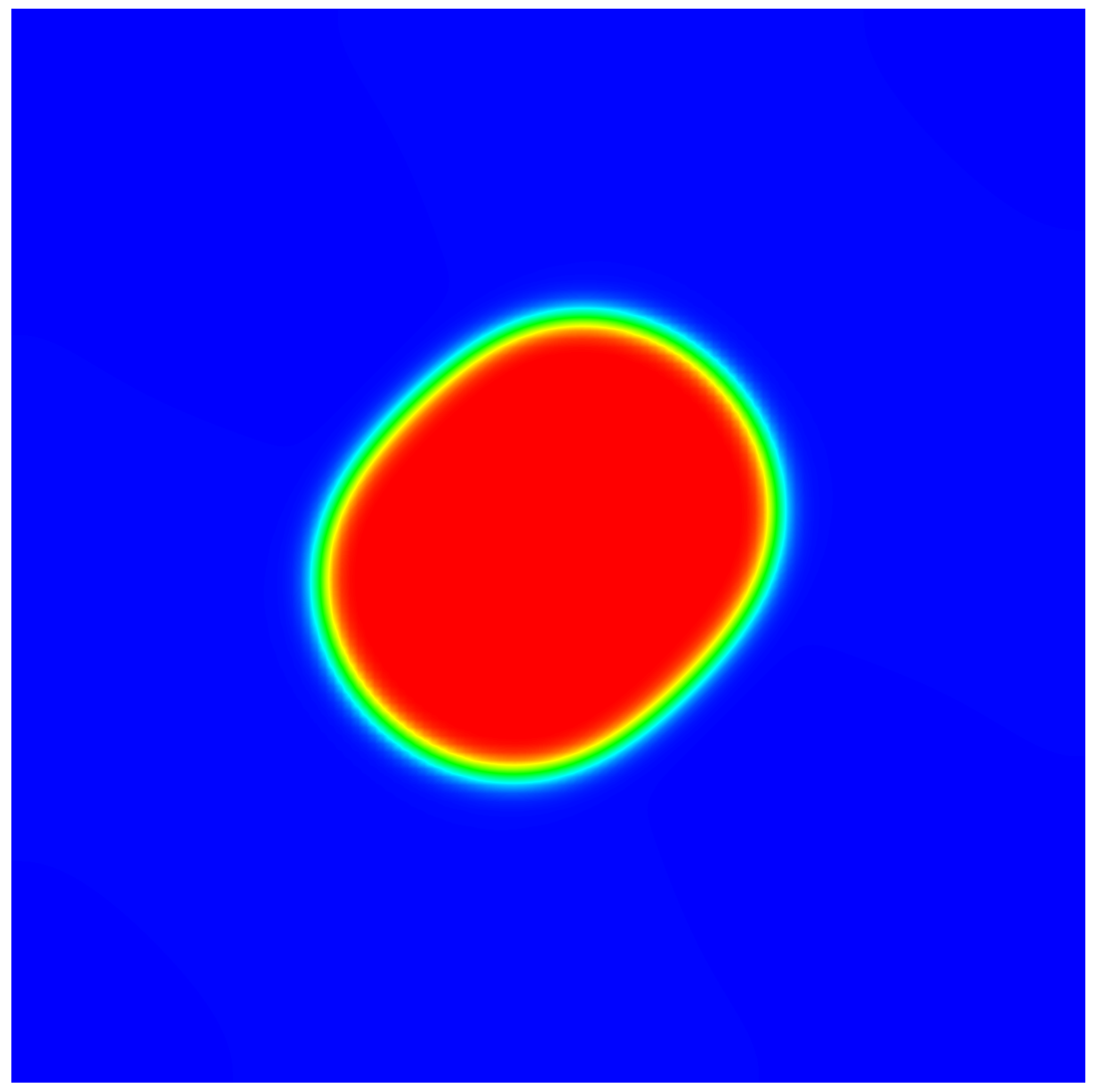}
}

\subfigure[$\phi$ at $t=3.4, 4.6, 4.8, 5, 10$]{
\includegraphics[width=0.19\textwidth]{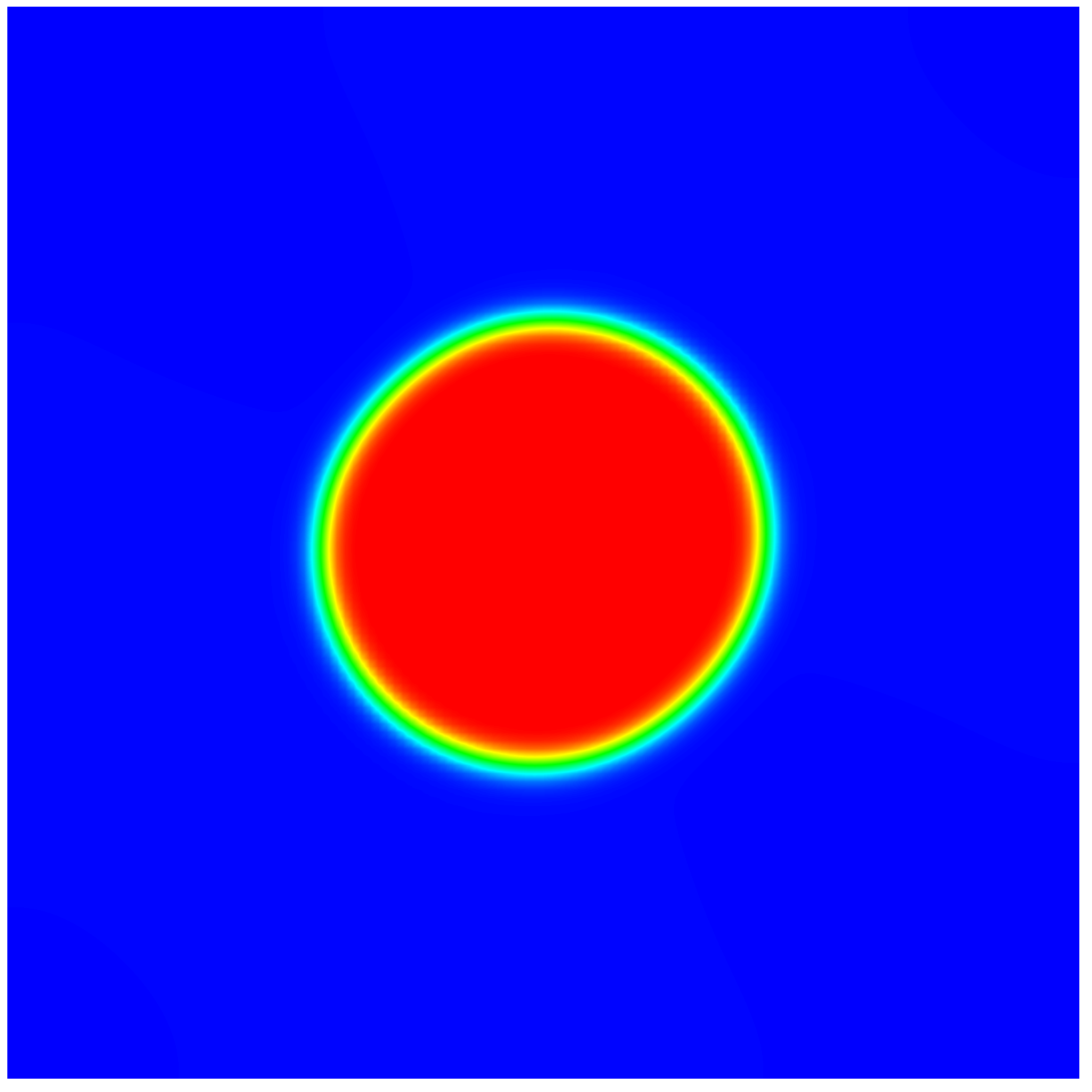}
\includegraphics[width=0.19\textwidth]{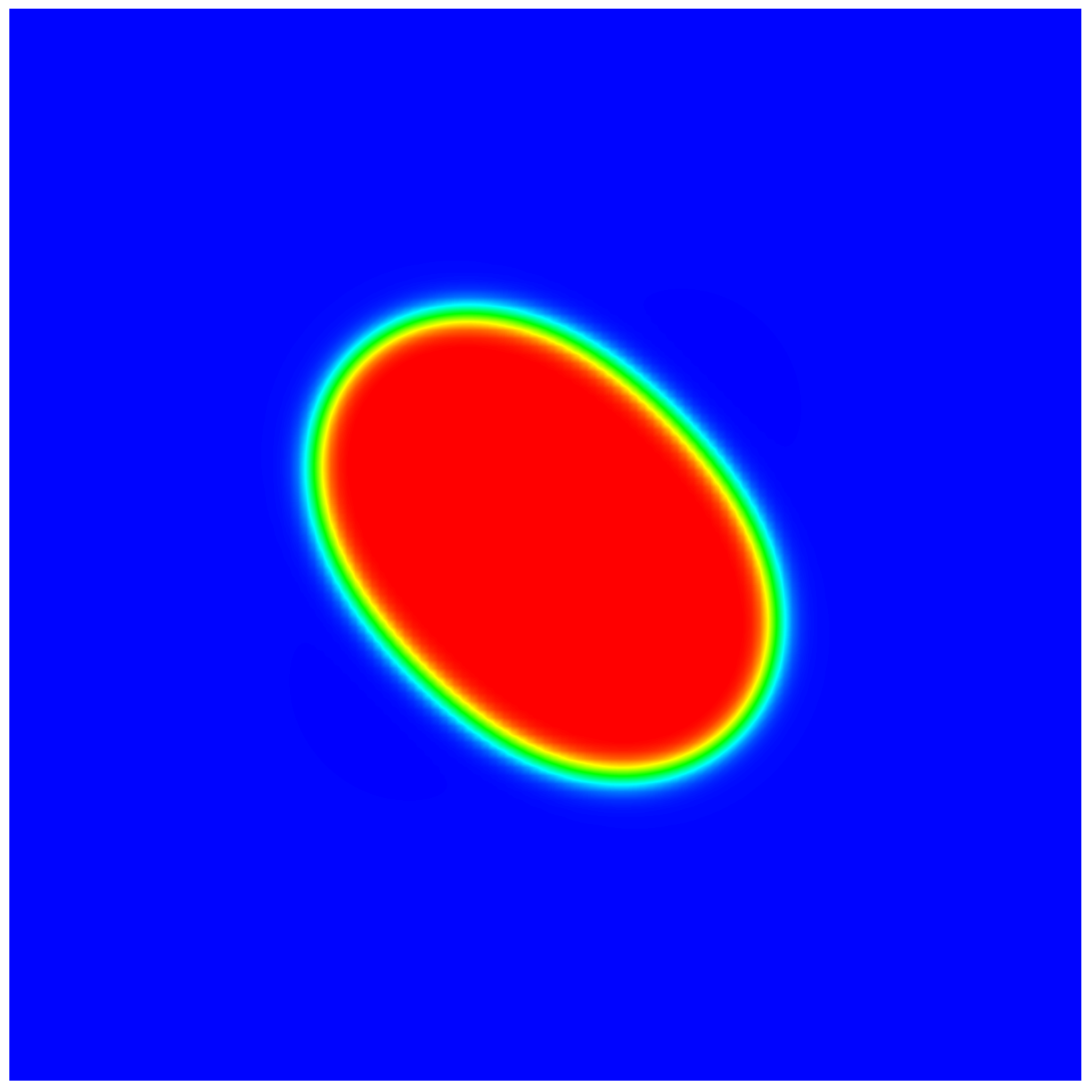}
\includegraphics[width=0.19\textwidth]{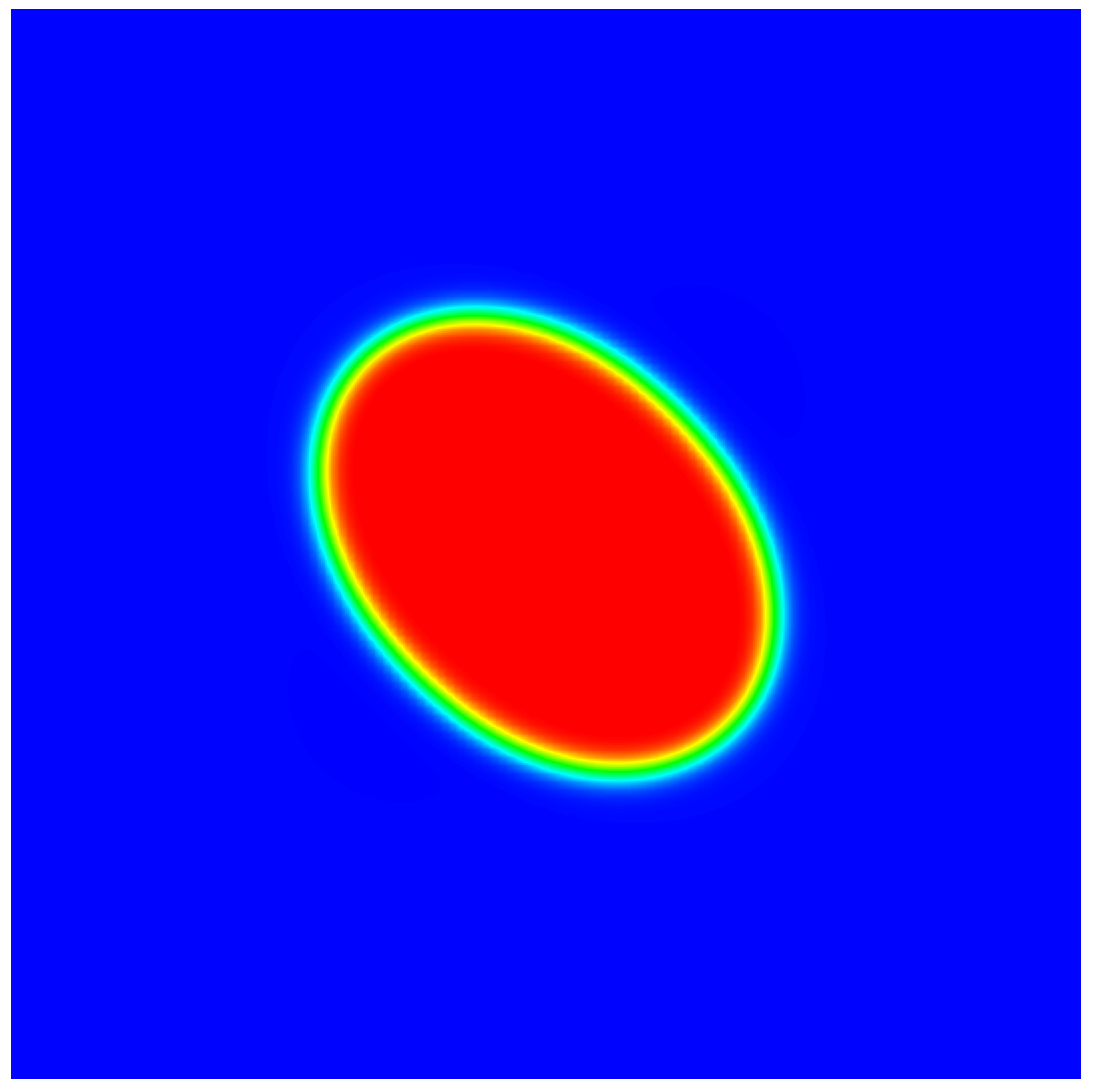}
\includegraphics[width=0.19\textwidth]{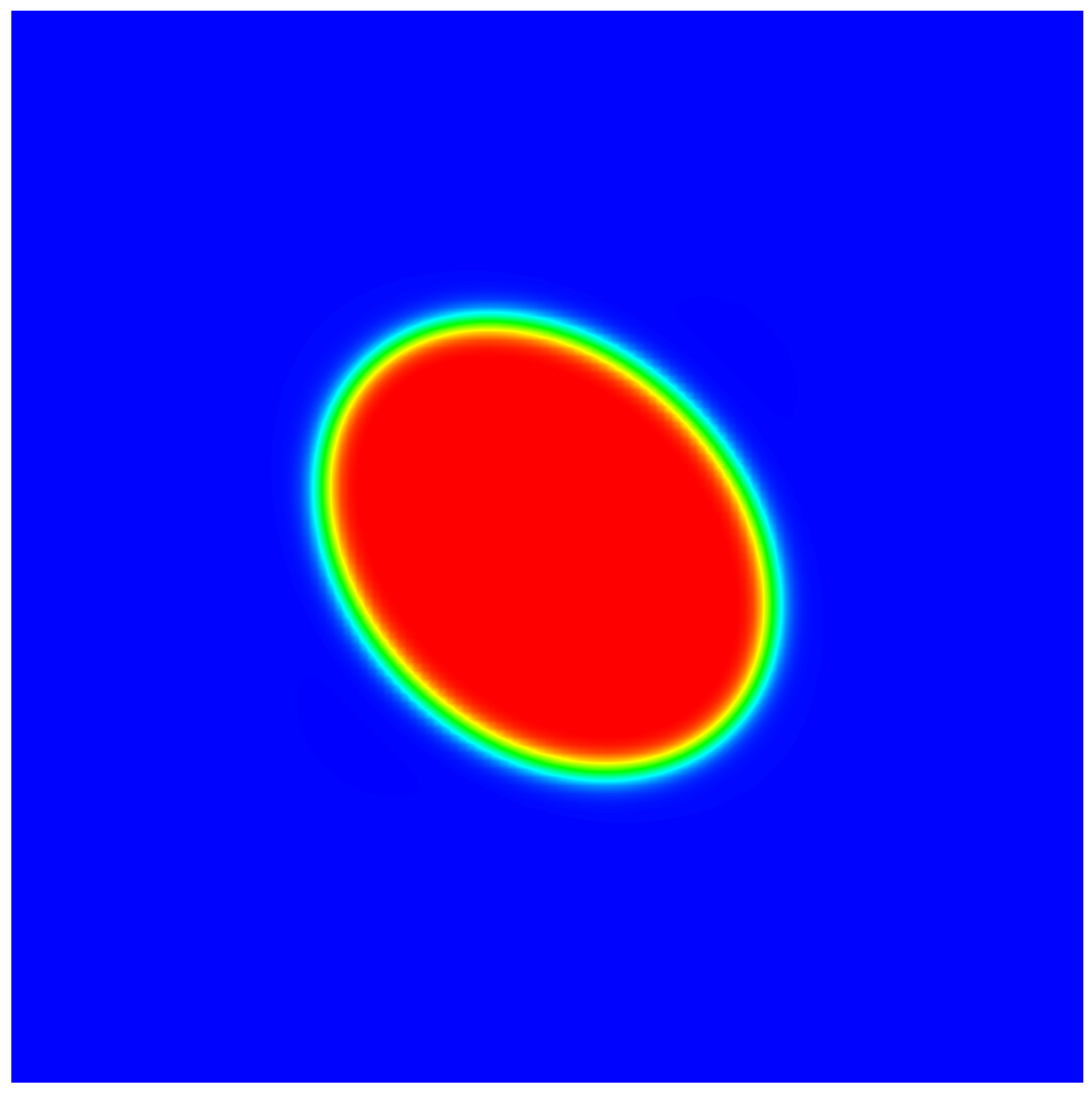}
\includegraphics[width=0.19\textwidth]{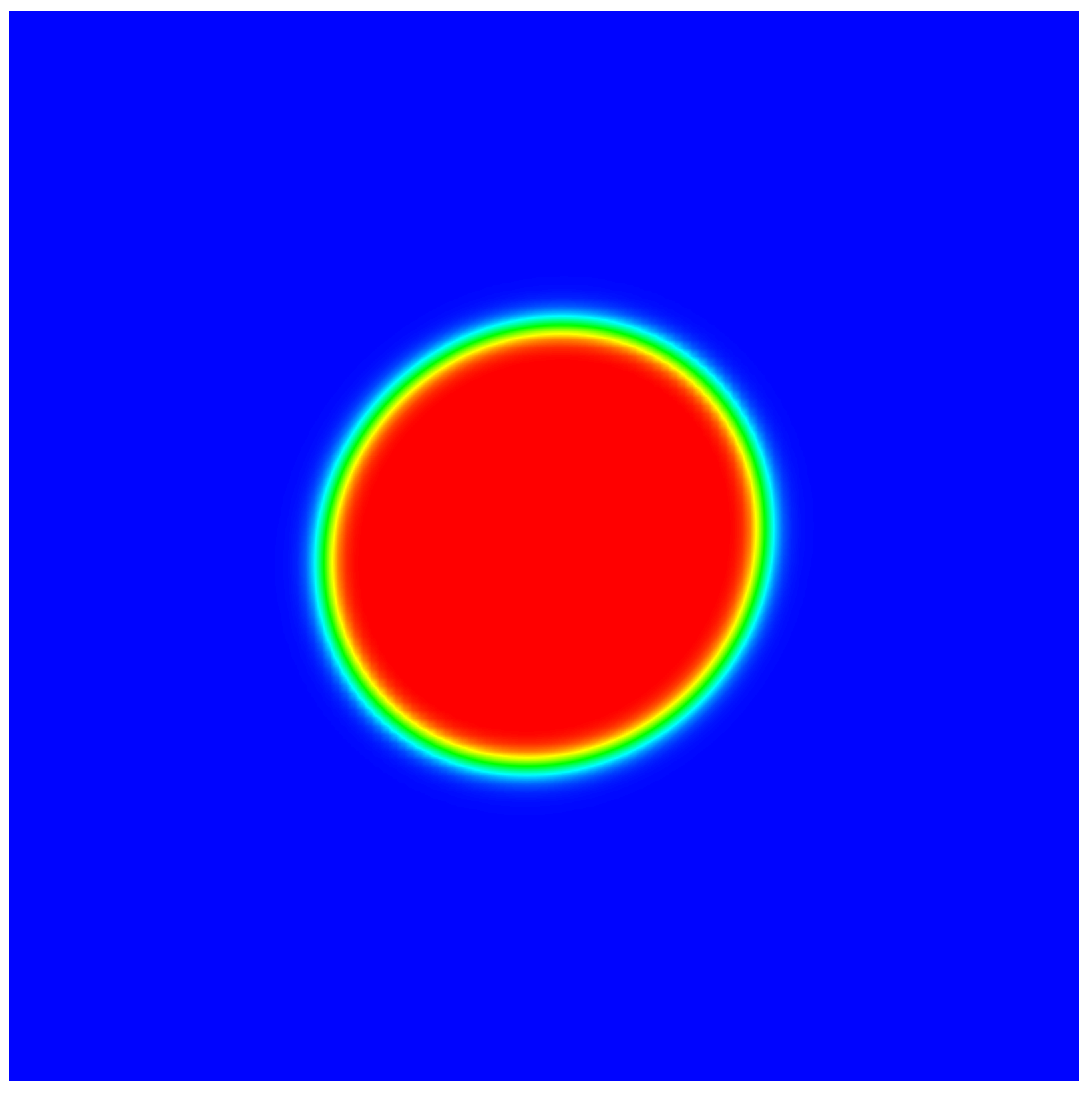}
}
\caption{The dynamics of bubble merging under hydrodynamic environment with viscosity $\eta = 0.001$. In this figure,  the profiles of the phase variable $\phi$ are shown at different times.} 
\label{fig:ETA3}
\end{figure}

To further compare the dynamics for the two cases above, we also visualize the velocity fields for both cases, with the results summarized in Figure \ref{fig:ETA-velo}. It can be shown that the kinetic energy is high for the case with smaller viscosity, and the kinetic energy shows oscillating and damping dynamics (by noticing the magnitude of the velocity field is decreasing with time).

\begin{figure}[H]
\center
\subfigure[Velocity field for Figure \ref{fig:ETA2} at time $t=1,2,5$]{
\includegraphics[width=0.32\textwidth]{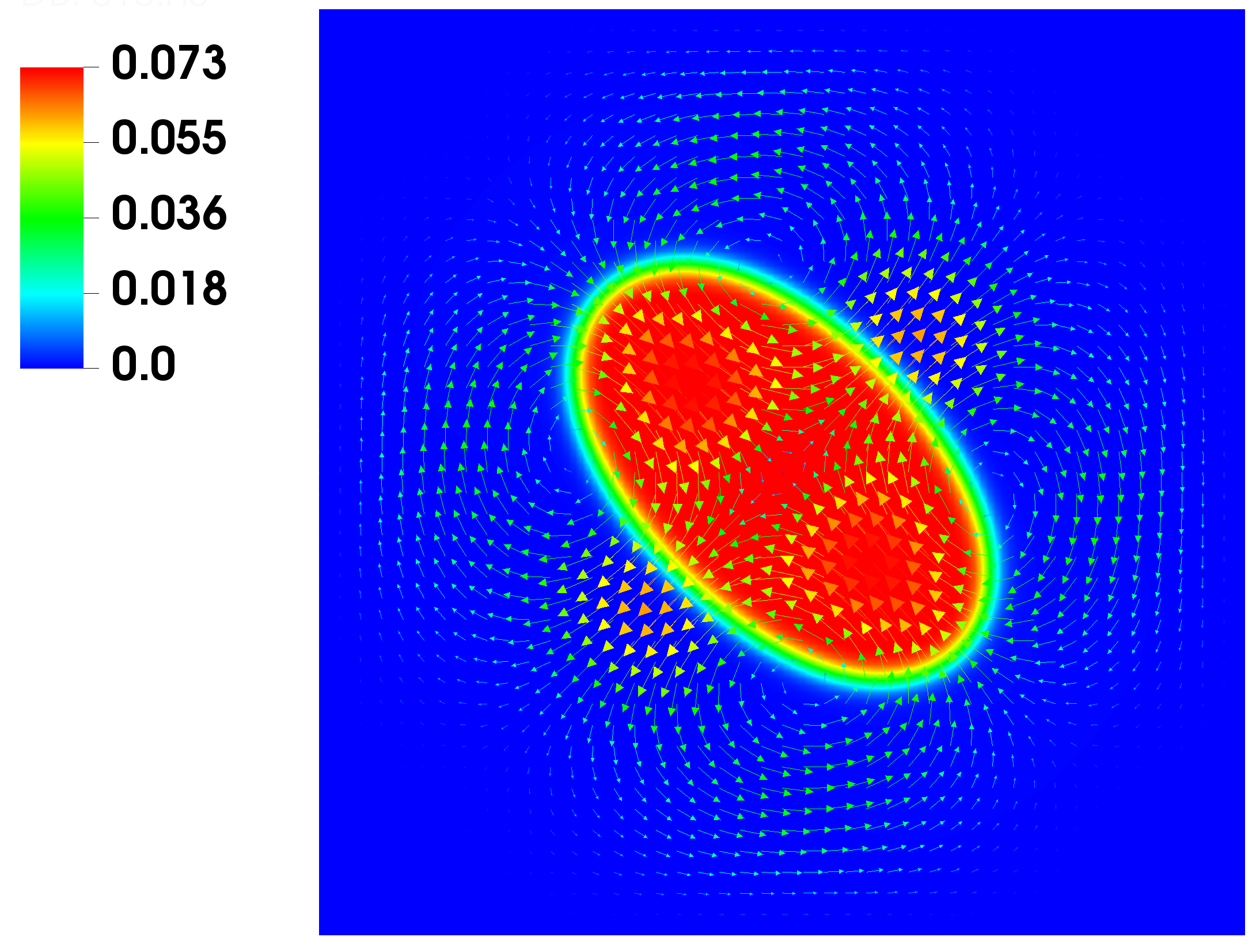}
\includegraphics[width=0.32\textwidth]{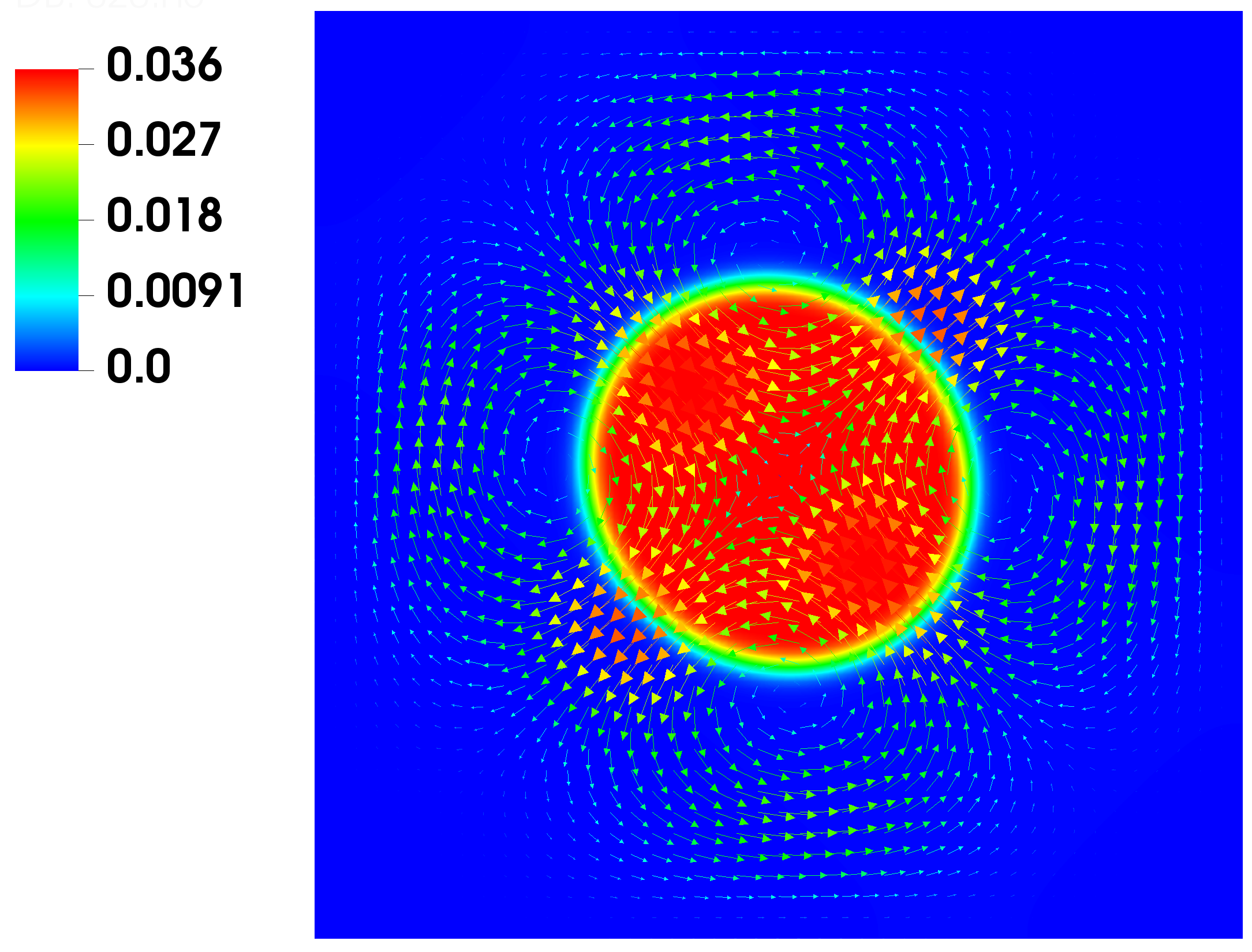}
\includegraphics[width=0.32\textwidth]{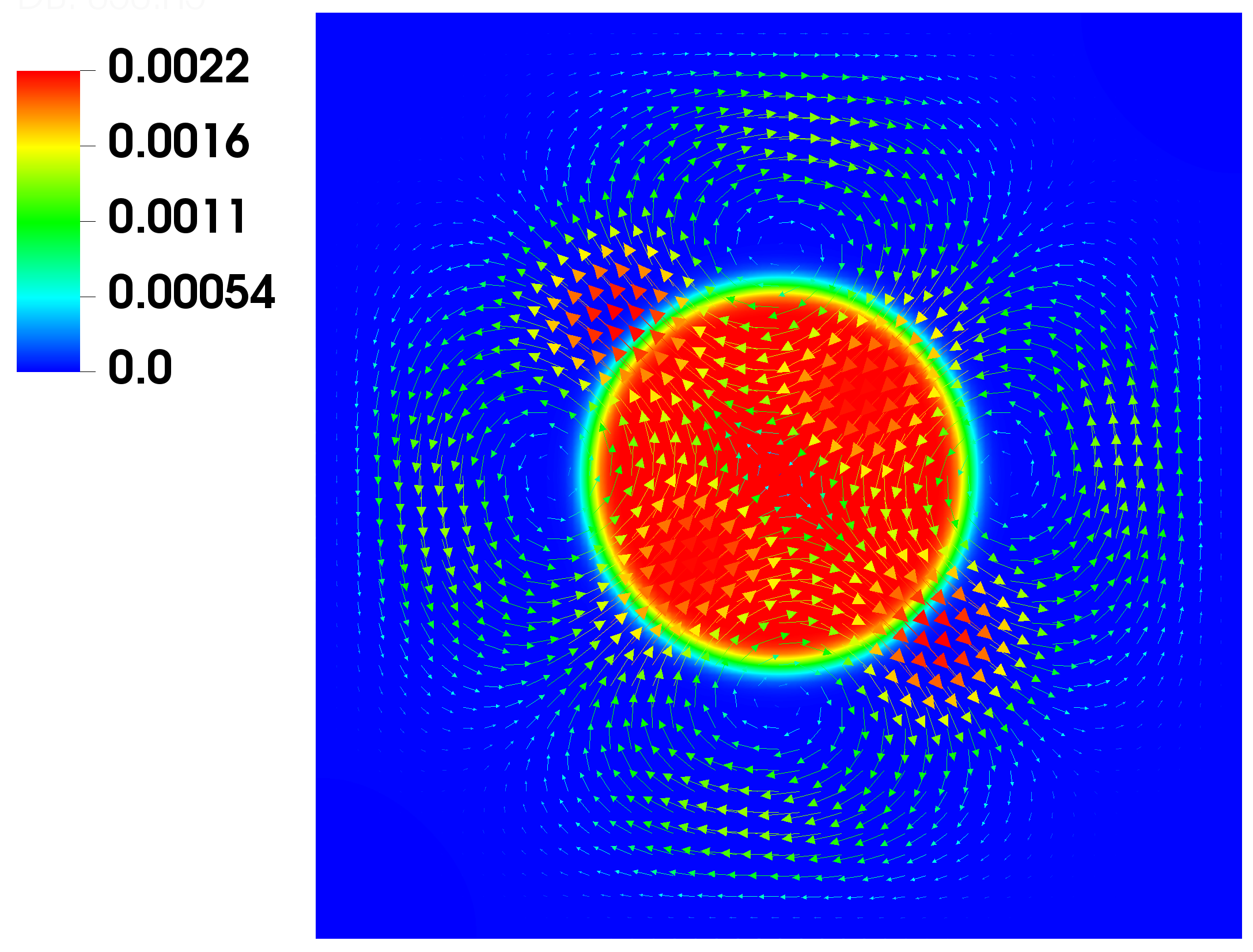}}

\subfigure[Velocity field for Figure \ref{fig:ETA3} at time $t=1,2,5$]{
\includegraphics[width=0.32\textwidth]{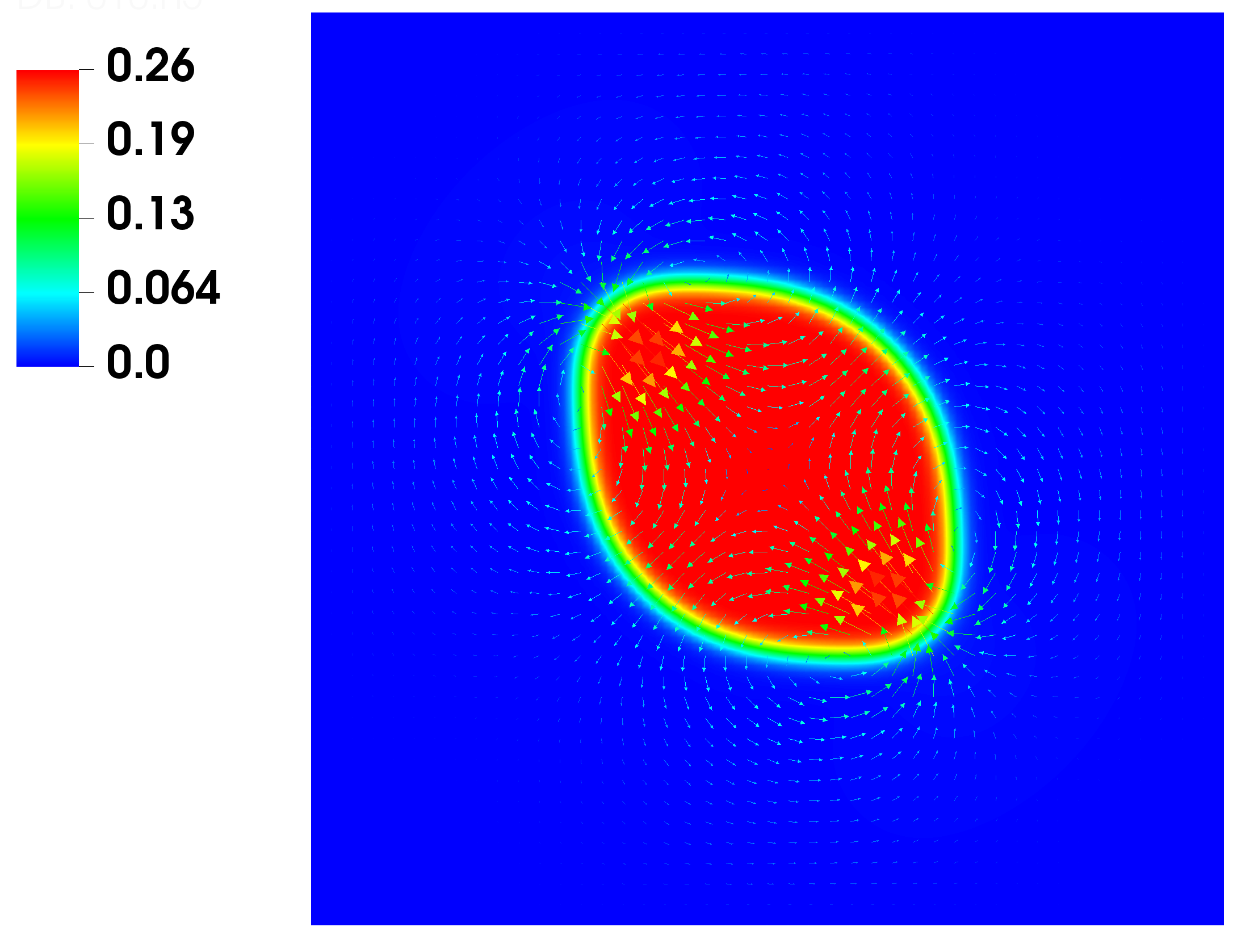}
\includegraphics[width=0.32\textwidth]{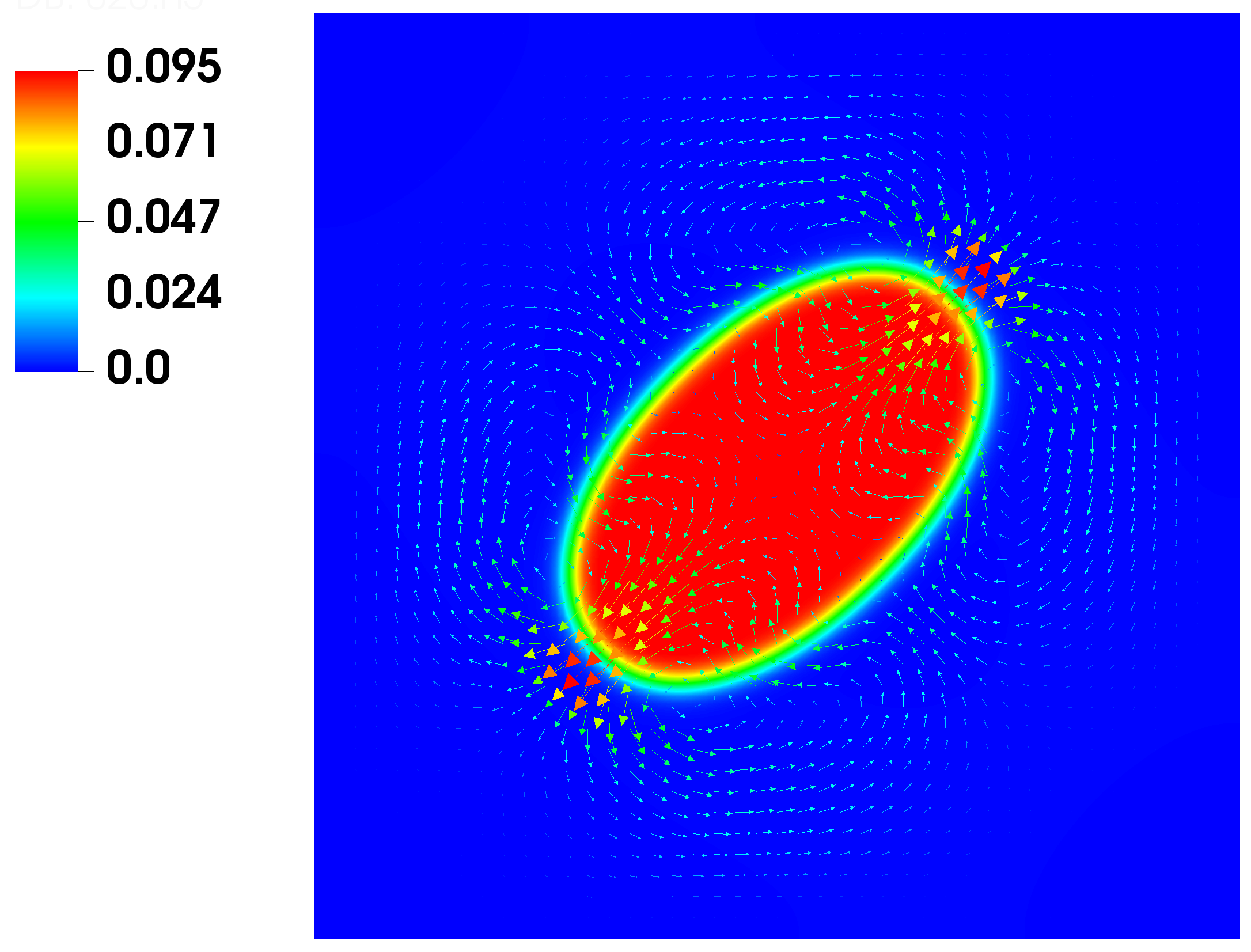}
\includegraphics[width=0.32\textwidth]{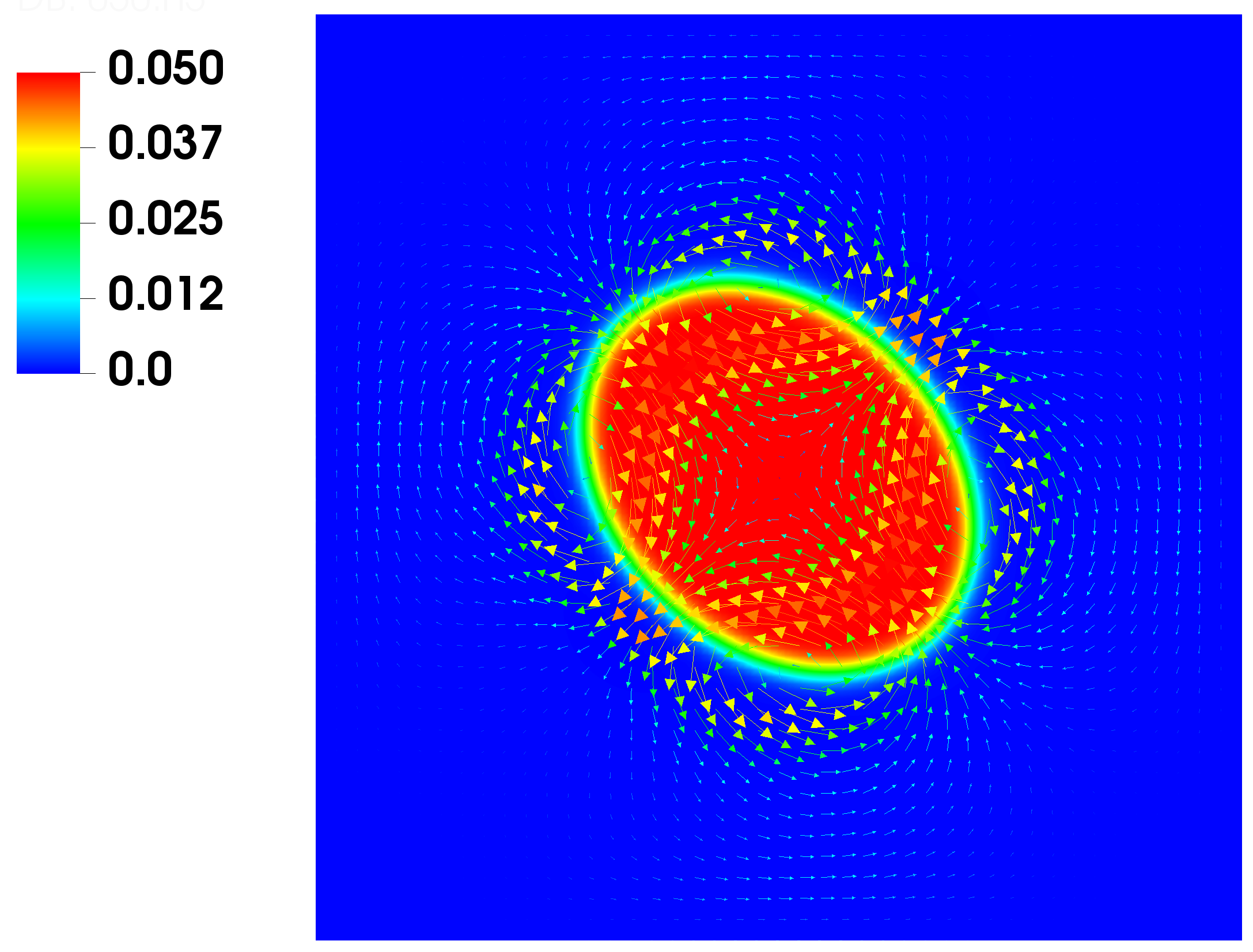}
}

\caption{Visualization of the velocity field for the dynamics in Figure \ref{fig:ETA2} and \ref{fig:ETA3}. (a) the velocity field for Figure \ref{fig:ETA2}; (b) the velocity field for Figure \ref{fig:ETA3}.}
\label{fig:ETA-velo}
\end{figure}

Meanwhile, as a double-verification of the energy stable property of our proposed scheme \ref{Scheme:sch-3}, the energy evolution with time for both cases are summarized in Figure \ref{fig:ETA-E}. We can observe that the energy is dissipating in time for both cases. The one with smaller viscosity is dissipating slower, which is reasonable, as the dissipation rate is proportional to the viscosity as shown in  \ref{eq:energy-law-continous}.

\begin{figure}[H]
\center
\subfigure[Energy evolution for Figure \ref{fig:ETA2}]{\includegraphics[width=0.45\textwidth]{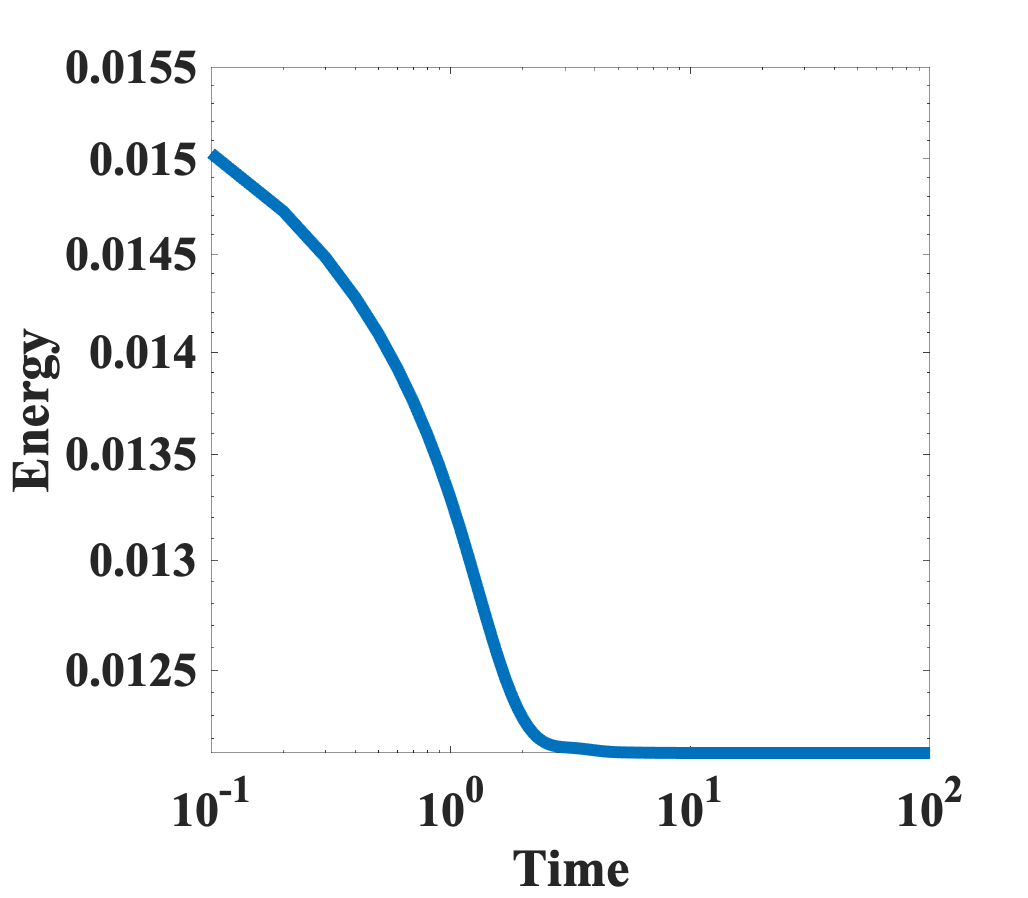}}
\subfigure[Energy evolution for Figure \ref{fig:ETA3}]{\includegraphics[width=0.45\textwidth]{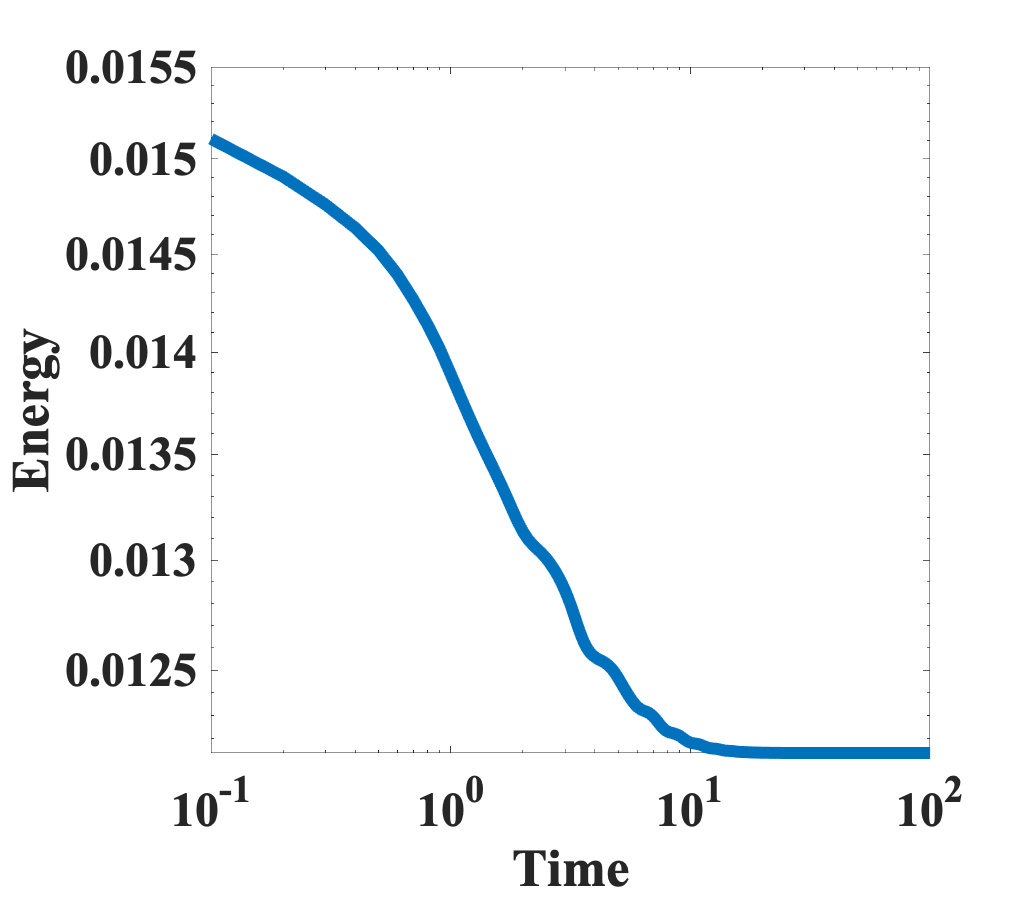}}
\caption{The time evolution of the energy in log-log scale.}
\label{fig:ETA-E}
\end{figure}

\subsection{Coarsening}

Next, we investigate the coarsening dynamics under hydrodynamics environments. We use the domain $\Omega=[0, L_x] \times [0, L_y]$ with $L_x=1$, $L_y=2$. The parameters are picked as $\rho =1$, $\eta = 1.0$, $\lambda = 0.01$, $\varepsilon=0.01$.  To solve the problem, we use uniform meshes with $Nx=128$ and $Ny=256$. And we pick different surface tension $\gamma$. Se set a random initial condition for the phase variable 
$\phi(x, y, t=0) = 0.9 (\frac{y}{L_y} - 0.5) + 0.001 * rand(-1, 1)$, and $\bu(x,y, t=0) = \mathbf{0}$. 

We choose two different surface tension $\gamma=0.1$ and $\gamma=0.01$. The numerical results are summarized in Figure \ref{fig:coarsening-gam}. We observe that when the volume fraction of two phases is similar, saying in the middle of the domain, spinodal decomposition takes more effect. Meanwhile, when the volume fractions of each phase differ dramatically, the nucleation takes more effect. This agrees well with the results in the literature. 

\begin{figure}[H]
\center 
\subfigure[$\phi$ at $t=0.1, 0.5, 1, 2, 6.5$ for the case $\gamma=0.1$]{
\includegraphics[width=0.19\textwidth]{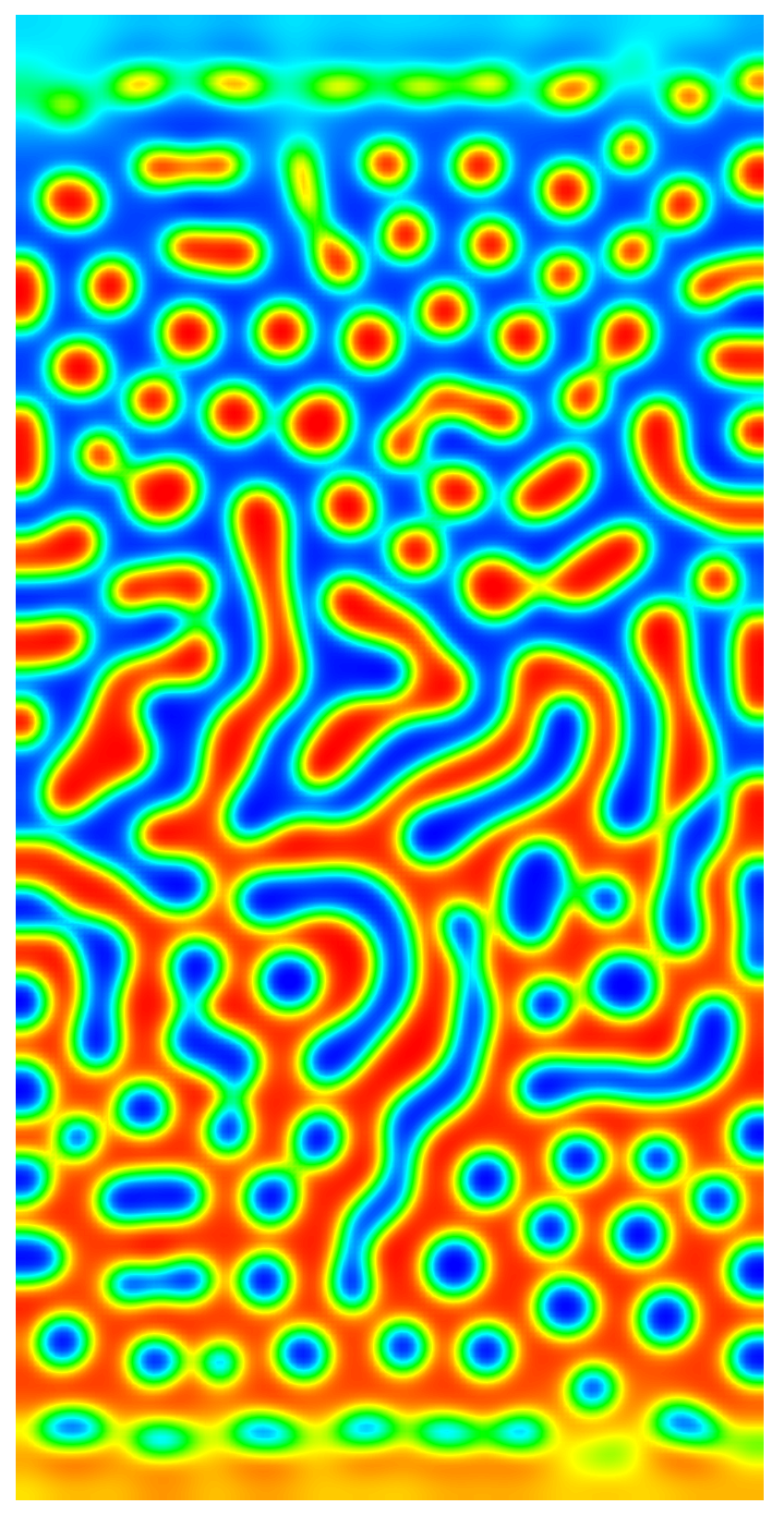}
\includegraphics[width=0.19\textwidth]{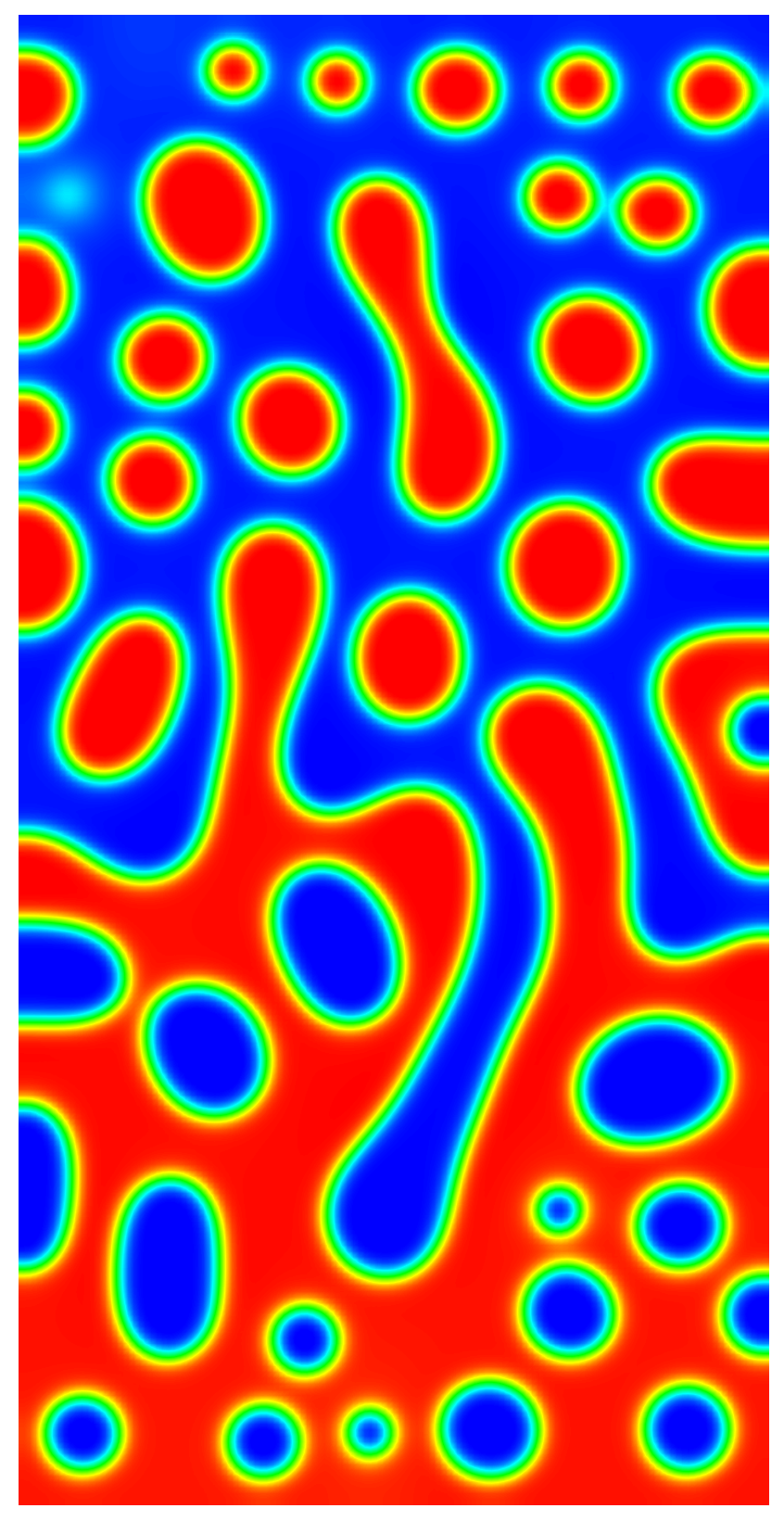}
\includegraphics[width=0.19\textwidth]{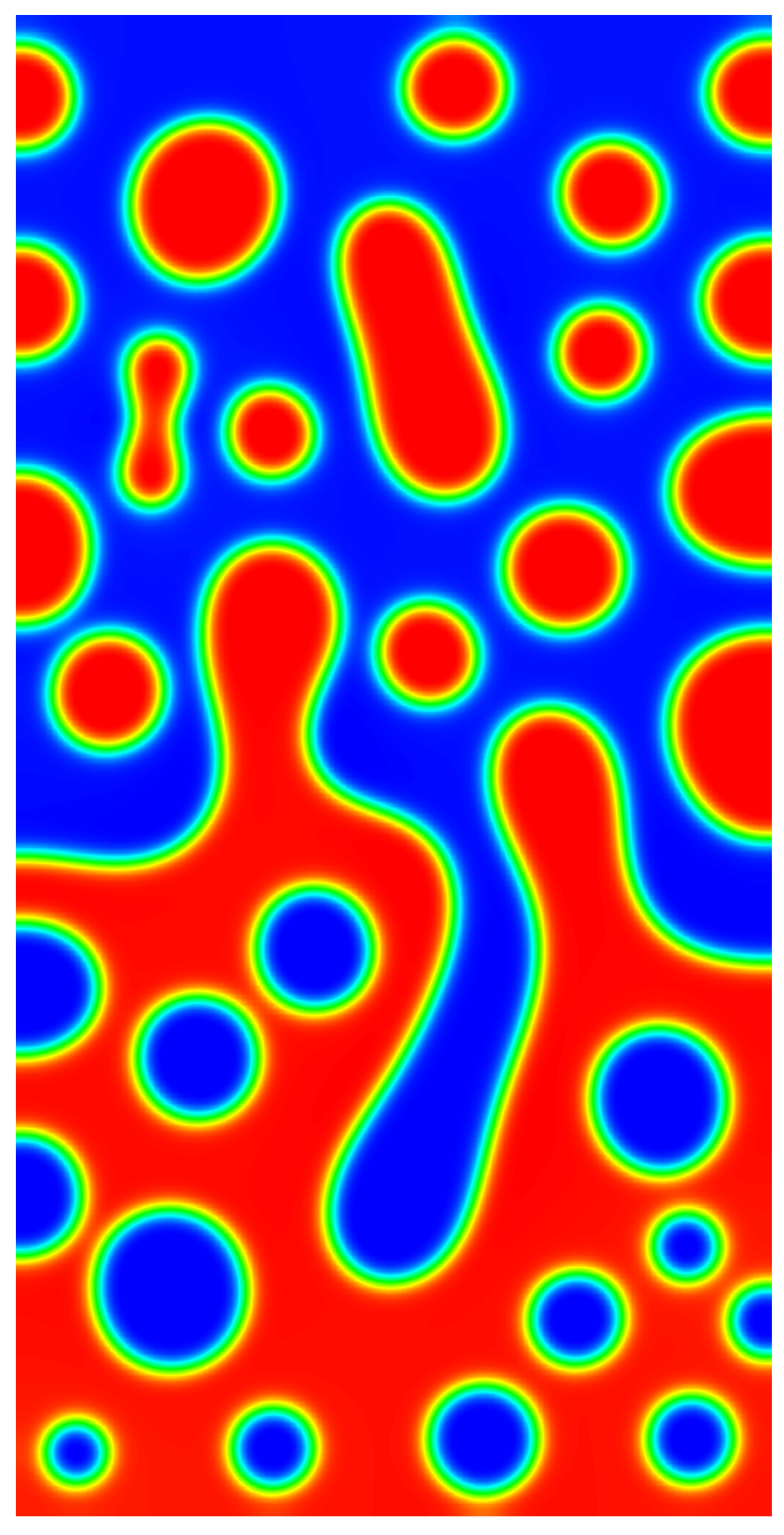}
\includegraphics[width=0.19\textwidth]{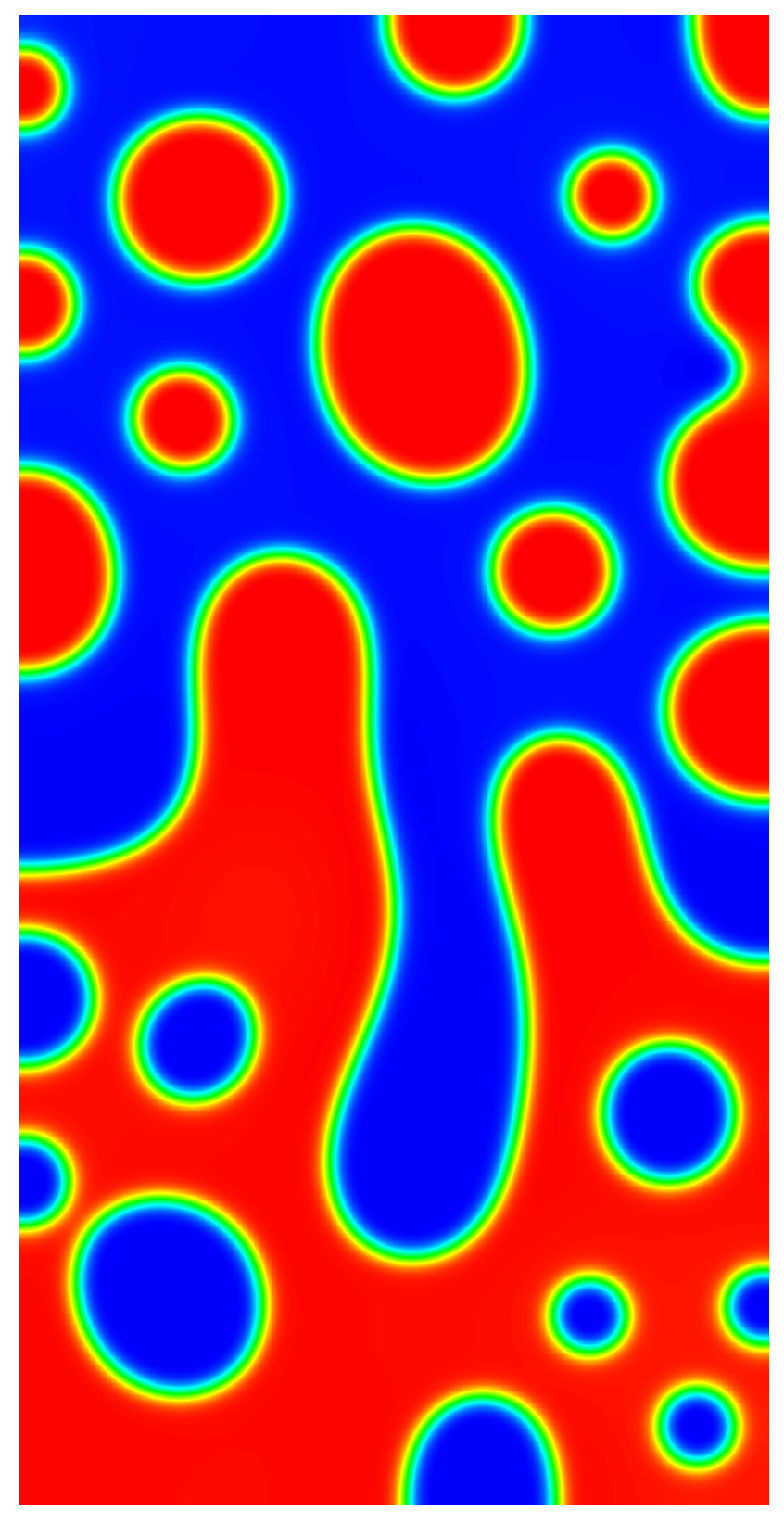}
\includegraphics[width=0.19\textwidth]{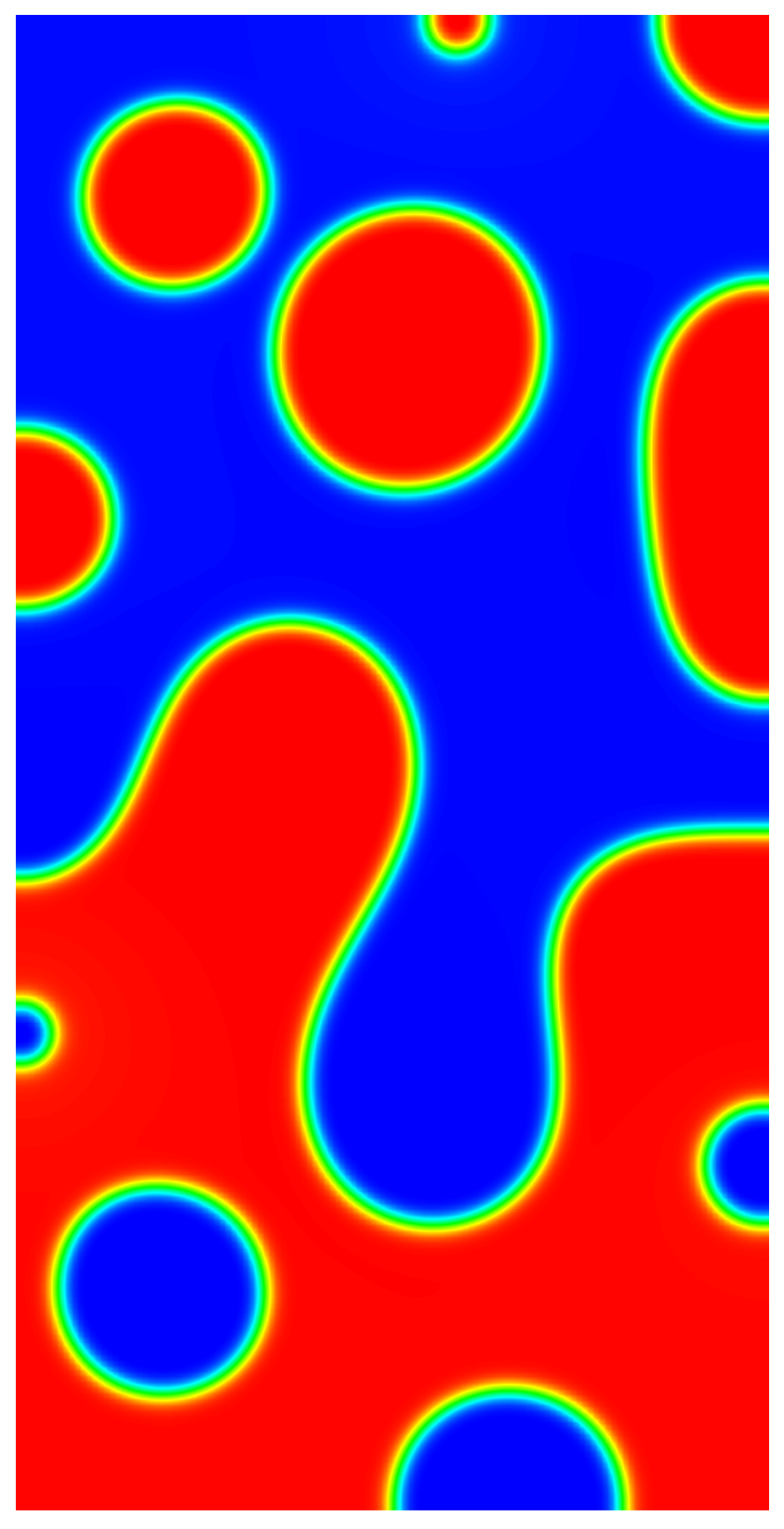}
}

\subfigure[$\phi$ at $t=0.1, 0.5, 1, 2, 6.5$ for the case $\gamma=0.01$]{
\includegraphics[width=0.19\textwidth]{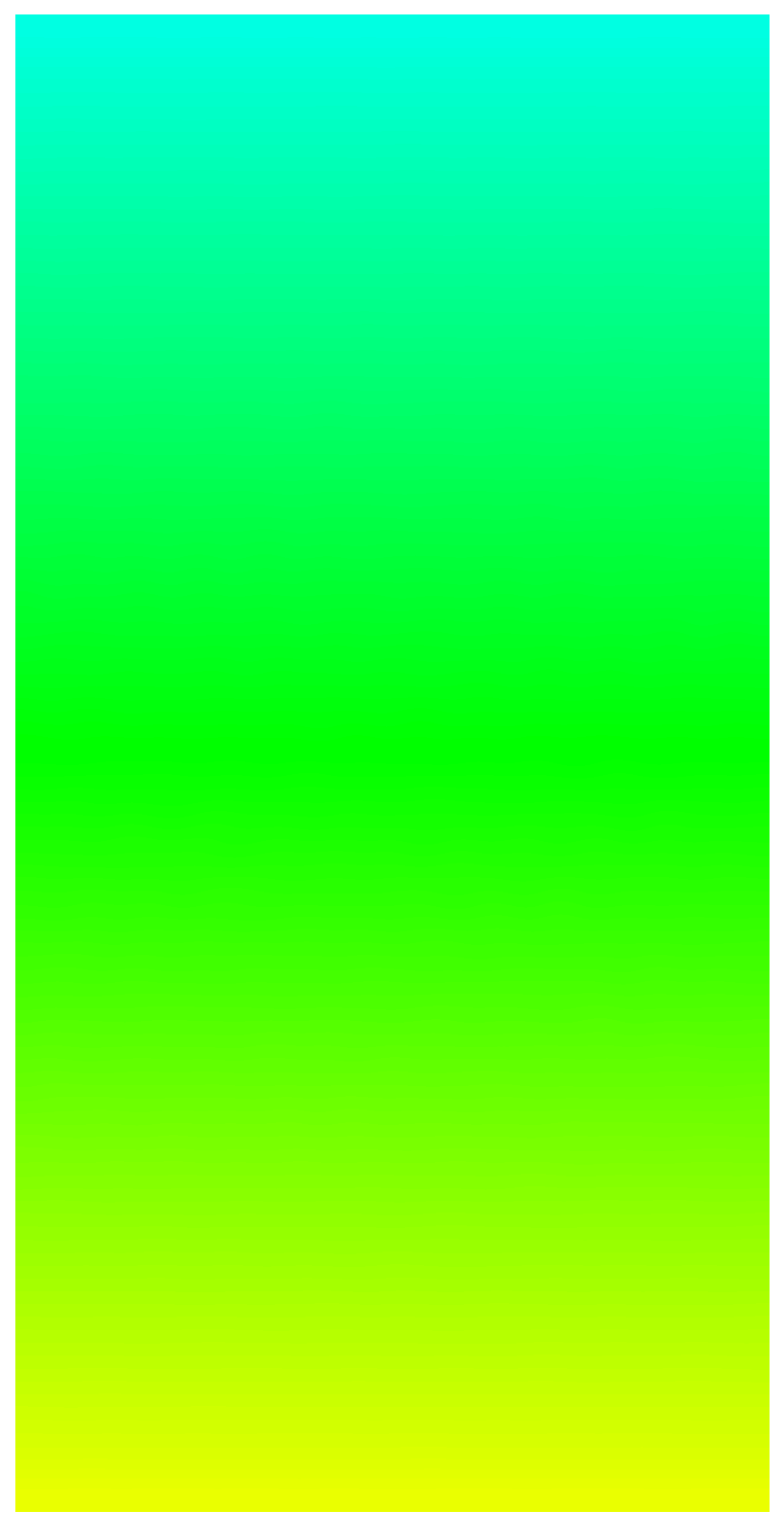}
\includegraphics[width=0.19\textwidth]{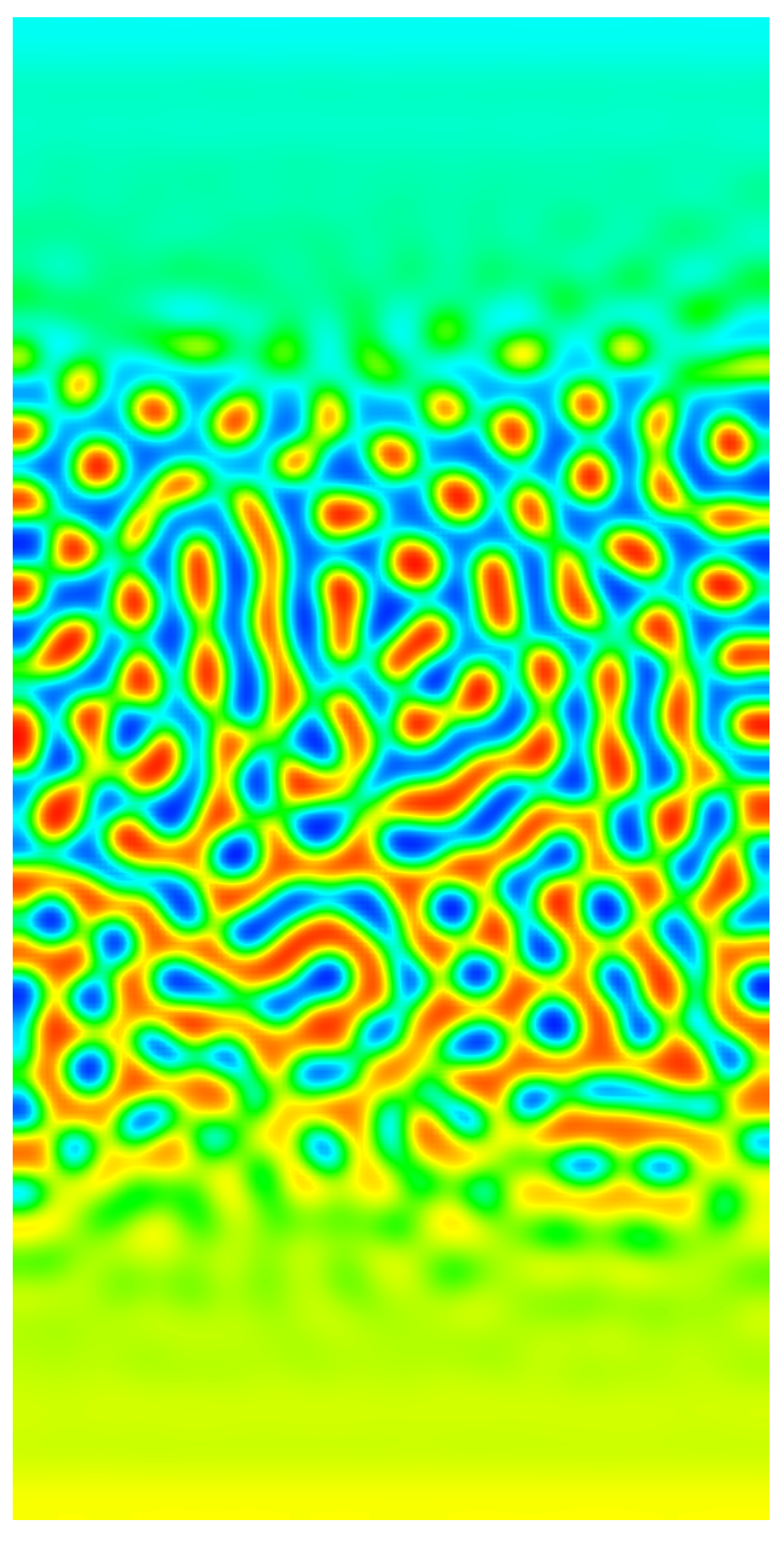}
\includegraphics[width=0.19\textwidth]{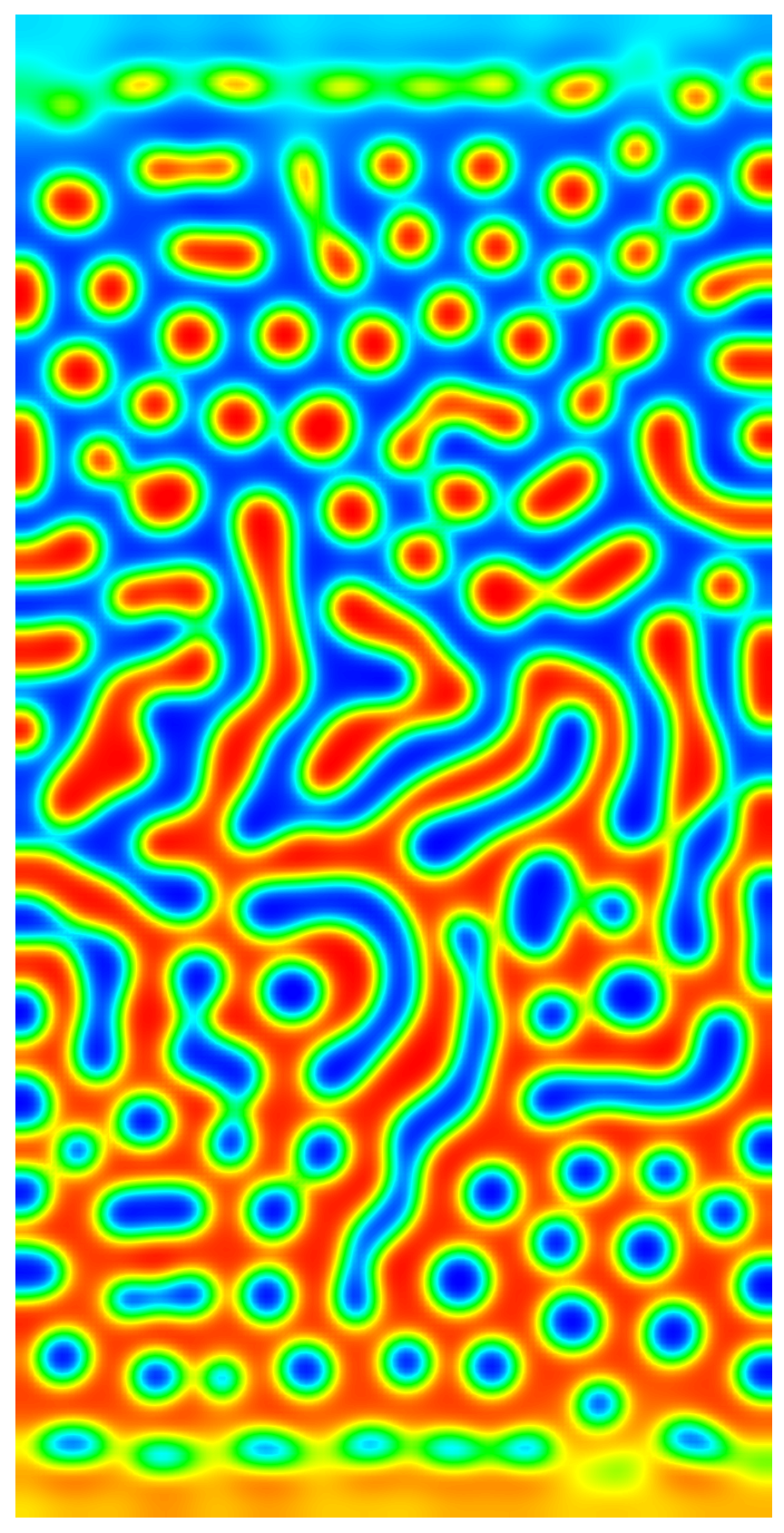}
\includegraphics[width=0.19\textwidth]{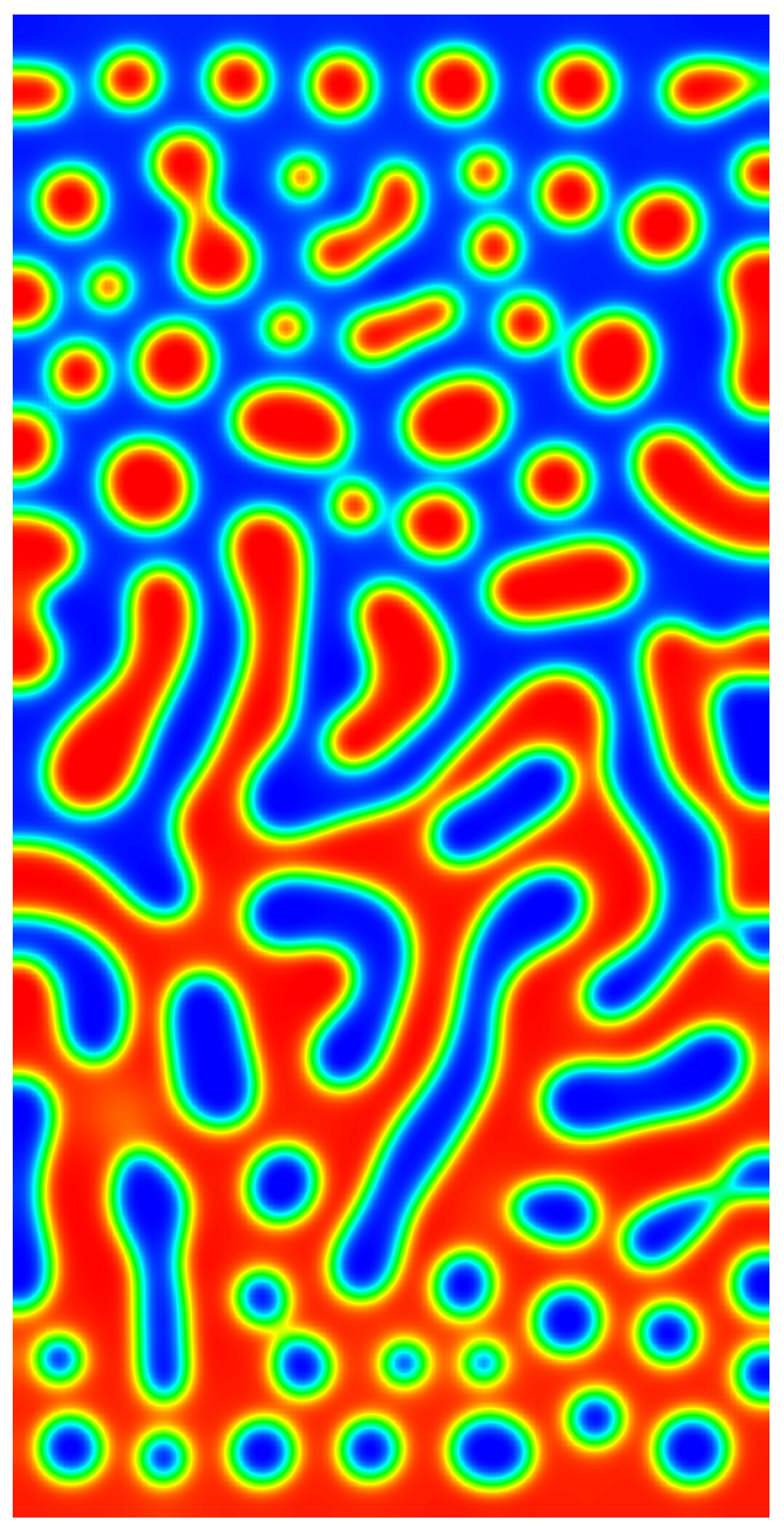}
\includegraphics[width=0.19\textwidth]{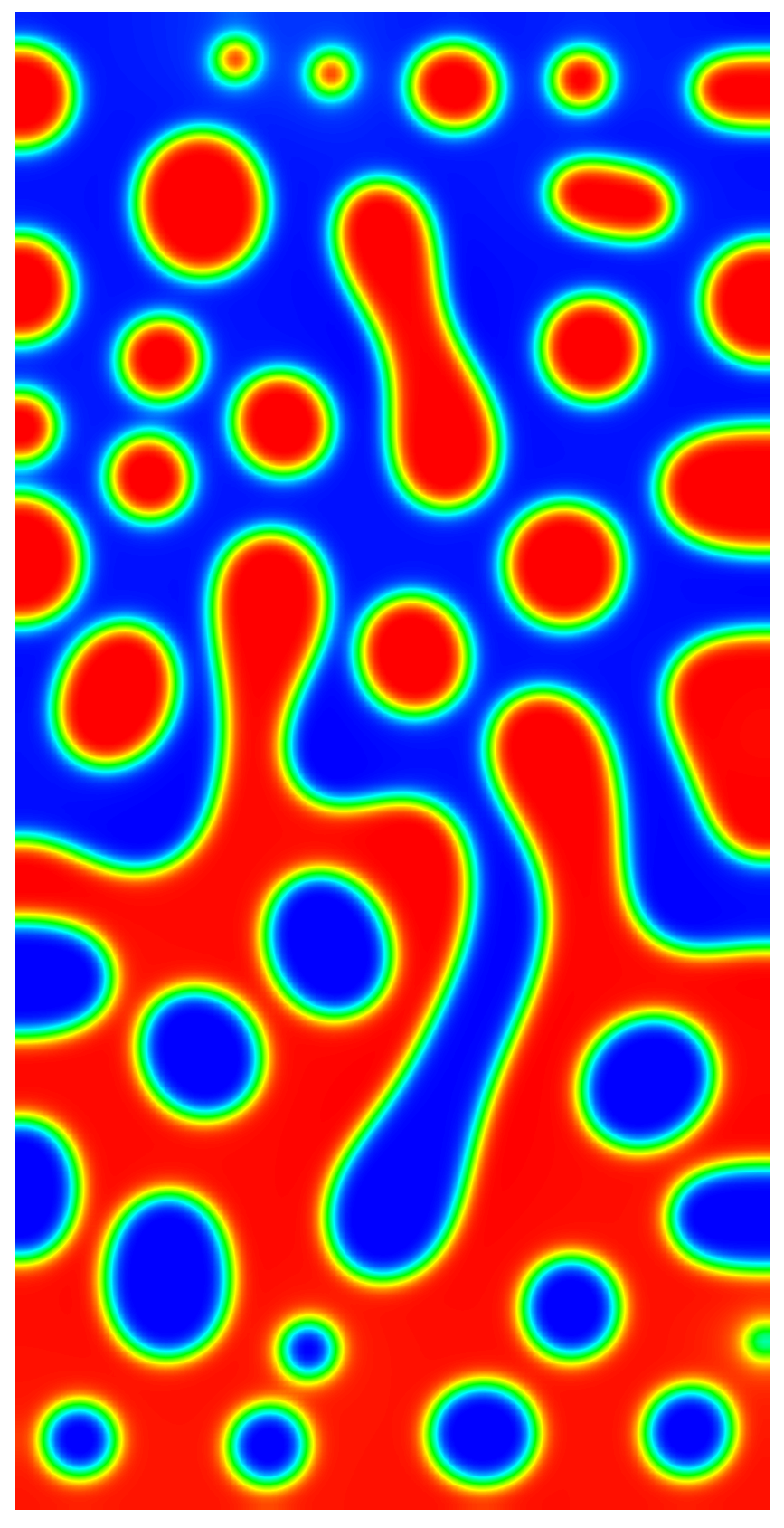}
}

\caption{Time evolution of phase variables with different surface tension. In (a) the profiles of $\phi$ at various times are shown for the case $\gamma=0.1$; (b) the profiles of $\phi$ at various times are shown for the case $\gamma=0.01$.}
\label{fig:coarsening-gam}
\end{figure}

Furthermore, the velocity fields for both cases are shown in Figure \ref{fig:coarsening-velo}. It is observed that the velocity field has a larger magnitude at the regions that changing rapidly, which means the hydrodynamics (kinetic energy) is interacting with the surface tension (Helmholtz free energy).

\begin{figure}[H]
\center
\subfigure[velocity field at $t=2, 6.5$ for Figure \ref{fig:coarsening-gam}(a)]{
\includegraphics[width=0.24\textwidth]{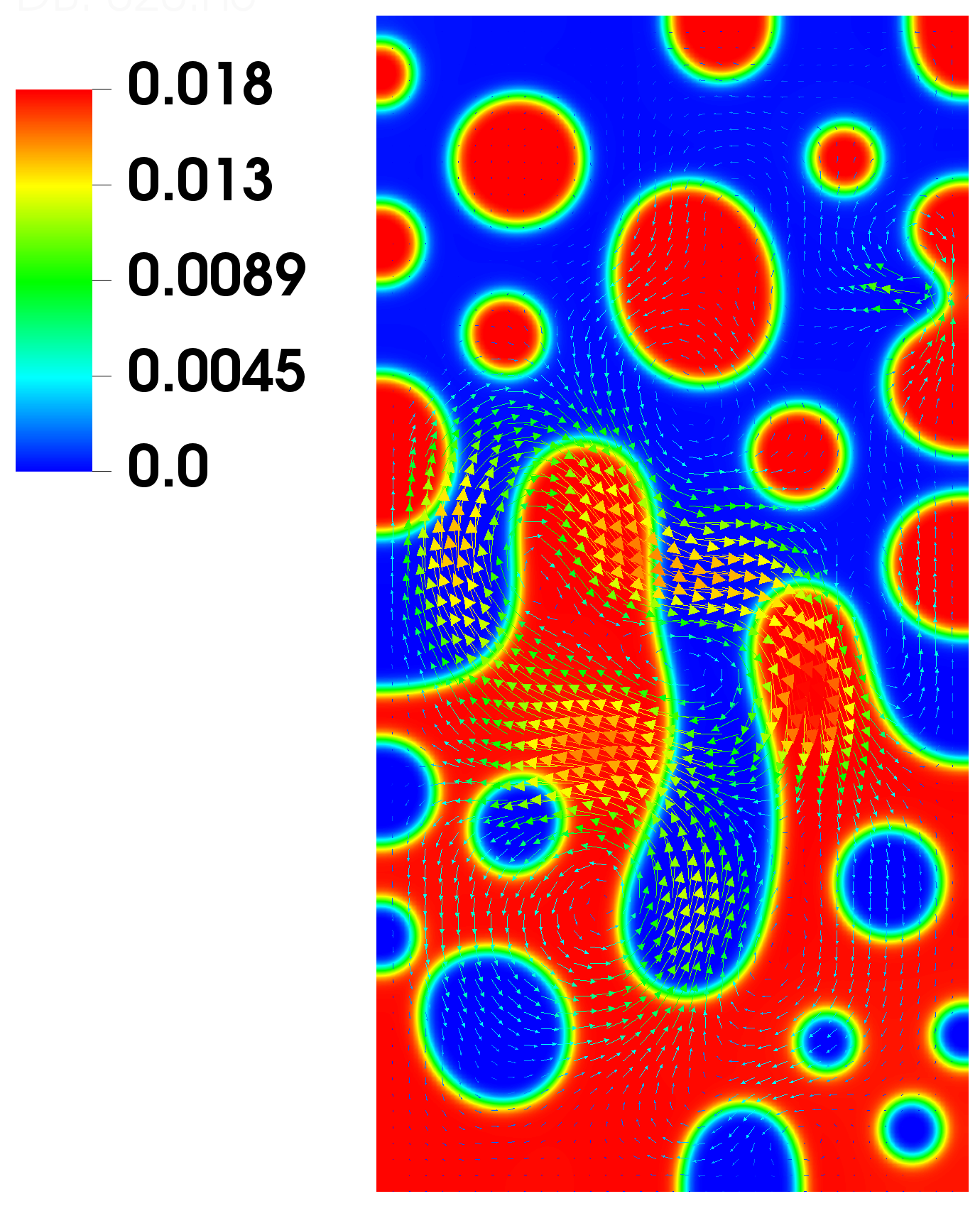}
\includegraphics[width=0.24\textwidth]{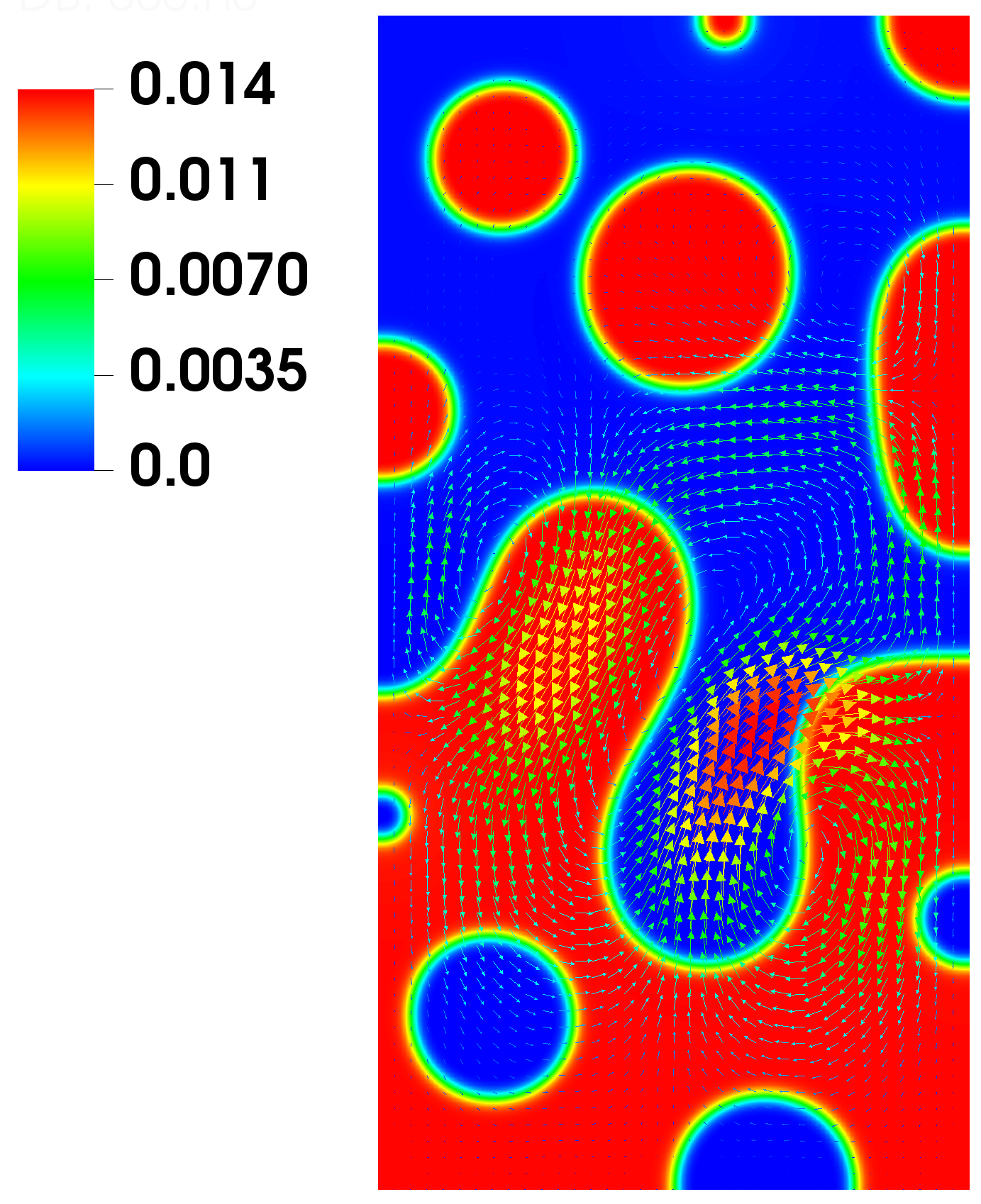}}
\subfigure[velocity field at $t=2, 6.5$ for Figure \ref{fig:coarsening-gam}(b)]{
\includegraphics[width=0.24\textwidth]{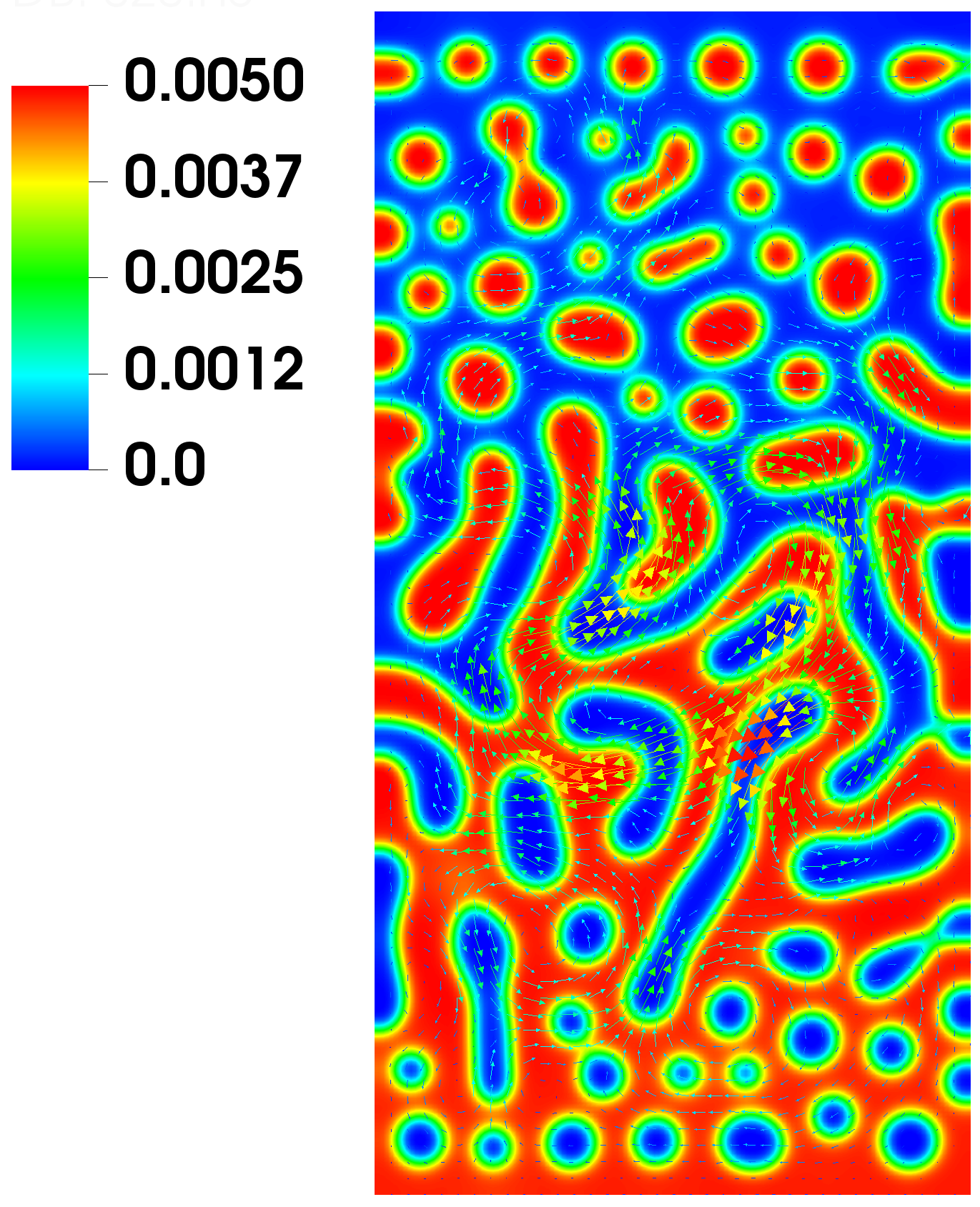}
\includegraphics[width=0.24\textwidth]{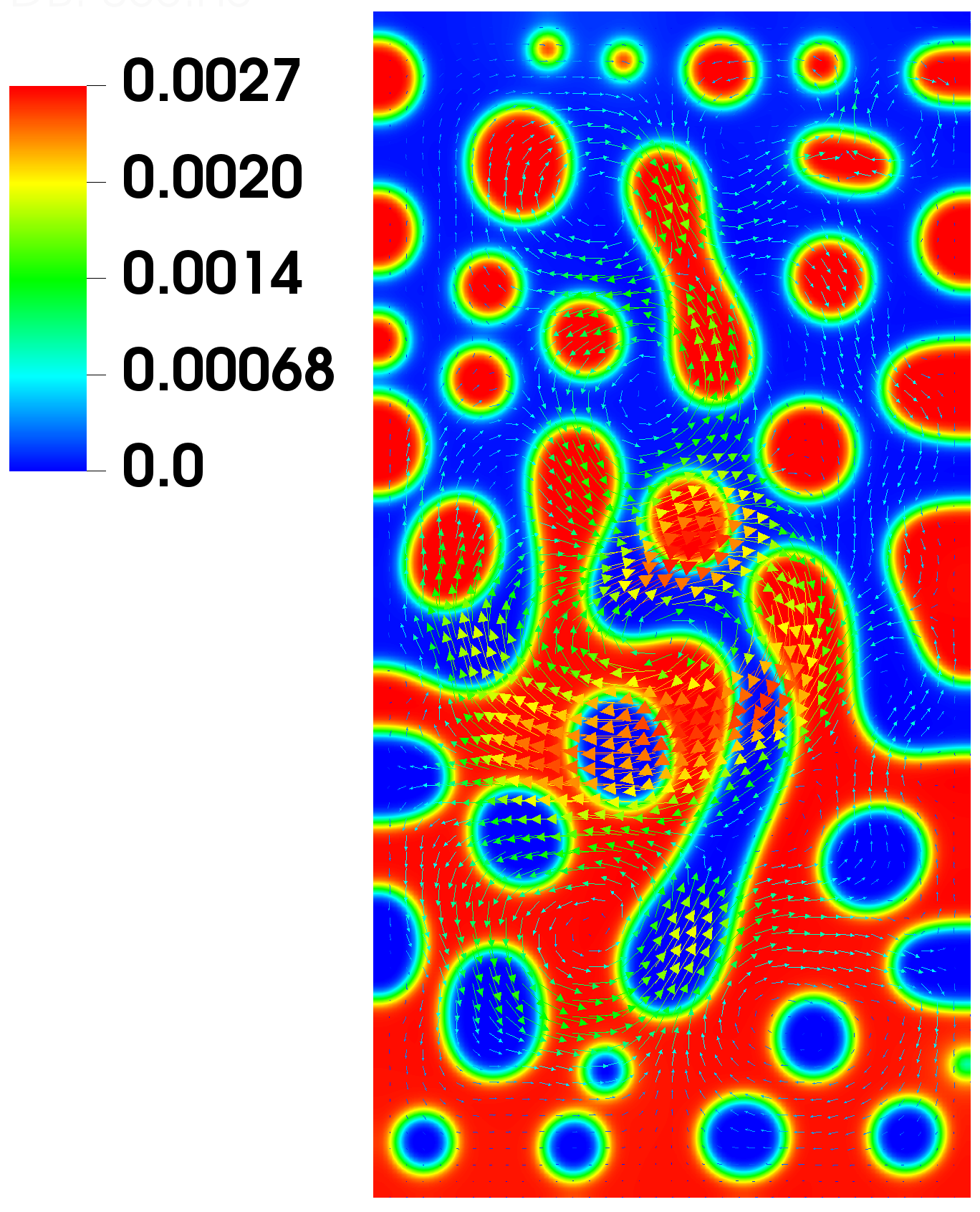}}
\caption{Velocity field for the coarsening dynamics in Figure \ref{fig:coarsening-gam}.}
\label{fig:coarsening-velo}
\end{figure}

Meanwhile, the energy evolution for both cases are summarized in Figure \ref{fig:coarsening-E}. We observe that when the surface tension is high, the coarsening changes faster. This is agreeable with the energy dissipation rate in \eqref{eq:energy-law-continous}.
\begin{figure}[H]
\center 
\subfigure[Energy evolution for Figure \ref{fig:coarsening-gam}(a)]{\includegraphics[width=0.45\textwidth]{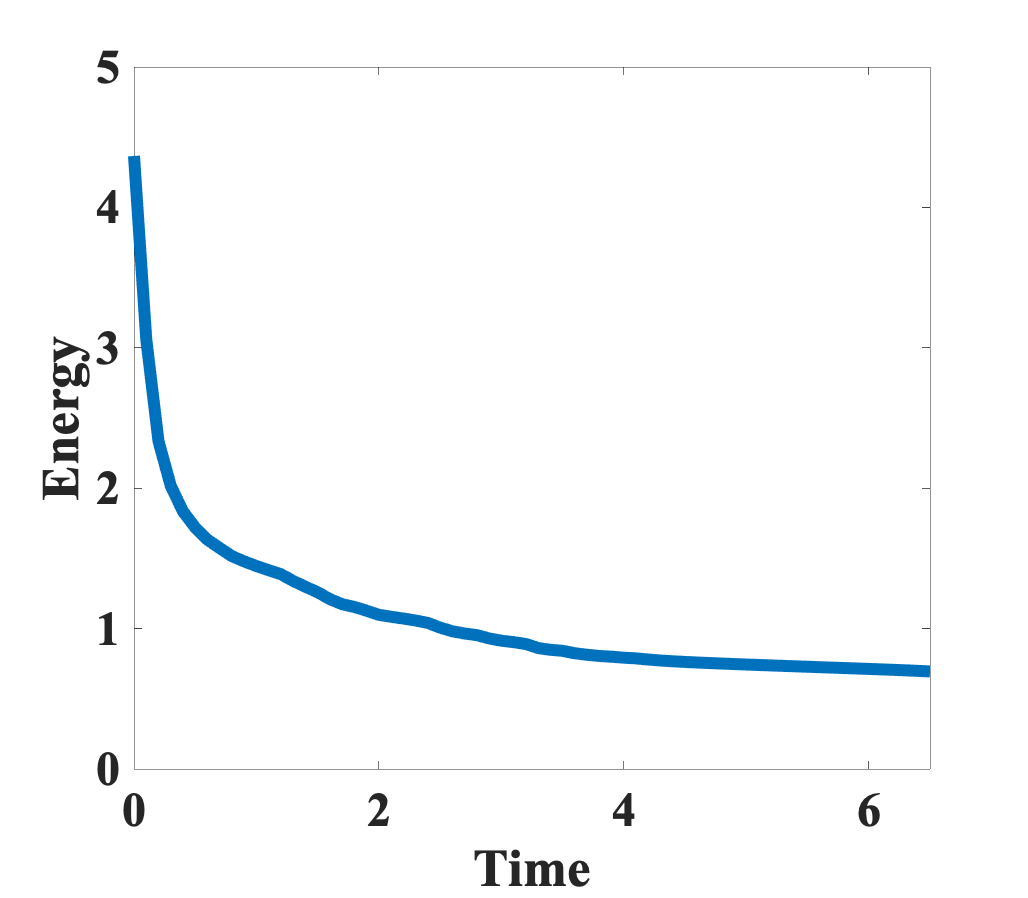}}
\subfigure[Energy evolution for Figure \ref{fig:coarsening-gam}(b)]{\includegraphics[width=0.45\textwidth]{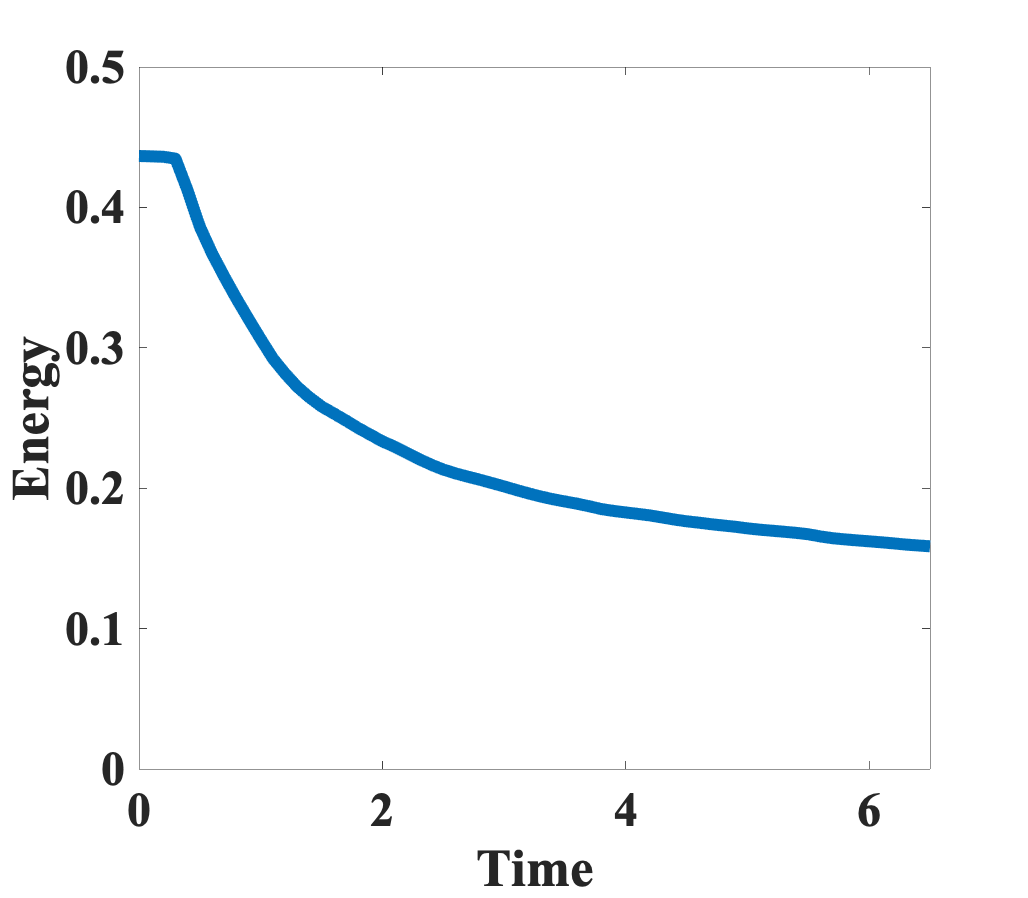}}
\caption{The time evolution of the energy for the coarsening dynamics in Figure \ref{fig:coarsening-gam}.}
\label{fig:coarsening-E}
\end{figure}

\subsection{Lid-driven Cavity}
In this case, we investigate the lid-driven cavity problem. The domain is set up as $\Omega=[0, L_x] \times [0, L_y]$ with $L_x=L_y=1$, with shear on the top. This is related with the zero Dirichlet boundary condition for the velocity at the boundary, except at $y=1$ for which we propose
$\bu|_{y=1} = (1, 0)$. We choose the initial condition for the phase variable $$\phi(x, y, t=0) = \tanh \frac{ r - \sqrt{(x-0.5L_x)^2 + (y-0.5L_y)^2}}{2\varepsilon},$$ with $r=0.15$.
The boundary condition for the phase variable $\phi$ remains the same. We use the following parameters $\rho=1$, $\eta = 1$, $\lambda =0.01$, $\varepsilon=0.01$ and we use various surface tension parameter $\gamma$. We use mesh sizes $Nx=Ny=128$. The profiles of $\phi$ at various times are summarized in Figure \ref{fig:Cavity}. We observe that our proposed numerical algorithms work well to accurately solve the Cahn-Hilliard-Navier-Stokes system, and investigate complicated two phase fluid flow.

\begin{figure}[H]
\subfigure[$\phi$ at $t=0, 1, 2,3$]{
\includegraphics[width=0.24\textwidth]{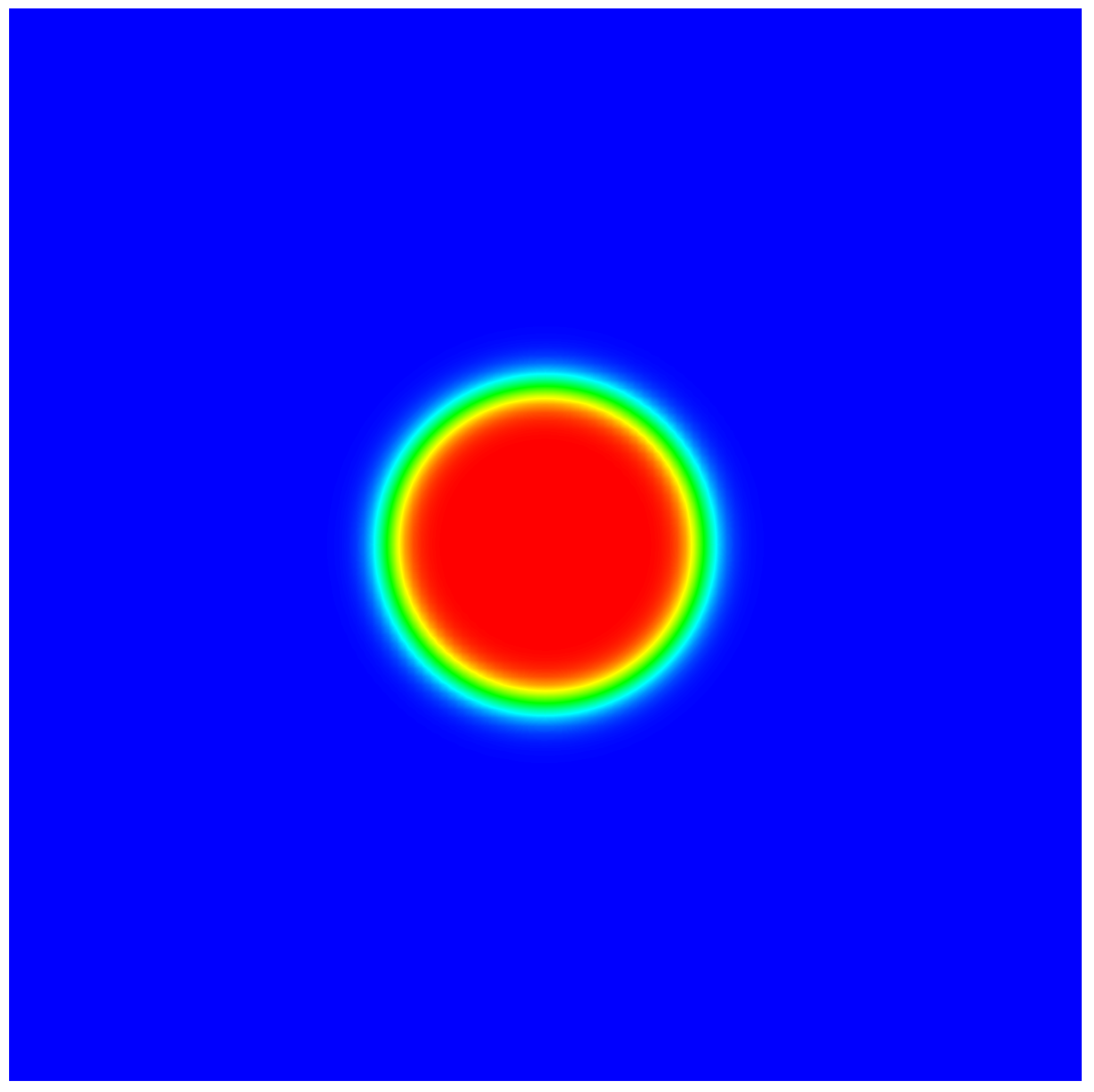}
\includegraphics[width=0.24\textwidth]{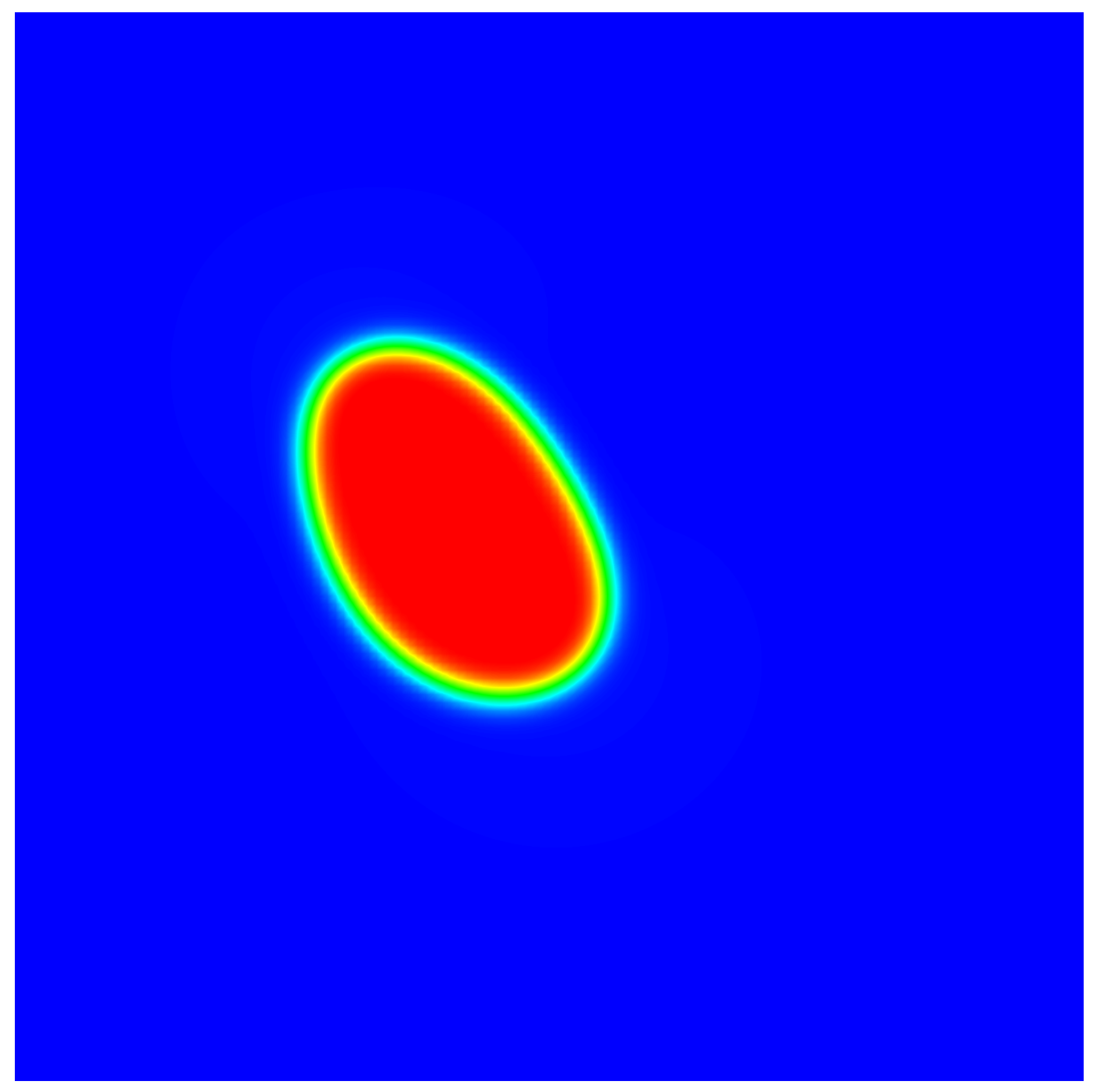}
\includegraphics[width=0.24\textwidth]{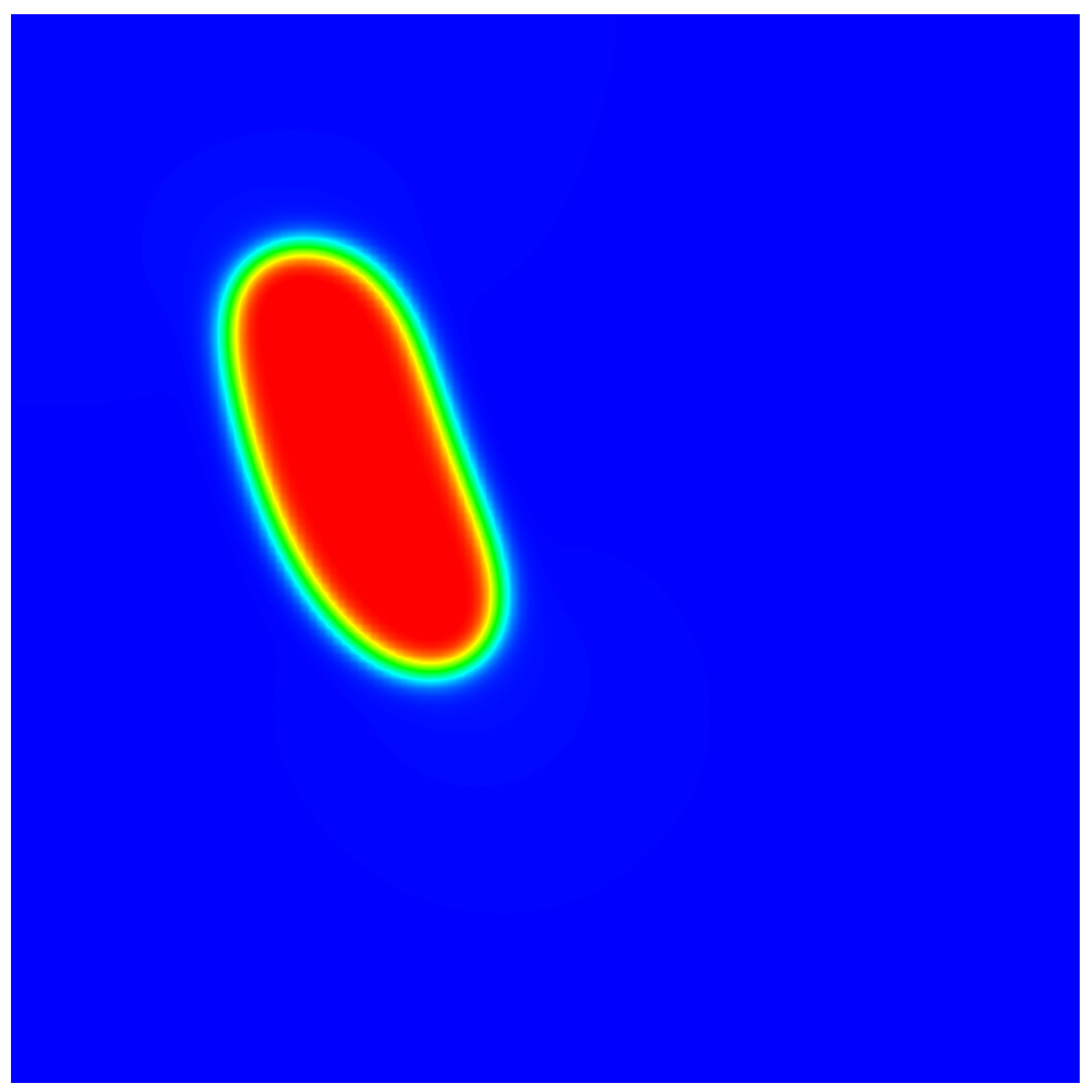}
\includegraphics[width=0.24\textwidth]{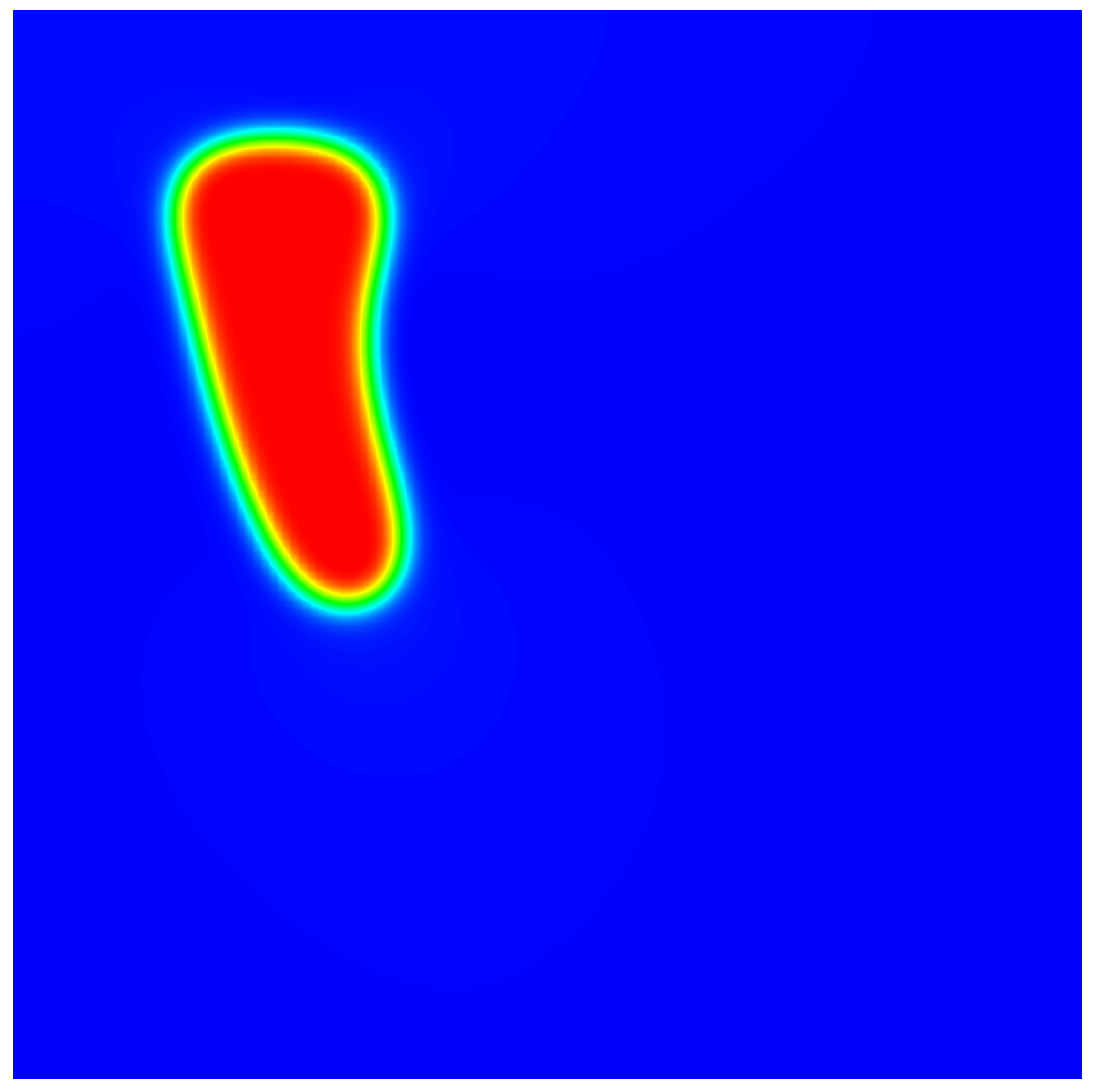}}

\subfigure[$\phi$ at $t=5,8,10,11$]{
\includegraphics[width=0.24\textwidth]{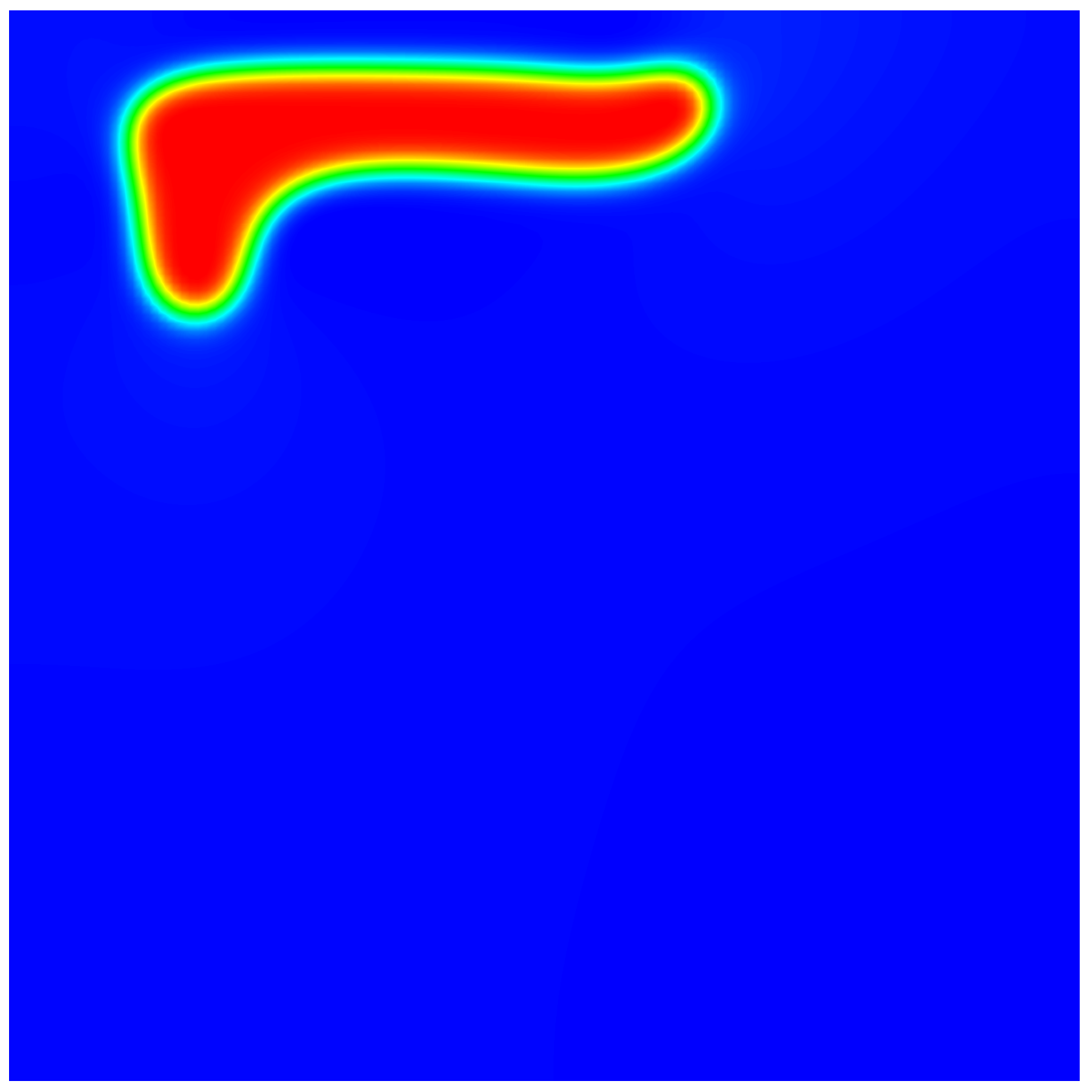}
\includegraphics[width=0.24\textwidth]{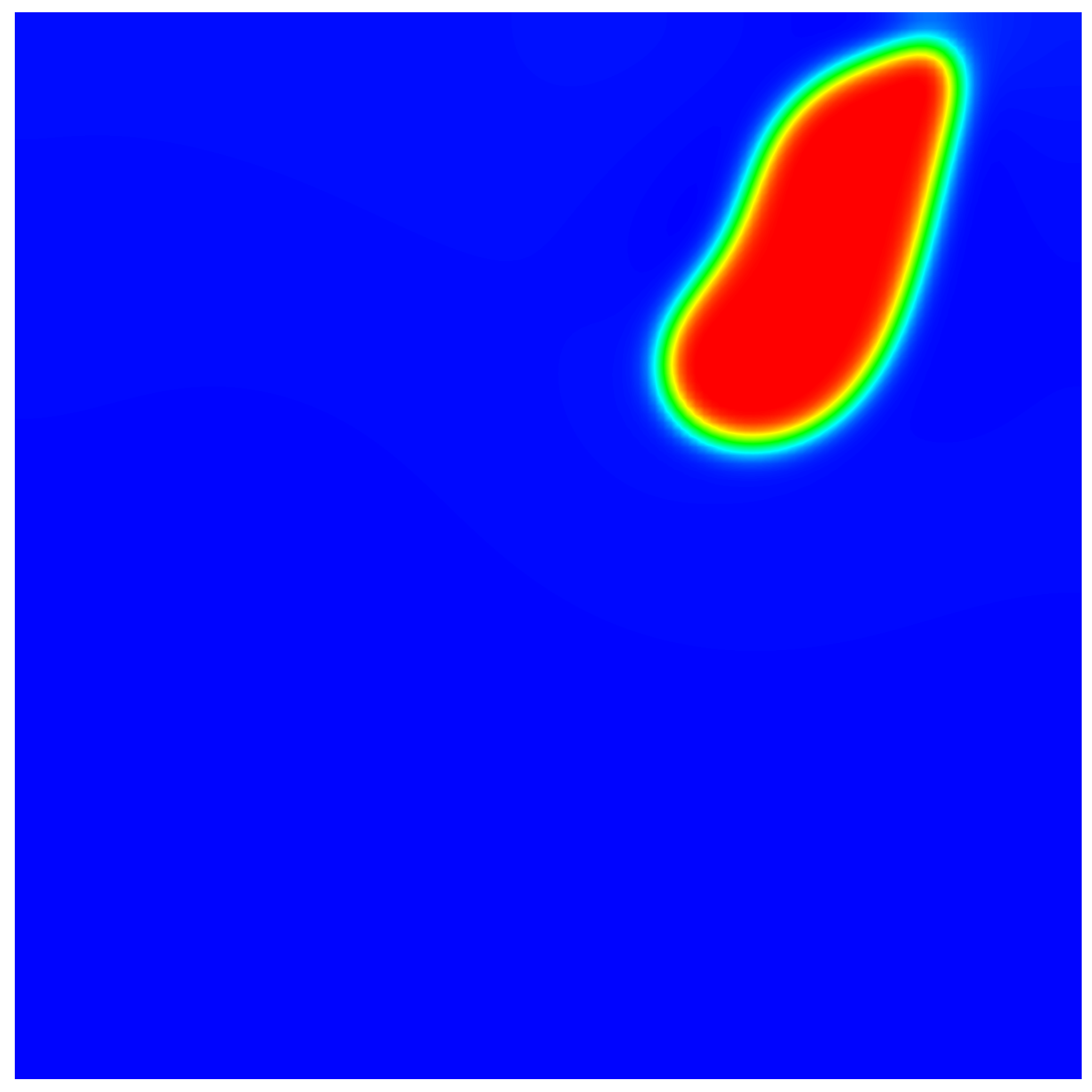}
\includegraphics[width=0.24\textwidth]{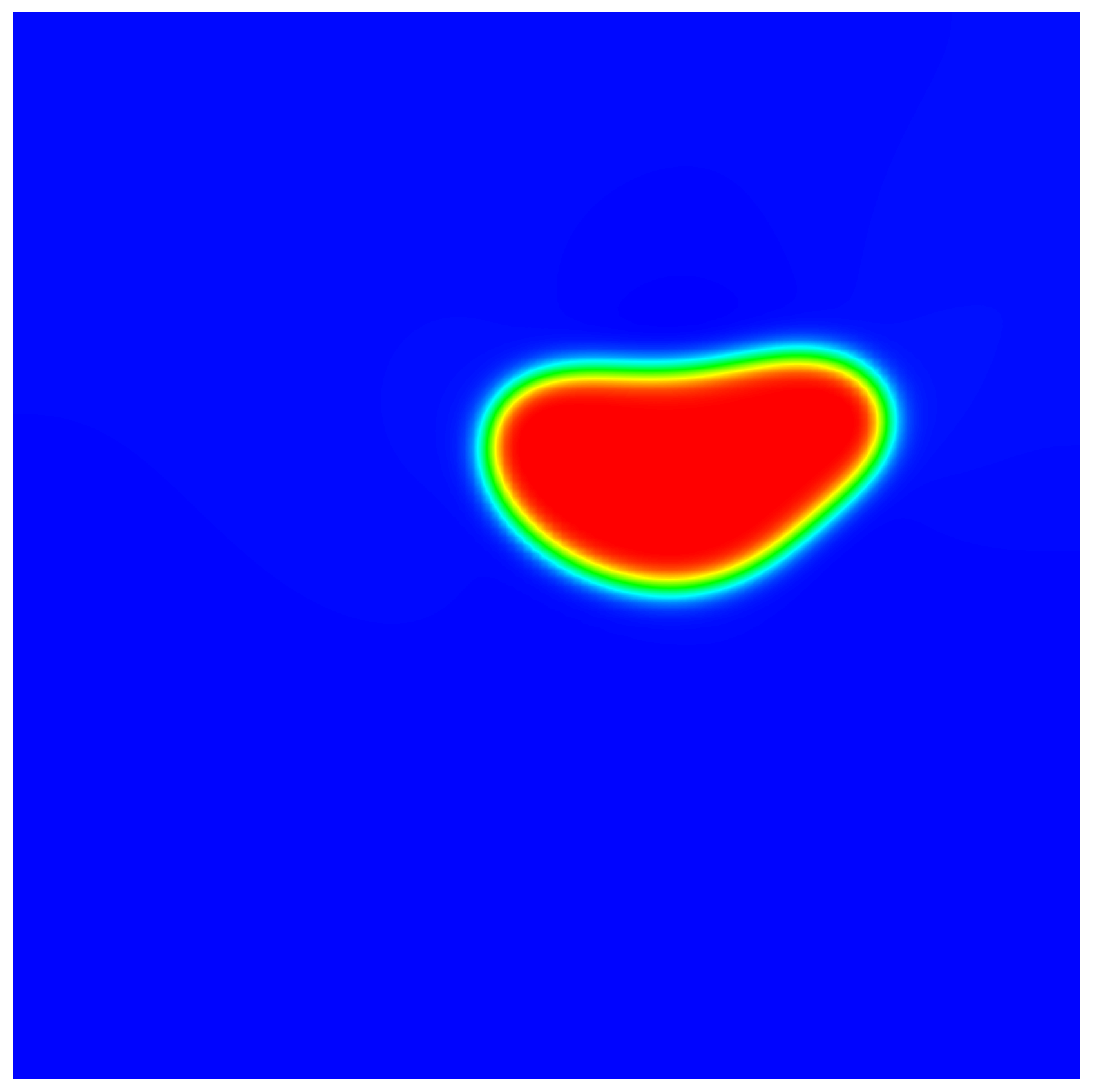}
\includegraphics[width=0.24\textwidth]{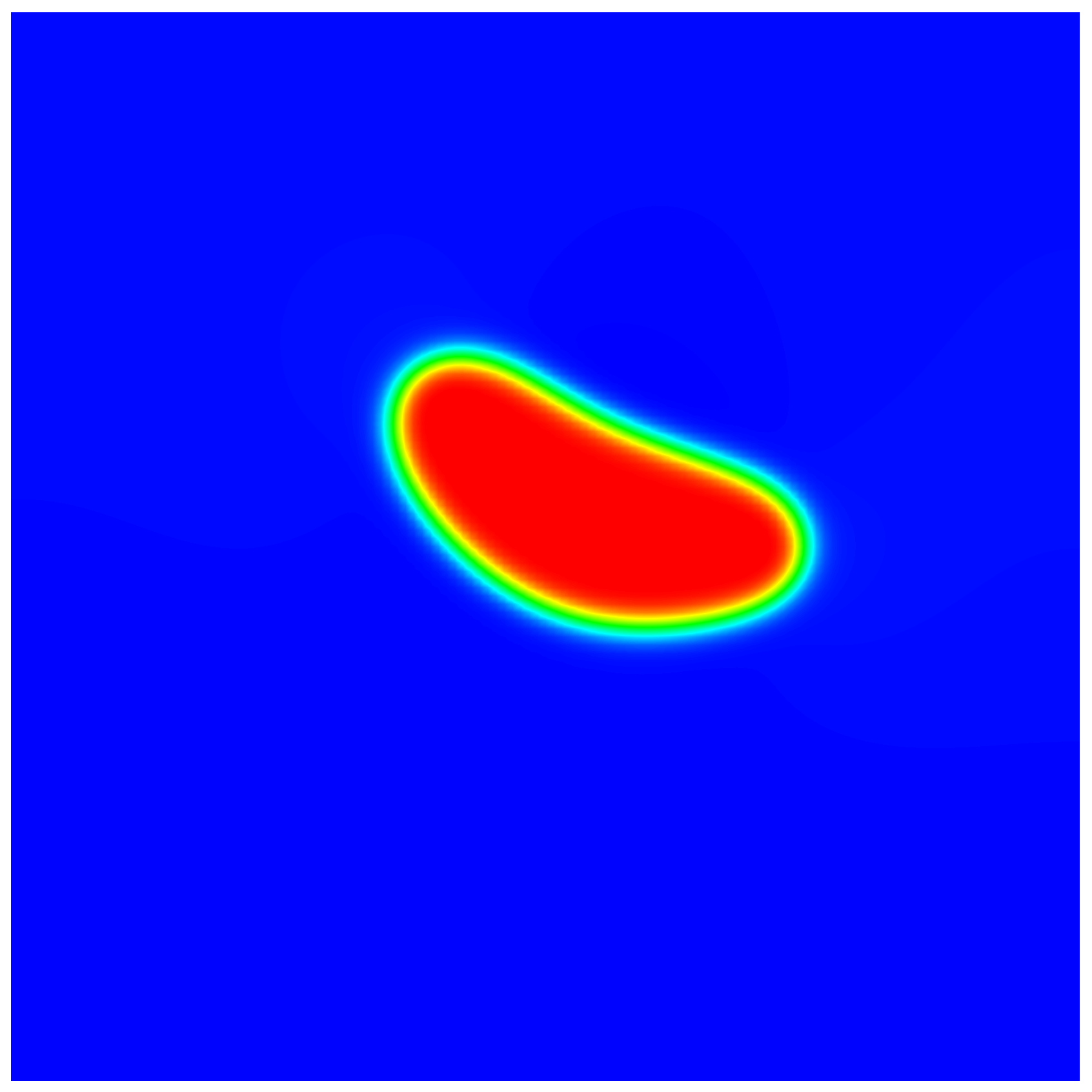}}
\caption{drop moving in  a lid-driven cavity flow. In this figure, the profiles of the phase variable $\phi$ are various times are shown. The drop is driven by the lid-driven cavity flow.}
\label{fig:Cavity}
\end{figure}

\section{Conclusion} \label{sect:conclusion}
How to design decoupled numerical algorithms for the well-known Cahn-Hilliard-Navier-Stokes (CHNS) system has been a long-standing problem. Many attempts are documented in published literature. So far, only first-order decoupled and energy stable schemes for the CHNS system is available.  Some recent attempts on designing second-order decoupled schemes using the SAV strategy also showing promising results. However, the schemes resulted from the SAV strategy only preserve a modified energy law, where the discrete laws are formulated using the auxiliary variables. Its connection with the discrete energy law using the original variables is not clear. 

In this paper, we are the first group to come up with a second-order decoupled numerical scheme for the CHNS that is energy stable. Our idea is mainly based on a reformulation of the CHNS system into a constraint gradient flow form such that the operator splitting techniques can be utilized without destroying the discrete energy laws. With this idea, we propose three variants of decoupled and second-order numerical algorithms. All of them are shown to be energy stable. They are efficient since only several simple elliptic equations shall be solved at every time step,  instead of solving a fully-coupled system. Their second-order accuracy is verified numerically. Furthermore, we also conduct several benchmark numerical simulations to justify the effectiveness of the proposed decoupled numerical schemes. 

Meanwhile, the idea introduced in this paper is not limited to decoupling the Navier-Stokes equation and the Cahn-Hilliard equation. It is widely applicable to a variety of hydrodynamics phase-field models. Its further extensions to these models and other thermodynamic-hydrodynamic models will be pursued in our later research.

\section{Acknowledgments}
Jia Zhao would like to acknowledge the support from National Science Foundation with grant NSF-DMS-1816783.
Jia Zhao would also like to acknowledge NVIDIA Corporation for the donation of a Quadro P6000 GPU for conducting some of the numerical simulations in this paper.

\bibliographystyle{unsrt}

\end{document}